\documentclass[a4paper]{amsart}

\usepackage{amsmath,amsthm,amssymb,enumerate}
\usepackage{epsfig}
\usepackage{amssymb}
\usepackage{amsmath}
\usepackage{amssymb}
\usepackage{amsmath,amsthm}
\usepackage[latin1]{inputenc}
\usepackage[T1]{fontenc}
\usepackage{ae,aecompl}
\usepackage{amsfonts}
\usepackage{amsxtra}
\usepackage{bbm,euscript,mathrsfs}
\usepackage{color}
\usepackage[left=3.3cm,right=3.3cm,top=4cm,bottom=3cm]{geometry}
\allowdisplaybreaks

\usepackage{enumitem}
\setenumerate{label={\rm (\alph{*})}}

\usepackage{amsfonts}
\usepackage{amsxtra}

\usepackage[frak,theorem_section]{paper_diening}
\numberwithin{equation}{section}

\newcommand{\Div}{\divergence}
\newcommand{\ep}{\bfvarepsilon}
\newcommand{\R}{\mathbb R}
\newcommand{\N}{\mathbb N}

\newcommand{\dd}{\mathrm{d}}
\newcommand{\dx}{\,\mathrm{d}x}
\newcommand{\dt}{\,\mathrm{d}t}
\newcommand{\dxt}{\,\mathrm{d}x\,\mathrm{d}t}
\newcommand{\ds}{\,\mathrm{d}\sigma}
\newcommand{\dxs}{\,\mathrm{d}x\,\mathrm{d}\sigma}
\newcommand{\dH}{{\, \dd \mathcal{H}^2}}

\DeclareMathOperator{\Span}{Span}

\DeclareMathOperator{\tr}{tr}

\newcommand{\dif}{\mathrm{d}}

\newcommand{\mr}{\mathbb{R}}

\newcommand{\mt}{\mathbb{R}^3}

\DeclareMathOperator{\dist}{dist}

\DeclareMathOperator{\diver}{div}

\newcommand{\bu}{\mathbf u}
\newcommand{\bq}{\mathbf q}

\DeclareMathOperator{\regkap}{\mathscr R_\kappa}
\DeclareMathOperator{\regzet}{{\mathscr R_\kappa}}
\DeclareMathOperator{\regeta}{{\mathscr R_\kappa}}

\allowdisplaybreaks

\newcommand{\toetai}{\to^\eta}
\newcommand{\weaktoetai}{\weakto^\eta}

\newcommand{\weaktoregzetan}{\weakto^\eta}
\newcommand{\toregzetan}{\to^\eta}

\newcommand{\seb}[1]{\textcolor[rgb]{0.00,0.00,1.00}{  #1}}

\begin{document}


\title{Compressible fluids interacting with\\ a linear-elastic shell}

\author{Dominic Breit}
\address[D. Breit]{
Department of Mathematics, Heriot-Watt University, Riccarton Edinburgh EH14 4AS, UK}
\email{d.breit@hw.ac.uk}

\author{Sebastian Schwarzacher}
\address[S. Schwarzacher]{Departement of Analysis, Faculty of Mathematics and Physics, Charles University, Sokolovsk\'a 83, 186 75 Praha, Czech Republic}
\email{schwarz@karlin.mff.cuni.cz}

\begin{abstract}
We study the Navier--Stokes equations governing the motion of an isentropic compressible fluid in three dimensions interacting with a flexible shell of Koiter type. The latter one
constitutes a moving part of the boundary of the physical domain. Its deformation is modeled by a linearized version of Koiter's elastic energy.
We show the existence of weak solutions to the corresponding system of PDEs provided the adiabatic exponent satisfies $\gamma>\frac{12}{7}$ ($\gamma>1$ in two dimensions). The solution exists until the moving boundary approaches a self-intersection. 
This provides a compressible counterpart of the results in [D. Lengeler, M. \Ruzicka, Weak Solutions for an Incompressible Newtonian Fluid Interacting with a Koiter Type Shell. Arch. Ration. Mech. Anal. 211 (2014), no. 1, 205--255] on incompressible Navier--Stokes equations. 
\end{abstract}

\subjclass[2010]{60H15, 35R60, 76N10,  35Q30}
\keywords{Compressible fluids, Navier--Stokes equations, weak solution, Koiter Shell, time dependent domains, moving boundary}

\date{\today}

\maketitle


\section{Introduction}
Fluid structure interactions have been studied intensively by engineers, physicists and also mathematicians. This is motivated by a plethora of applications anytime a fluid force is balanced by some flexible material; for instance in hydro- and aero-elasticity \cite{Su,Do} or biomechanics \cite{BGN}. 
In this work we consider the interaction of a baratropic compressible fluid (in particular a gas) in a three-dimensional body. A part of the boundary is assumed to be changing in time. The displacement of the boundary is prescribed via a two dimensional surface representing a Kirchhof-Love plate. Its material properties are deduced by assuming small strains and plane stresses parallel to the middle surface. 
We prove the existence of a weak solution to the coupled compressible Navier--Stokes system interacting with the Kirchhof-Love plate on a part of the boundary. The time interval of existence is only restricted as soon as self intersection of the moving boundary (namely the shell) is approached.

\subsection{Motivation \& state of art}
Over the last century mathematicians have been fascinated by the dynamics of fluid flows.
The theory of (long-time) weak solutions started with the pioneering work of Leray concerning incompressible Navier--Stokes equations \cite{Ler}. A compressible counterpart has been provided
by Lions \cite{Li2}. Lions' results have later been extended by Feireisl et al. \cite{feireisl1,fei3} to physically important situations (including, in particular, monoatomic gases). Today, there exists an
abundant amount of literature for both, incompressible as well as compressible fluids.
In the last decades fluid structure interactions have been the subject of
active research.
The interactions of fluids and elastic solids are of particular interest. A major mathematical difficulty is the parabolic-hyperbolic nature of the system resulting in regularity incompatibilities between the fluid- and the solid-phase.
First results concerning weak solutions in the incompressible case consider
regularized or damped elasticity laws, see \cite{Be,Bo,Ch,Leq}. The fluid interacts with an elastic shell which constitutes a moving part of the boundary of the physical domain in Lagrangian coordinates.  
The existence of strong solutions in short time was shown in \cite{ChSk,CoSh1,CoSh2} considering only forces of first order on the shell, hence excluding the flexural forces of the membrane. 
Long-time weak solutions in a similar setting have finally been obtained in \cite{LeRu}
 assuming a linearized elastic behavior of the shell. The authors of \cite{LeRu} consider a general three dimensional body in Eulerian coordinates. The elastic shell is a possible large part of the boundary and may deform in the direction of the outer normal. Its material behavior depends on membrane and bending forces.
A solution exists provided the magnitude of the displacement stays below some bound (depending only on the geometry of the reference domain) which excludes
self-intersections.
The results from
\cite{LeRu} have been extended to some incompressible non-Newtonian cases in \cite{Le}. See also \cite{sarka}.
Results for incompressible fluids in cylindrical domains have been shown in \cite{MuCa1,MuCa2} and \cite{BuMu}. The paper \cite{MuCa1} deals with a cylindrical linear elastic/viscoelastic Koiter shell in two dimensions (the shell is prescribed by a one-dimensional curve). The papers \cite{BuMu,MuCa2} extend this to cylindrical three-dimensional fluid flows. Note that in \cite{MuCa2} even nonlinear elastic behavior of the shell is allowed.\\ 
In contrast to the growing literature on incompressible fluids the knowledge about compressible fluids interacting with elastic solids is quite limited. Up to our knowledge, the only related result is \cite{KuTu}. Here, a compressible fluid interacts with a structure modeled by a linear wave equation in Lagrangian coordinates. The result of \cite{KuTu} concerns the existence of short-time strong solutions. It is related to earlier results about the incompressible setting, see \cite{CoSh1}.
Results on long-time weak solutions from problems coupling compressible fluids with a priori unknown elastic structures seem to be missing.
The aim of the present paper is to open this field by developing a compressible counterpart of the theory from \cite{LeRu}. More, precisely we are going to prove the existence of a weak solutions to the compressible Navier--Stokes system coupled with a linear elastic Koiter-type shell. The two dimensional shell is connected to the velocity field via boundary values on the free part of the boundary. Moreover, momentum forces
acting on the boundary are in equilibrium with the membrane forces and bending forces (flexural forces) of the shell.   

\subsection{The model}
We consider the Navier--Stokes system of an isentropic compressible viscous fluid interacting with a shell of Koiter-type of thickness $\varepsilon_0>0$. The Koiter shell model is a version of the Kirchhof-Love plate. More precisely it is a model reduction assuming small strains and plane stresses parallel to the middle
surface of the plate. Physically this means that the plate consists of a homogeneous, isotropic material.
Its mathematical formulation is as follows. Let $\Omega\subset\setR^3$ and let $T>0$. We devide $\partial\Omega$ into the fixed in time part $\Gamma$ and its compact complement $M$, the part where the shell is located. The shell is assumed to be driven solely in the direction of the outer normal $\nu$ of $\Omega$, cf. \cite{Ch,Gr}. This allows to write the energy for elastic shells via a scalar function $\eta:M\to (a,b)$ where the numbers $a,b$ are fixed but depend only on the geometry of $\Omega$, such that self intersections of the boundary are not possible. For example in case of a ball $\Omega=B_r$, the interval is $(-r,\infty)$. 
The elastic energy of the deformation is then modeled via Koiter's energy~\cite[eqs. (4.2), (8.1), (8.3)]{Koi2} 
\begin{align*}
 K(\eta)=\frac{1}{2}\varepsilon_0\int_M\bfC:\bfsigma(\eta\nu)\otimes\bfsigma(\eta\nu)\,\dd\mathcal H^2+\frac{1}{6}\varepsilon_0^3\int_M\bfC:\bftheta(\eta\nu)\otimes\bftheta(\eta\nu)\,\dd\mathcal H^2
\end{align*}
which is the sum of two terms reflecting different material properties. The first term is the membrane part of the energy which would remain even in the case of thin films. Indeed, the term $\bfsigma(\eta\nu)$ depends linearly on the pullback of the first fundamental form of the two dimensional surface $\eta(M)$.  The second term reflects the flexural part of the energy. Respectively, the argument $\bftheta$ depends linearly on the pullback of the second fundamental form (the change of curvature). The coefficient tensor $\bfC$ is a non-linear function of the first fundamental form. For more details on the derivation of this model we refer to \cite{Koi,Koi2}, where Koiter's energy for nonlinear elastic shells has been introduced. See also \cite{Ci2,Ci3} for a more recent exposition. 
Following \cite[Thm. 4.2-1 \& Thm. 4.2-2]{Ci3} one can linearize $\bfsigma$ and $\bftheta$ with respect to $\eta$ and obtain (for suitable coordinates)
$$K'(\eta)=m\Delta^2\eta+B\eta$$
for the $L^2$-gradient $K'$ of $K$. Here $m>0$ depends on the shell material (to be precise on $\varepsilon_0$ and the Lam\'e constants) and $B$ is a second order differential operator.
In particular, it is shown in \cite[Thm. 4.4-2]{Ci3}, that
\begin{align*}
K(\eta)=\frac{1}{2}\int_MK'(\eta)\,\eta\,\dH\geq c_0\int_M|\nabla^2\eta|^2\,\dH
\end{align*} 
for all $\eta\in H^{2}_0(M)$ with some $c_0>0$. The above is a Kirchhoff--Love plate equation for transverse displacements, see \cite{Ci1}. The technical restriction, that we only allow forces to act on the shell in (a fixed) normal direction is the most severe restriction in our paper. Under this assumption we will, however, show long time weak solutions. They exists as long as the plate does not approach a self intersection. Observe also, that no long time existence seems to be available for less restrictive geometric assumptions. Even for one dimensional boundaries and for incompressible  fluids no long time existence result seem to be available for less severe restrictions.
Finally observe, that since $\eta$ is assumed to have zero boundary values on $\partial M$, there is a canonical extension by zero to $\partial\Omega$, which we will use in the following without further remark.
\\


We denote by $\Omega_{\eta(t)}$ the variable in time domain. With a slight abuse of notation we denote by $I\times\Omega_\eta=\bigcup_{t\in I}\set{t}\times\Omega_{\eta(t)}$ the deformed time-space cylinder, defined via its boundary
\[
\partial\Omega_{\eta(t)}=\set{x+\eta(t,x)\nu:x\in \partial\Omega}.
\]
Recall that $\Omega$ is a given (smooth) reference domain with outer normal $\nu$. 

 Along this cylinder we observe the flow of an isentropic compressible fluid subject to the volume force $\bff:I\times \Omega_\eta\rightarrow\R^3$. We seek the density $\varrho:I\times \Omega_\eta\rightarrow\R$ and velocity field $\bu:I\times \Omega_\eta\rightarrow\R^3$ solving the following system
 \begin{align}
 \partial_t\varrho+\diver(\varrho\bu)&=0,&\text{ in } &I\times \Omega_\eta,\label{eq1}
 \\ \partial_t(\varrho\bu)+\diver(\varrho\bu\otimes\bu)&=\mu\Delta\bu+(\lambda+\mu)\nabla\diver\bfu-\nabla p(\varrho)+\seb{\varrho}\bff&\text{ in } &I\times\Omega_\eta,\label{eq2}
 \\
 \bfu(t,x+\eta(x)\nu(x))&=\partial_t\eta(t,x)\nu(x)&\text{ on } &I\times M,\label{eq3a}
 \\
 \bfu&=0&\text{ on } &I\times \Gamma,\label{eq3aa}
 \\
  \varrho(0)=\varrho_0,\quad (\varrho\bu)(0)&=\bq_0&\text{ in } &\Omega_{\eta_0}.\label{eq3}
 \end{align}
 Here $p(\varrho)$ is the pressure which is assumed to follow the $\gamma$-law,  for simplicity $p(\varrho)=a\varrho^\gamma$ where $a>0$ and $\gamma>1$. Note that
in \eqref{eq2} we suppose Newton's rheological law
$$\bfS=\bfS(\nabla\bfu)=\mu\Big(\frac{\nabla\bfu+\nabla\bfu^T}{2}-\frac{1}{3}\Div\bfu\,I\Big)+\Big(\lambda+\frac{2}{3}\mu\Big)\Div\bfu\,I$$
 with viscosity coefficients $\mu,\,\lambda$ satisfying
$$\mu>0,\quad \lambda+\frac{2}{3}\mu>0,$$
see Remark \ref{rem:bulk} for the case $\lambda+\frac{2}{3}\mu=0$.
 The shell should response optimally with respect to the forces, which act on the boundary. Therefore we have
 \begin{align}\label{eq:4}
  \varepsilon_0\varrho_S\partial_t^2\eta+K'(\eta)&=
  g+\nu \cdot \bfF\quad\text{on}\quad I\times M,
\end{align}
where $\varrho_S>0$ is the density of the shell.
Here $g:[0,T]\times M\rightarrow\R$ is a given force and $\bfF$ is given by
\begin{align*}
\bfF&:=\big(-\bftau\nu_\eta \big)\circ \bfPsi_{\eta(t)}|\det D\bfPsi_{\eta(t)}|
\\
\bftau&:=-\mu\nabla\bfu -(\lambda+\mu)\Div\bfu\, \mathcal{I} +p(\varrho)\mathcal{I}. 
 \end{align*} 
Here $\bfPsi_{\eta(t)}:\partial\Omega\rightarrow\partial\Omega_{\eta(t)}$ is a change of coordinates and $\bftau=\bfS-p I$ is the Cauchy stress. To simplify the presentation in \eqref{eq:4}
we will assume 
\[
\varepsilon_0\varrho_S=1
\]
 throughout the paper. 
We assume the following boundary and initial values for $\eta$
\begin{align}
\label{initiala}
\eta(0,\cdot)=\eta_0,\quad \partial_t\eta(0,\cdot)=\eta_1\quad\text{on}\quad M,\\
 \label{initial}\eta=0,\quad\nabla \eta=0\quad\text{on}\quad\partial M,
\end{align}
where $\eta_0,\eta_1:M\rightarrow\R$ are given functions. 
Where we assume that
\[
\text{Im}(\eta_0)\subset (a,b).
\]
In view of \eqref{eq3a} we have to assume the compatibility condition
\begin{align}\label{eq:compa}
\eta_1(x)\nu(x)=\frac{\bfq_0}{\rho_0}(x+\eta(x)\nu(x))\quad\text{on}\quad M. 
\end{align}
By the canonical extension of $\eta,\partial\eta$ by $0$ to $\partial\Omega$ we can unify \eqref{eq3a} and \eqref{eq3aa} to
\begin{align}
\label{boundary:unified}
 \bfu(t,x+\eta(t,x)\nu(x))=\partial_t\eta(t,x)\nu\quad\text{on}\quad I\times \partial\Omega.
\end{align}
Our main result is tho following existence theorem. The system \eqref{eq1}--\eqref{eq:compa} can be written in a natural way as a weak solution. The precise formulation can be found in Section \ref{sec:6}, cf. \eqref{eq:apvarrho0final} and \eqref{eq:apufinal}.
 The main result of this paper is the following.
\begin{theorem} 
For regular data and $\gamma\in \big(\frac{12}7,\infty)$
there exists a weak solution $(\eta,\varrho, \bfu)$ to \eqref{eq1}--\eqref{initial} satisfying the energy estimate
\begin{align}\label{eq:apriori0}
\begin{aligned}
&\sup_{I}\int_{\Omega_{\eta}}\varrho|\bfu|^2\dx
+\sup_{I}\int_{\Omega_\eta}\varrho^{\gamma}\dx+\int_I\int_{\Omega_{\eta}}|\nabla\bfu|^2\dxs\\&
+\sup_{I}\int_M \frac{|\partial_t\eta|^2}{2}\,\dd\mathcal H^2+ \sup_{I}\int_M|\nabla^2\eta|^2\,\dd\mathcal H^2\\
&\leq\,c(\bff,g,\bfq_0,\varrho_0,\eta_0,\eta_1).
\end{aligned}
\end{align} 
The intervall of existence is of form $I=[0,t)$, where $t<T$ only in case $\Omega_{\eta(s)}$ approaches a self-intersection with $s\to t$.   
\end{theorem}
The function space of existence for a weak solution to \eqref{eq1}--\eqref{initial} is determined by the left-hand side of
\eqref{eq:apriori0} taking into account the variable domain.
For the precise assumptions on the given data, as well as the precise definition of a weak solution, see Theorem \ref{thm:final} in the last section of the paper.
{We remark that in the three dimensional case the bound $\gamma>\frac{12}{7}$ is less restrictive then the bound
$\gamma>\frac{9}{5}$ appearing in the pioneering work of Lions \cite{Li2} 
but more restrictive than the bound $\gamma>\frac{3}{2}$ arising in the theory by Feireisl et al. \cite{feireisl1}. A detailed explanation can be found in the next subsection.} For more explanations on the restrictions of the growth condition on the pressure see Remark~\ref{rem:exp} at the end of this section.

\subsection{Mathematical significance \& novelties}
A main task is to understand how to pass to the limit
in a sequence of solutions $(\eta_n,\bfu_n,\varrho_n)$ to \eqref{eq1}--\eqref{initial} which enjoys suitable regularity properties and satisfies the uniform estimate \eqref{eq:apriori0}.
The passage to the limit in the convective terms $\varrho_n\bfu_n$ and $\varrho_n\bfu_n\otimes\bfu_n$ follows by local arguments combined with global integrability, see \eqref{conv:rhov2} and \eqref{conv:rhovv2}. Problems with the moving boundary can be avoided.
Note that this is totally different to the incompressible system studied in \cite{LeRu} where huge difficulties arise due to the divergence-free constraint.
 As common for the compressible Navier--Stokes system the major difficulty is to pass to the limit in the nonlinear pressure. A key step is to improve the (space)-integrability of the density to ensure that $p(\varrho_n)$ actually converges to a measurable function (and not just to a measure). {Locally, where the effect of the moving boundary disappears, this can be done by the standard method, see Proposition \ref{prop:higher}. However, our test-functions in the weak formulation are not compactly supported. This is crucial for the coupling of fluid and shell.
Note in particular, that this is different from \cite{Fe}, where the interaction of compressible fluids and rigid bodies is studied. In \cite{Fe}, at least the gradients of test-functions are supported away from the area of interaction.
In our case, however, it is essential to exclude the concentration of $p(\varrho_n)$ at the boundary. On account of the limited regularity of the moving boundary (it is not even Lipschitz in three dimensions, see \eqref{eq:apriori0}) the common approach based on the \Bogovskii\, operator fails. We solve this problem inspired by a method introduced in \cite{Kuk}
for compressible Navier--Stokes equations in irregular domains, see Proposition \ref{prop:higherb} and \ref{prop:higherb'}. 
It can be used to exclude the concentration of the pressure at the boundary. This, in turn, allows to prove that the weak continuity of the effective viscous flux $$p(\varrho)-(\lambda+2\mu)\Div\bfu$$ 
holds globally, see \eqref{eq:fluxpsi} and \eqref{eq:fluxpsi'}. 
In order to combine this with the renormalized
continuity equation we are confronted with another problem: we do not have a zero boundary conditions for the velocity at the shell.
In general, it seems to be extremely difficult if not impossible to combine the properties of the effective viscous flux with the renormalized continuity equation in this case (see the remarks in \cite[Section 7.12.5]{novot}). This is due to the additional boundary term which appears when extending the continuity equation to the whole space.
However, due to the natural interplay
between fluid flow and elastic shell, our situation can be understood as no-slip boundary conditions with respect to the moving shell. Hence, the just mentioned boundary term disappears due to the Lagrangian background of the material derivative. To make this observation accessible, a careful study of the damped continuity equation in time dependent domains is necessary. We refer to Subsection~\ref{ssec:damp} and in particular Theorem~\ref{lem:warme}, which collects the necessary regularity results for the density function on time changing domains. It implies for instance the respective renormalized formulation. 
This is the second essential tool which allows to show strong convergence of an approximate sequence $\varrho_n$ and hence to establish the correct form of the pressure in the limit equation.
\\
The third difficulty is to construct a sequence of solutions. In the present case, this is rather difficult to do, since the geometry and the solution are highly coupled via the partial differential equations. Hence, in order to use the ideas explained above rigorously, we need a four layer approximation of the system as follows.
\begin{itemize}
\item Artificial pressure ($\delta$-layer): replace $p(\varrho)=a\varrho^\gamma$ by
$p_\delta(\varrho)=a\varrho^\gamma+\delta\varrho^\beta$ where $\beta$ is chosen large enough.
\item Artificial viscosity ($\varepsilon$-layer): add $\varepsilon\Delta\varrho$ to the right-hand side of \eqref{eq1}.
\item Regularization of the boundary ($\kappa$-layer): Replace the underlying domain $\Omega_\eta$ by $\Omega_{\eta_\kappa}$ where ${\eta_\kappa}$ is a suitable regularization of $\eta$. Accordingly, the convective terms and the pressure have to be regularized as well.
\item Finite-dimensional approximation ($N$-layer): the momentum equation has to be solved by means of a Galerkin-approximation.
\end{itemize} 
The first two layers are common in the theory of compressible Navier--Stokes equations, see \cite{feireisl1}. The third layer is needed additionally due to the low regularity of the shell described by $\eta$. By \eqref{eq:apriori0} we have $\eta\in W^{2,2}(M)$
Such that Sobolev's embedding implies $\eta\in W^{1,q}(M)$ for all $q<\infty$ but not necessarily $\eta\in W^{1,\infty}(M)$. So, we do not have a Lipschitz boundary.
In addition, it is necessary to regularize the convective terms in \eqref{eq1} and \eqref{eq2} (see the comments on the $N$-layer below for a detailed explanation).\\
On the last layer we are confronted with the problem that the function space depends on the solution itself. As a consequence a finite-dimensional Galerkin approximation is not possible: the Ansatz functions depend on the solution itself. Motivated by \cite{LeRu} we apply a fixpoint argument in $\eta$ and $\bfu$ for a linearized problem (roughly speaking we replace $\varrho\bfu\otimes\bfu$ in \eqref{eq2} by $\varrho\bfu\otimes\bfv$ and $\varrho\bfu$ by $\varrho\bfv$ in \eqref{eq1} for $\bfv$ given, see Subsection \ref{subsec:dec}). It is crucial for our fixed point argument that the momentum equation is linear in $\bfu$. For $(\zeta,\bfv)$ given we solve the system on the domain $\Omega_\zeta$. The domain still varies in time but is independent of the solution. Note here, that $\varrho$ is computed by solving the continuity equation with convective term independent of $\bfu$. The
existence of a weak solution $(\eta,\bfu)$ to the decoupled system can than be shown by Galerkin approximation without further problems. This is due to the good a-priori information that was derived for $\varrho$ in Theorem~\ref{lem:warme}, see Theorem \ref{thm:decu}. The next difficulty is to find a fixed point of the mapping $(\zeta,\bfv)\mapsto(\eta,\bfu)$ in an appropriate function space. 
The compactness of the mapping situated on the shell is rather easys as we apply a proper regularization with arbitrary smoothness. The main issue is the compactness of the velocity. Inspired by ideas from \cite{LeRu} we can prove compactness of $\bfu_n$ in $L^2(I\times\R^3)$ (where $\bfu_n$ is extended by zero). 
 It is based on Lemma \ref{thm:weakstrong}, where we prove a variant of the Aubin-Lions compactness theorem for variable domains.
It is noteworthy, that we are unable to exclude vacuum even in the situation of a damped continuity equation. To prevent the problem with the vacuum we replace on the $\kappa$-level the momentum $\partial_t(\varrho\bfu)$ by $\partial_t((\varrho+\kappa)\bfu)$ in the momentum equation, which can be used to show that $\bfu_n$ is strongly compact in $L^2$.

\subsection{Outline of the paper} In section \ref{sec:2} we present basic concerning variable domains as well the functional analytic set-up. In Section \ref{subsec:continuity} we we study the continuity equation (with artificial viscosity) on variable domains. The renormalized formulation is of particular importance.
 Section \ref{sec:3} is concerned with the decoupled system, its finite dimensional approximation and the fixed point argument. The main result of this section
is the existence of a weak solution to the regularized system
with artificial viscosity and pressure. In section \ref{sec:4} we pass to the limit in the regularization (of domain and convective terms) and gain a weak solution to the system with artificial viscosity. Compactness of the density can be shown as in the fixed point argument. Hence we can pass to the limit in all nonlinearities without further difficulties. The proceeding sections \ref{sec:5} and \ref{sec:6} deal with the vanishing artificial viscosity
and vanishing artificial pressure limit respectively. Both follow a similar scheme where the major difficulty is the strong convergence of the density. The argumentation is based on the weak continuity of the effective viscous flux, oscillation defect measures and the renormalized continuity equation. The main result of this paper (the existence of weak solutions to \eqref{eq1}--\eqref{initial})
follows after passing to the limit with $\delta\rightarrow0$ in Section \ref{sec:6}. The full statement is given in Theorem \ref{thm:final}.
\begin{remark}
\label{rem:exp}
The restriction the bound $\gamma>\frac{12}{7}$ in three space dimensions is needed 
exclude concentration of the pressure near the moving boundary. Indeed, such concentrations are excluded by constructing a test-function $\bfphi_n$ whose divergence explodes at the boundary while $\int p(\varrho_n)\,\Div\bfphi_n$ is still bounded. This requires, in particular, to estimate the integral
$\int_I\int_{\Omega_{\eta_n}} \varrho_n\bfu_n\,\partial_t\bfphi_n\dxt.$
Naturally, the function $\bfphi_n$ depends on the distance to the boundary and as such on the shape of the moving boundary which only has low regularity. Indeed, the given a priori estimates imply that $\partial_t\bfphi_n$ can only be bounded in $L^2(I;L^q)$ for all $q<4$ (using \eqref{eq3a}). 
Hence we need to know that $\varrho_n\bfu_n$ is bounded in $L^2(I;L^p)$ uniformly in $n$ for some $p>\frac{4}{3}$. This follows from the a priori estimates provided we have $\gamma>\frac{12}{7}$ (using that $\varrho_n\bfu_n\in L^2(I;L^{\frac{6\gamma}{\gamma+6}})$ in three dimensions).\\
In the two dimensional case we have instead $\partial_t\bfphi_n\in L^\infty(I;L^q)$ for all $q<\infty$. Consequently, no additional restrictions on $\gamma$ is needed and the result holds for all $\gamma>1$.
\end{remark}
\begin{remark}
\label{rem:bulk}
Our proof requires the bulk viscosity $\lambda+\frac{2}{3}\mu$ to be strictly positive. In case $\lambda+\frac{2}{3}\mu=0$ it is necessary
to control the full gradient by the deviatoric part of the symmetric gradient. Such a Korn-type inequality is well-known for Lipschitz domains, see \cite{resh}.
In our context of domains with less regularity, a Korn-type inequality for symmetric gradients is shown in
\cite[Prop. 2.9]{Le} following ideas of \cite{Ac}. The integrability of the full gradient is, however, less than the one of the symmetric gradient.
We believe that a corresponding trace-free version can be shown following similar ideas. So, the case $\lambda+\frac{2}{3}\mu=0$ could be included for the price that the velocity only belongs to $W^{1,p}$ for all $p<2$.
\end{remark}

\section{Preliminaries}
\label{sec:2} 
The variable domain $\Omega_\eta$ can be parametrized in terms of the reference domain $\Omega$ via a mapping $\bfPsi_\eta$ such that
\begin{align}
\label{map2}
\begin{aligned}
\bfPsi_\eta&: \Omega\to \Omega_\eta\text{ is invertible }
\\
\text{and  }\bfPsi_\eta|_{\partial\Omega}&:\partial\Omega\to   \partial \Omega_\eta\text{ is invertible. }
\end{aligned}
\end{align}
 The explicit construction can be found below in \eqref{map}. 
 
Throughout the paper we will make heavily use of Reynolds transport theorem, which we will use without any further reference. It says that
\begin{align}
\label{reynold}
\frac{\dd}{\dt}\int_{\Omega_{\eta(t)}}\!\!\!\!g\dx=\int_{\Omega_{\eta(t)}}\!\!\!\! \partial_tg\dx+\int_{\partial\Omega_{\eta(t)}}\!\!\!\!\partial_t\eta\circ\bfPsi_\eta^{-1}\nu\cdot\nu_\eta g\dH,
\end{align}
provided all terms are well-defined. 
The above can easily be justified by transposition and chain rule. The heuristic beyond is that $\nu$ is the direction in which the domain changes (which in our model is a fixed prescribed direction) and $\partial_t\eta$ describes the velocity of change. Therefore the scalar $\partial_t\eta\circ\bfPsi_\eta^{-1} \nu\cdot \nu_\eta$ is the derivative of the change of domains; i.e. the forces acting in direction of the outer normal of $\partial\Omega_{\eta(t)}$. For a couple of functions which satisfies $\trace_{\eta}(\phi)=b$ in the sense of Lemma \ref{lem:2.28} we have
\[
\int_{\partial{\Omega_{\eta}}}\phi \dH=\int_{\partial\Omega} \varphi\circ\bfPsi_\eta\, \abs{\det (D\bfPsi_{\zeta})}\dH.
\]
\subsection{Formal a priori estimates \& weak solutions}
We introduce the weak formulation of the momentum equation, which will be coupled to the material law of the shell. It is motivated by the a priori estimates. We will now derive these estimates formally. First, we multiply the momentum equation by $\bfu$ and integrate with respect to space (at a fixed time). We multiply the continuity equation by $\tfrac{|\bfu|^2}{2}$ and integrate with respect to space (at the same time). Subtracting both and applying chain rule yields
\begin{align*}
\int_{\Omega_{\eta}}&\partial_t\Big(\varrho\frac{|\bfu|^2}{2}\Big)\dx=\int_{\Omega_{\eta}}\partial_t(\varrho\bfu)\cdot\bfu\dx-\int_{\Omega_{\eta}}\partial_t\varrho\frac{|\bfu|^2}{2}\dx
\\
&=\int_{\Omega_{\eta}}\big(-\diver(\varrho\bfu\otimes \bfu)+\mu\Delta\bu+(\lambda+\mu)\nabla\diver\bfu
-a\nabla \varrho^{\gamma}+\varrho\bff\big)\cdot\bfu\dx
+\int_{\Omega_{\eta}}\Div(\varrho\bfu)\,\tfrac{|\bfu|^2}{2}\dx
\\
&=-\int_{\partial\Omega_{\eta}}\varrho\tfrac{|\bfu|^2}{2}\bfu\cdot\nu_\eta\dH+\int_{\Omega_{\eta}}\Div\bftau\cdot\bfu\dx+\int_{\Omega_\eta}\varrho\bff\cdot\bfu\dx
\\
&=-\int_{\partial\Omega_{\eta}}\varrho\tfrac{|\bfu|^2}{2}\bfu\cdot\nu_\eta\dH+\int_{\partial\Omega_{\eta}}\bftau\bfu\cdot\nu_\eta\dH-\int_{\Omega_{\eta}}\bftau:\nabla\bfu\dx+\int_{\Omega_\eta}\varrho\bff\cdot\bfu\dx.
\end{align*}
By Reynolds transport theorem we get
\begin{align*}
\frac{\dd}{\dd t}\int_{\Omega_{\eta}}&\varrho\frac{|\bfu|^2}{2}\dx=\int_{M}\bftau\,(\bfu\cdot\nu_\eta)\circ\bfPsi_\eta \abs{\det D\bfPsi_\eta}\dH-\mu\int_{\Omega_{\eta}}|\nabla\bfu|^2\dx
\\&\quad-(\lambda+\mu)\int_{\Omega_{\eta}}|\Div\bfu|^2\dx
+\int_{\Omega_{\eta}}a\varrho^\gamma\diver\bfu\dx+\int_{\Omega_\eta}\varrho\bff\cdot\bfu\dx.
\end{align*}
To control the pressure term, we multiply the continuity equation by $\gamma\varrho^{\gamma-1}$ and obtain by Reynold's transport theorem and the assumed boundary values that
\begin{align*}
0=&\frac{\dd}{\dt}\int_{\Omega_\eta}\varrho^\gamma\dx + (\gamma-1)\int_{\Omega_\eta}\varrho^\gamma \diver\bfu \dx.
\end{align*}
Later we will make this step rigorous via the use of so called renormalized formulations of the continuity equation, see Subsection \ref{ssec:rns} below.
Hence,
\begin{align*}
\frac{\dd}{\dd t}\int_{\Omega_{\eta}}&\Big(\varrho\frac{|\bfu|^2}{2}+\frac{a}{\gamma-1}\varrho^\gamma\Big)\dx=\int_{M}\bftau\,(\bfu\cdot\nu^\eta)\circ\bfPsi_\eta \abs{\det D\bfPsi_\eta}\dH-\mu\int_{\Omega_{\eta}}|\nabla\bfu|^2\dx
\\&\quad-(\lambda+\mu)\int_{\Omega_{\eta}}|\Div\bfu|^2\dx
+\int_{\Omega_\eta}\varrho\bff\cdot\bfu\dx.
\end{align*}
The boundary term represents the forces which are acting on the shell. Naturally these have to be in equilibrium with the bending and membrane potentials of the shell. Formally this is achieved by multiplying the shell equation~\eqref{eq:4}\footnote{Recall, that we assumed $\varepsilon_0\rho_S=1$} with $\partial_t\eta$. Using once more that
 $\bfu\circ\bfPsi_\eta=\partial_t\eta\nu$ on $M$ we find that
\begin{align*}
\frac{\dd}{\dt}\int_M \frac{|\partial_t\eta|^2}{2}\dx+\frac{\dd}{\dt} K(\eta)=\int_M \bfF\cdot\nu\partial_t\eta\dx=\int_{M}\bfF\cdot\bfu\circ\bfPsi_\eta\dx.
\end{align*}
So the right-hand sides of  both equations cancel.
Finally, we gain
\begin{align*}
&\int_{\Omega_{\eta}}\frac{\varrho(t)|\bfu(t)|^2}{2}\dx+\int_{\Omega_\eta}\frac{a}{\gamma-1}\varrho^{\gamma}(t)\dx
+\mu\int_0^t\int_{\Omega_{\eta}}|\nabla\bfu|^2\dxs+(\lambda+\mu)\int_0^t\int_{\Omega_{\eta}}|\Div\bfu|^2\dxs
\\&
\qquad+\int_M \frac{|\partial_t\eta|^2}{2}\,\dd\mathcal H^2+ \frac{K(\eta)}{2}
\\
&=\int_0^t\int_{\Omega_{\eta}}\varrho\bff\cdot\bfu\dxs+\int_0^t\int_M g\,\partial_t\eta\,\dd\mathcal H^2\ds
+\int_{\Omega_{\eta_0}}\frac{|\bfq_0|^2}{2}\dx+\frac{K(\eta_0)}{2}+\int_M\frac{|\eta_1|^2}{2}\,\dd\mathcal H^2.
\end{align*} 
This implies by H\"older's inequality and absorption~\eqref{eq:apriori0}.
In coherence with the a-priori estimates we introduce the following weak formulation of the coupled momentum equation
\begin{align}\label{eq:wu}
\begin{aligned}
&\int_I\bigg(\frac{\dd}{\dt}\int_{\Omega_{ \eta}}\varrho\bfu\cdot \bfphi\dx-\int_{\Omega_{\eta}} \varrho\bfu \cdot\partial_t\bfphi +\varrho\bfu\otimes \bfu:\nabla \bfphi\dx\bigg)\dt
\\
&\quad+\int_I\int_{\Omega_{ \eta }}\Big(
\mu\nabla\bfu:\nabla\bfphi +(\lambda+\mu)\Div\bfu\,\Div\bfphi\dxt-a\varrho^\gamma\,\Div\bfphi\Big)\dxt
\\
&\quad
+\int_I\frac{\dd}{\dt}\int_M \partial_t \eta b\dH-\int_M \partial_t\eta\,\partial_t b\dH + \int_M K'(\eta)\,b\dH\dt
\\&=\int_I\int_{\Omega_{\regkap \eta}}\varrho\bff\cdot\bfphi\dxt+\int_I\int_M g\,b\,\dd x\dt
\end{aligned}
\end{align} 
for all test-functions $(b,\bfphi)\in C^\infty_0(M)\times C^\infty(\overline{I}\times\R^3)$ with $\mathrm{tr}_\eta\bfphi=b\nu$.

\subsection{Geometry}

\label{ssec:geom}
In this section we present the background for variable domains, see \cite{LeRu} for further details. Let $\Omega\subset\R^3$ be a bounded domain with boundary $\partial\Omega$ of class $C^4$ with outer unit normal $\nu$. In the following $\Omega$ will be called the reference domain.
We define for $\alpha>0$ the
$$S_\alpha:=\{x\in \R^3:\,\,\mathrm{dist}(x,\partial\Omega)<\alpha\}.$$ 
There exists a positive number $L>0$ such that the mapping
\begin{align}
\label{Lambda}
\Lambda:\partial\Omega\times(-L,L)\rightarrow S_L,\quad\Lambda(q,s)=q+s\nu(q)
\end{align}
is a $C^3$-diffeomorphism. It is the so called Hanzawa transform.
The details of this construction may be found in \cite{Lee}. This is due to the fact, that for $C^2$-domains the closest point projection is well defined in a strip around the boundary. Indeed, its inverse $\Lambda^{-1}$ will be written as $\Lambda^{-1}(x)=(q(x),s(x))$. Here $q(x)=\text{arg min}\set{\abs{q-x}|q\in \partial\Omega}$ is the closest boundary point to $x$ (which is an orthogonal projection) and $s(x)=(x-q(x))\cdot\nu(q(x))$. For the sake of simpler notation we assume with no loss of generality that 
\[
\Lambda(q,s)\in \Omega\text{ for all }s\in [-L,0).
\]
Hence $s(x)$ is the negative distance to the boundary if $x\in \Omega$ and the positive distance, if $x\notin\Omega$. The orthogonality of the mapping is best characterized via the equation $\nabla s(x) =\nu(q(x))$. 
For a continuous function $\eta:\partial\Omega\rightarrow[-L,L]$ we define the variable domain
\begin{align}
\label{Oeta}
\Omega_\eta:=\Omega\setminus S_L\cup\{x\in S_L:\,\,s(x)<\eta(q(x))\}.
\end{align}
We set
\begin{align}
\nu_\eta(x)\text{ is defined as the outer normal at the point }x\in \partial\Omega_\eta.
\end{align}
\begin{definition}{(Function spaces)}
For $I=(0,T)$, $T>0$, and $\eta\in C(\overline{I}\times\partial\Omega)$ with $\|\eta\|< L$ we set $I\times\Omega_\eta:=\bigcup_{t\in I}\{t\}\times\Omega_{\eta(t)} \subset\R^4$. We define for $1\leq p,r\leq\infty$
\begin{align*}
L^p(I;L^r(\Omega_\eta))&:=\big\{v\in L^1(I\times\Omega_\eta):\,\,v(t,\cdot)\in L^r(\Omega_{\eta(t)})\,\,\text{for a.e. }t,\,\,\|v(t,\cdot)\|_{L^r(\Omega_{\eta(t)})}\in L^p(I)\big\},\\
L^p(I;W^{1,r}(\Omega_\eta))&:=\big\{v\in L^p(I;L^r(\Omega_\eta)):\,\,\nabla v\in L^p(I;L^r(\Omega_\eta))\big\}.
\end{align*}
\end{definition}

\begin{lemma}[\cite{LeRu}, p. 210, 211 and references given there.]
\label{lem:diffeo}
Let $\eta:\partial\Omega\rightarrow(-L,L)$ be a continuous function.
\begin{itemize}
\item[a)] There is a homomorphism $\bfPsi_\eta:\overline\Omega\rightarrow\overline{\Omega}_\eta$ such that $\bfPsi_\eta|_{\Omega\setminus S_L}$ is the identity.
\item[b)] If $\eta\in C^k(\partial\Omega)$ for $k\in\{1,2,3\}$ then $\bfPsi_\eta$ is a $C^k$ diffeomorphism.
\end{itemize}
\end{lemma}
As the impact of the geometry on the PDE is quite severe we will include an explicit construction of $\bfPsi_\eta$. 
 Since we will use the parametrisation below locally we will consider $\text{Im}(\eta)\subset [-\frac{L}{2},\frac{L}{2}]$, where $L$ is a a fixed size, such that $\Lambda$ above is well defined on the et $\partial\Omega\times [-L,L]$.
We relate to any $\eta:\partial\Omega\rightarrow(-L,L)$ the mapping $\bfPsi_\eta: \Omega  \to \Omega_\eta$, such that
\begin{align}
\label{map4}
\begin{aligned}
\bfPsi_\eta&: \Omega\to \Omega_\eta\text{ is invertible},
\\
\bfPsi_\eta&: \partial\Omega \to \partial \Omega_\eta \text{ is invertible}
\end{aligned}
\end{align}
It can be constructed as follows. Let $\phi\in C^\infty\big((-\frac{3L}{4},\infty),[0,1]\big)$ such that 
$\phi\equiv 0$ in $[-\frac{3L}{4},-\frac{L}{2}]$ and $\phi\equiv 1$ in $[-\frac{L}{4},\infty)$. Moreover, we assume that $\phi$ is a $C^k$ diffeomorphism on $[-\frac{L}{2},-\frac{L}{4}]$ with $\phi^{(l)}(-\frac{L}{2})=0=\phi^{(l)}(-\frac{L}{4})$ for all $l\in\set{1,...,k}$.
We relate to any  $\eta:\partial\Omega\rightarrow(-L,L)$ the mapping $\bfPsi_\eta: \Omega  \to \Omega_\eta$
given by
\begin{align}
\label{map}
\begin{aligned}
\bfPsi_\eta(x)&= \begin{cases}q(x)+\Big(s(x)+\eta(q(x))\phi(s(x))\Big)\nu(q(x)),&\text{ if } \dist(x,\partial\Omega)<L,\\
\quad x,\quad &\text{elsewhere}
\end{cases}.
\end{aligned}
\end{align}
Hence the two one-to-one relations in \eqref{map4} are satisfied.\\
If $\|\eta\|_\infty<\frac{L}{2}$, the mapping $\bfPsi_\eta$ can be extended such that
\begin{align}
\label{map4'}
\begin{aligned}
\bfPsi_\eta&:   \Omega_{\frac{L}{2}-\eta}\to \Omega\cup S_{\frac{L}{2}}\text{ is a invertible},
\end{aligned}
\end{align}
where due to the assumption $\|\eta\|_\infty<\frac{L}{2}$ we have that $\overline\Omega \subset \Omega_{\frac{L}{2}-\eta} \subset \Omega\cup S_{L}$.
This can be done by setting
\begin{align}
\label{map3}
\begin{aligned}
\bfPsi_\eta(x)&= 
\begin{cases}
q(x)+(s(x)+\eta(q(x)))\nu(q(x)),&\text{ if }x\not\in \Omega\text{ and }s(x)+\eta(q(x))\leq \frac{L}{2},
\\
q(x)+\Big(s(x)+\eta(q(x))\phi(s(x))\Big)\nu(q(x)),&\text{ if } x\in \Omega\text{ and }s(x)<L,
\\
\quad x,\quad &\text{elsewhere}.
\end{cases}
\end{aligned}
\end{align}

We collect a few properties of the above mapping $\bfPsi_\eta$.
\begin{lemma}
\label{lem:sobdif} Let $1<p\leq\infty$ and $\sigma\in (0,1]$.
\begin{itemize}
\item[a)]  If $\eta\in W^{2,2}(\partial\Omega)$ with $\|\eta\|_\infty<L$, then the linear mapping $\bfv\mapsto \bfv\circ \bfPsi_\eta$ ($\bfv\mapsto \bfv\circ \bfPsi_\eta^{-1}$) is continuous from $L^p(\Omega_\eta)$ to $L^r(\Omega)$ (from $L^p(\Omega)$ to $L^r(\Omega_\eta)$) for all $1\leq r<p$. 
\item[b)]If $\eta\in W^{2,2}(\partial\Omega)$ with $\|\eta\|_\infty<L$, then the linear mapping $\bfv\mapsto \bfv\circ \bfPsi_\eta$ ($\bfv\mapsto \bfv\circ \bfPsi_\eta^{-1}$) is continuous from 
$W^{1,p}(\Omega)$ to $W^{1,r}(\Omega_\eta)$ (from $W^{1,p}(\Omega_\eta)$ to $W^{1,r}(\Omega)$) for all $1\leq r<p$.
\item[c)] If $\eta\in C^{0,1}(\partial\Omega)$with $\|\eta\|_\infty<L$, then the linear mapping $\bfv\mapsto \bfv\circ \bfPsi_\eta$ ($\bfv\mapsto \bfv\circ \bfPsi_\eta^{-1}$) is continuous from 
$W^{\sigma,p}(\Omega)$ to $W^{\sigma,p}(\Omega_\eta)$ (from $W^{\sigma,p}(\Omega_\eta)$ to $W^{\sigma,p}(\Omega)$).
\end{itemize}
The continuity constants depend only on $\Omega,p,r,\sigma$, the respective norms of $\eta$.
\end{lemma}
\begin{proof}
The first two properties can be found in \cite[Lemma 2.6]{LeRu}. The third assertion follows by transposition rule and the fact, that $\nabla \bfPsi_\eta,\nabla \bfPsi_\eta^{-1}$ are uniformly bounded. Indeed, let us assume that $f\in W^{\sigma,p}(\Omega)$ for some $\sigma \in (0,1)$. Recall that this means that 
\begin{align}
\label{frac}
|f|_{W^{\sigma,p}(\Omega)}^p=\int_{\Omega}\int_{\Omega}\frac{|f(x)-f(y)|^p}{|x-y|^{3+p\sigma}}\dx\,\dd y<\infty.
\end{align}
Then
\begin{align*}
&\int_{\Omega_\eta}\int_{\Omega_\eta}\frac{|f(x)-f(y)|^p}{|x-y|^{3+p\sigma}}\dx\,\dd y
=\int_{\Omega}\int_{\Omega}\frac{|f(\bfPsi_\eta(a))-f(\bfPsi_\eta(b))|^p}{|\bfPsi_\eta(a)-\bfPsi_\eta(b)|^{3+r\sigma}}|\det(D\bfPsi_\eta)|\,\dd a\,|\det(D\bfPsi_\eta)|\,\dd b
\\
&\quad \int_{\Omega}\int_{\Omega}\frac{|\tilde f(a)-\tilde f(b)|^p}{\abs{a-b}^{n+p\sigma}}\frac{\abs{a-b}^{n+r\sigma}}{|\bfPsi_\eta(a)-\bfPsi_\eta(b)|^{3+r\theta}}|\det(D\bfPsi_\eta)|\,\dd a\,|\det(D\bfPsi_\eta)|\,\dd b
\\
&\quad =\int_{\Omega}\int_{\Omega}\frac{|\tilde f(a)-\tilde f(b)|^p}{\abs{a-b}^{3+p\sigma}}\frac{1}{|\dashint_a^b\partial_{b-a}\bfPsi_\eta(s)\dd s|^{3+p\sigma}}|\det(D\bfPsi_\eta)|\,\dd a\,|\det(D\bfPsi_\eta)|\,\dd b
\\
&\quad\leq  c\norm{\det(D\bfPsi_\eta)}_\infty^2\norm{\nabla \bfPsi_\eta^{-1}}_\infty \int_{\Omega}\int_{\Omega}\frac{|\tilde f(a)-\tilde f(b)|^p}{\abs{a-b}^{3+p\sigma}}\dd a\dd b
\end{align*}
In case $\sigma=1$ the result follows directly by transposition rule.
Since the argument can be applied analogously for $f\circ\bfPsi_{\eta}^{-1}$ the proof is completed.
\end{proof}
Combining Lemma \cite{LeRu}, Cor. 2.9] with Sobolev's embedding in two dimensions we obtain the following trace embedding.
\begin{lemma}\label{lem:2.28}
Let $1<p<3$ and $\eta\in W^{2,2}(\partial\Omega)$ with $\|\eta\|_{L^\infty(\partial\Omega)}<L$. Then
\begin{enumerate}
\item 
The linear mapping
$\tr_\eta:v\mapsto v\circ\bfPsi_\eta|_{\partial\Omega}$ is well defined and continuous from
 $W^{1,p}(\Omega_\eta)$ to $W^{1-\frac1r,r}(\partial \Omega)$ for all $r\in (1,p)$ and well defined and continuous from  $W^{1,p}(\Omega_\eta)$ to $L^{q}(\partial\Omega)$ for all $1<q<\frac{2p}{3-p}$.
\item the linear mapping
$\tr_\eta:v\mapsto v|_{\partial\Omega_\eta}$ is well defined and continuous from
 $W^{1,p}(\Omega_\eta)$ to $W^{1-\frac1r,r}(\partial \Omega_{\eta})$ for all $r\in (1,p)$
and well defined and continuous from  $W^{1,p}(\Omega_\eta)$ to $L^{q}(\partial\Omega_{\eta})$ for all $1<q<\frac{2p}{3-p}$.
\end{enumerate}
The continuity constant depends only on $\Omega,p,$ $\|\eta\|_{W^{2,2}}$ and
$\tau(\eta)$.
\end{lemma}
The following lemma allows us the extend functions defined on the variable domain
to the whole space $\R^3$. This is not trivial for $\eta\in W^{2,2}$ because the boundary is not Lipschitz. However, it requires the additional assumption $\|\eta\|_{L^\infty(\partial\Omega)}<\frac{L}{2}$.
\begin{lemma}
\label{lem:extension}
Let $1\leq r<p<\infty$ and $\eta\in W^{2,2}(\partial\Omega)$ with $\|\eta\|_{L^\infty(\partial\Omega)}<\frac{L}{2}$. There is a continuous linear operator $\mathscr E_\eta:W^{1,p}(\Omega_\eta)\rightarrow W^{1,r}(\R^3)$ such that $\mathscr E_\eta\big|_{\Omega_\eta}=\mathrm{Id}$.
\end{lemma}
\begin{proof}
If $\eta\in W^{1,\infty}(\partial\Omega)$ the result is standard.
There is a continuous linear operator
$$E_{p,A}:W^{1,p}(A)\rightarrow W^{1,p}(\R^3)$$
for any bounded Lipschtz domain $A$ and $1\leq p<\infty$, see, for instance \cite[Thm. 5.28]{AdFo}, where even slightly less regularity is required.\\
For the general case we use the extension above, to transfer from a functions space over the variable domain to a function from a functions space over the reference domain. Indeed Lemma~\ref{lem:sobdif} implies that
\begin{align*}
W^{1,p}(\Omega_{\eta})&\ni u\mapsto u^{\eta}\in W^{1,\tilde{r}}(\Omega),\quad u^{\eta}(x)=u(\bfPsi_{\eta}(x)),
\end{align*}
for any $\tilde r<p$.
Since $\Omega$ is assumed to have a uniform Lipschitz boundary it is possible to extend the function on the whole space. Hence, using the Extension operator on $\Omega$, we find
\[
E_{\tilde{r},\Omega}(u^\eta)\in W^{1,\tilde{r}}(\setR^3)\text{ and }E_{\tilde{r},\Omega}(u^\eta)|_{\Omega}=u^\eta.
\]
In order to transform back we use the fact, that $\bfPsi_\eta$ is invertible
on $\Omega^*_\eta$ where $\overline\Omega \subset\Omega_\eta^\ast$ due to the assumption $\|\eta\|_{L^\infty(\partial\Omega)}<\frac{L}{2}$, cf. \eqref{map4'}. By Lemma~\ref{lem:sobdif} we find
\begin{align*}
W^{1,\tilde{r}}(\Omega_{\eta}^*)&\ni v\mapsto v_{\eta}\in W^{1,r}(\Omega\cup S_{\frac{L}{2}}),\quad v_{\eta}(x)=v(\bfPsi_{\eta}^{-1}(x)).
\end{align*}
for any $r<\tilde r<p$.
Finally, we set
\begin{align*}
\mathscr E_\eta u= E_{r,\Omega\cup S_{\frac{L}{2}}}\big((E_{\tilde r,\Omega} u^\eta)\big|_{\Omega^*_\eta}\big)_\eta\quad u\in W^{1,p}(\Omega_\eta).
\end{align*}
It is now easy to check that $\mathscr E_\eta$ has the required properties.
\end{proof}

\begin{remark}
If $\eta\in L^\infty(I;W^{2,2}(\partial\Omega))$ we obtain non-stationary variants of the results stated above.
\end{remark}


\subsection{Convergence on variable domains.}
Due to the variable domain the framework of Bochner spaces is not available. Hence, we cannot use the Aubin-Lions compactness theorem. In this subsection we are concerned with the question how to get compactness anyway.
We start with the following definition of convergence in variable domains. We start wit convergence in Lebesgue spaces which follows from n extension by zero.
\begin{definition}
\label{def:conv}
Let $(\eta_i)\subset C(\overline{I}\times \partial\Omega;[-\theta L, \theta L])$, $\theta\in(0,1)$, be a sequence with $\eta_i\rightrightarrows \eta$. Let $p\in [1,\infty]$ and $k\in\N_0$.
\begin{enumerate}
\item We say that  a sequence $(g_i) \subset L^p(I,L^q(\Omega_{\eta_i}))$ converges to $g$ in $L^p(I,L^q(\Omega_{\eta_i}))$ strongly with respect to $(\eta_i)$, in symbols
$
g_i\toetai g \,\text{in}\, L^p(I,L^q(\Omega_{\eta_i})),
$
if
\begin{align*}
\chi_{\Omega_{\eta_i}}g_i\to \chi_{\Omega_\eta}g \quad\text{in}\quad L^p(I,L^q(\R^3)).
\end{align*}
\item Let $p,q<\infty$. We say that  a sequence $(g_i) \subset L^p(I,L^q(\Omega_{\eta_i}))$ converges to $g$ in $L^p(I,L^q(\Omega_{\eta_i}))$ weakly with respect to $(\eta_i)$, in symbols
$
g_i\weaktoetai g \,\text{in}\, L^p(I,L^q(\Omega_{\eta_i})),
$
if
\begin{align*}
\chi_{\Omega_{\eta_i}}g_i\weakto \chi_{\Omega_\eta}g \quad\text{in}\quad L^p(I,L^q(\R^3)).
\end{align*}
\item Let $p=\infty$ and $q<\infty$. We say that  a sequence $(g_i) \subset L^\infty(I,L^q(\Omega_{\eta_i}))$ converges to $g$ in $L^\infty(I,L^q(\Omega_{\eta_i}))$ weakly$^*$ with respect to $(\eta_i)$, in symbols
$
g_i\weakto^{\ast,\eta} g \,\text{in}\, L^\infty(I,L^q(\Omega_{\eta_i})),
$
if
\begin{align*}
\chi_{\Omega_{\eta_i}}g_i\weakto^* \chi_{\Omega_\eta}g \quad\text{in}\quad L^\infty(I,L^q(\R^3)).
\end{align*}
\end{enumerate}
\end{definition}
Note that in case of one single $\eta$ (i.e. not sequence) the space
$L^p(I,L^q(\Omega_{\eta}))$ (with $1\leq p<\infty$ and $1<q<\infty$) is reflexive and we have the usual duality pairing
\begin{align}\label{0903}
L^p(I,L^q(\Omega_{\eta}))\cong L^{p'}(I,L^{q'}(\Omega_{\eta}))
\end{align} 
provided $\eta$ is smooth enough, see \cite{NRL}.
Definition \ref{def:conv} can be extended in a canonical way to Sobolev spaces.
We say that  a sequence $(g_i) \subset L^p(I,W^{1,q}(\Omega_{\eta_i}))$ converges to $g \in L^p(I,W^{1,q}(\Omega_{\eta_i}))$ strongly with respect to $(\eta_i)$, in symbols
\begin{align*}
g_i\toetai g \quad\text{in}\quad L^p(I,W^{1,p}(\Omega_{\eta_i})),
\end{align*}
if both $g_i$ and $\nabla g_i$ converges (to $g$ and $\nabla g$ respectively) in $L^p(I,W^{1,q}(\Omega_{\eta_i}))$ strongly with respect to $(\eta_i)$ (in the sense of Definition \ref{def:conv} a)).
We also define weak and weak$^*$ convergence in Sobolev spaces with respect to $(\eta_i)$ with an obvious meaning. Note that an extension to higher order Sobolev spaces is possible but not needed for our purposes.\\
\subsection{A lemma of Aubin-Lions type for time dependent domains}
For the next compactness lemma we require the following on the functions describing the boundary 
\begin{enumerate}[label={\rm (A\arabic{*})}, leftmargin=*]
\item\label{A1} Assume that the sequence $(\eta_i)\subset C(\overline I\times M;[-\theta L, \theta L])$, $\theta\in(0,1)$,  satisfies
\begin{align*}
\eta_i&\rightharpoonup^\ast\eta\quad\text{in}\quad L^\infty(I,W^{2,2}_0(M)),\\
\partial_t\eta_i&\rightharpoonup^\ast\partial_t\eta\quad\text{in}\quad L^\infty(I,L^{2}(M)).
\end{align*}
\item\label{A2} Let $(v_i)$ be a sequence such that for $p,q\in[1,\infty)$
$$v_i\weaktoetai v\quad\text{in}\quad L^p(I;W^{1,s}(\Omega_{\eta_i})).$$
\item\label{A3} Let $(r_i)$ be a sequence such that for $m,b\in[1,\infty)$
$$r_i\weaktoetai r\quad\text{in}\quad L^m(I;L^{b}(\Omega_{\eta_i})).$$
Assume further that there are sequence $(\bfH^1_i)$, $(\bfH^2_i)$ and $(h_i)$, bounded in $L^m(I;L^{b}(\Omega_{\eta_i}))$, such that $\partial_t r_i=\Div\Div \bfH_i^1+\Div \bfH_i^2+h_i$ in the sense of distributions, i.e.,
\begin{align*}
\int_I\int_{\Omega_{\eta_i}}r_i\,\partial_t\phi\dxt=\int_I\int_{\Omega_{\eta_i}}\bfH^1_i\cdot\nabla^2\phi\dxt
-\int_I\int_{\Omega_{\eta_i}}\bfH^2_i\cdot\nabla\phi\dxt
+\int_I\int_{\Omega_{\eta_i}}h_i\cdot\nabla\phi\dxt
\end{align*}
for all $\phi\in C^\infty_0(I\times\Omega_{\eta_i})$.
\end{enumerate}
\begin{lemma}
\label{thm:weakstrong}
Let $(\eta_i)$, $(v_i)$ and $(r_i)$ be sequence satisfying \ref{A1}--\ref{A3}
 where $\frac{1}{s^*}+\frac{1}{b}=\frac{1}{a}<1$\footnote{Here we set $s^*=\frac{3s}{3-s}$, if $s\in (1,3)$ and otherwise $s^*$ can be fixed in $(1,\infty)$ conveniently.} and $\frac{1}{m}+\frac{1}{p}=\frac{1}{q}<1$. Then there is a subsequence with
\begin{align}
\label{al}
v_i r_i\weakto v r\text{ weakly in }L^{q}(I,L^a(\Omega_{\eta_i})).
\end{align}
\end{lemma}
\begin{remark}
Assumption \ref{A3} in Lemma \ref{thm:weakstrong} can be extended in an obvious way to the case of higher order distributional derivatives. We have chosen the version above as it is most suitable for our applications.
\end{remark}
\begin{proof} First we show local compactness. Consider a cube $Q=J\times B$
such that $Q\Subset \Omega_{\eta_i}^I$ for all $i$ large enough. By \ref{A3} we know that $r_i$ is bounded in $L^m(I;L^{b}(B))$ and that $\partial_t r_i$ is bounded in $L^m(I;W^{-1,b}(B))$. We can apply the classical Aubin-Lions compactness Theorem \cite{Lialt} for the triple
\begin{align*}
L^b(B)\hookrightarrow\hookrightarrow W^{-1,b}(B)\hookrightarrow W^{-2,b}(B),
\end{align*}
and gain
\begin{align}\label{3101}
r_i\rightarrow r\quad\text{in}\quad L^m(I;W^{-1,b}(B)).
\end{align}
Note that we do not have to take a subsequence since the original sequence already converges by \ref{A3}.
Now note that \ref{A1} implies
$$\eta_i\rightarrow\eta\quad\text{in}\quad C^\alpha(I\times M)$$
for some $\alpha\in(0,1)$ by interpolation. Consequently, for a given $\kappa>0$ there is $i_0=i_0(\kappa)$ such that
\begin{align}\label{3101b}
\mathcal L^4\Big(I\times(\Omega_{\eta}\setminus \Omega_{l})\Big)\leq \kappa\quad \forall l\geq i_0,
\end{align}
where we have set
\begin{align*}
\Omega_{l}=\bigcap_{i\geq l} \Omega_{\eta_i}.
\end{align*}
Now we fix a measurable set $A_\kappa\Subset I\times\Omega_{l}$ with $\mathcal L^4\Big(\big(I\times\Omega_{l}\big)\setminus A_\kappa\Big)\leq \kappa$ and cover it by at most countable many cubes
$Q_k=J_k\times B_k$ such that
\begin{align*}
A_\kappa\subset  \bigcup_k Q_k\Subset \big(I\times\Omega_{l}\big).
\end{align*}
They can be chosen in such a way that we find $\psi_k$ a partition of unity for the family $Q_k$ such that $\psi_k\in C_0^\infty(Q_k)$ and 
\begin{align*}
\sum \psi_k =1 \text{ on }A_\kappa.
\end{align*}
In particular, by taking a diagonal sequence we can assume that \eqref{3101} holds with $Q=Q_k$. For $w\in C^\infty_0(I\times\R^3)$ we have
\begin{align*}
\int_I\int_{\R^3}\big(\chi_{\Omega_{\eta_i}} r_i v_i-\chi_{\Omega_{\eta}} r v\big)\,w\dxt&=\sum_{k}\int_I\int_{\R^3}\big( r_i v_i- r v\big)\,\psi_kw\dxt \\
&+\int_I\int_{\R^3}\big(\chi_{\Omega_{\eta_i}} r_i v_i-\chi_{\Omega_{\eta}} r v\big)\,\sum_{k}(1-\psi_k)w\dxt.
\end{align*}
On account of  \eqref{3101} (with $Q=Q_k$) and \ref{A2} the first integral on the right-hand side converges to zero. Due to \eqref{3101b} the second integral can be bounded in terms of $\kappa$.
Here we took into account the boundedness of $\chi_{\Omega_{\eta_i}} r_i v_i$
in $L^r(I\times \R^3)$ for some $r>1$ which is a consequence of the assumptions $\frac{1}{s}+\frac{1}{b}<1$ and $\frac{1}{m}+\frac{1}{p}<1$. As $\kappa$ is arbitrary we have shown 
\begin{align*}
\lim_{i\rightarrow\infty}\int_I\int_{\R^3}\big(\chi_{\Omega_{\eta_i}} r_i v_i-\chi_{\Omega_{\eta}} r v\big)\,w\dxt&=0,
\end{align*}
which means we have
\begin{align}\label{3101c}
\chi_{\Omega_{\eta_i}} r_i v_i\rightarrow \chi_{\Omega_{\eta}} r v\quad\text{in}\quad \mathcal D'(I\times\R^3).
\end{align}
However, our assumptions imply that $\chi_{\Omega_{\eta_i}} r_i v_i$ converges weakly in $L^q(I,L^a(\R^3))$ at least after taking a subsequence.
As a consequence of \eqref{3101c} we can identify the limit and the claim follows.
\end{proof}
\begin{remark}
\label{rem:strong}
In the case $r_i=v_i$, we find that
\[
\int_0^T\int_{\Omega_{\eta_i}}\abs{v_i}^2\, dx\, dt\to \int_0^T\int_{\Omega_{\eta}}\abs{v}^2\, dx\, dt.
\]
Since weak convergence and norm convergence implies strong convergence, we find (by interpolation) that
\[
v_i\to v\text{ strongly in }L^{2}(I, L^{2}(\Omega_{\eta_i})).
\]
Showing such a result for the velocity field in the context of incompressible fluid mechanics is the main achievement of the paper \cite{LeRu}. There, different from \ref{A3}, the time derivative is only a distribution acting on divergence-free test-functions in the incompressible. In contrast to the compactness arguments in \cite{LeRu}, the proof of Lemma \ref{thm:weakstrong} does not face this difficulty.
\end{remark}

\section{The continuity equation in variable domains}
\label{subsec:continuity}
\subsection{Renormalized solutions in time dependent domains} 
\label{ssec:rns}
This subsection is concerned with the study of the continuity
\begin{align}\label{eq:2003strong}
\partial_t \varrho&+\Div(\varrho\bfu)=0
\end{align}
in a variable domain $\Omega_\eta$ with $\eta\in L^\infty(I;W^{2,2}(\partial\Omega))$ and $\bfu\in L^2(I;W^{1,2}(\Omega_\eta))$. Observe the following interplay of the two terms of the material derivative, that shall be used many times within this work. A (strong) solution to \eqref{eq:warme} satisfies for any $\psi\in C^\infty(\overline{I}\times \R^3)$
\begin{align*}
&\int_{\Omega_\eta}\Big(\partial_t \varrho \psi +\divergence (\varrho \bfu)\psi\Big)\dx
=\int_{\Omega_\zeta}\Big(\partial_t (\varrho \psi) +\divergence (\varrho \bfu\psi)\Big)\dx-\int_{\Omega_\eta}\Big( \varrho \partial_t\psi +\varrho \bfu\cdot\nabla \psi\Big)\dx\\
&\quad=\frac{\dd}{\dt}\int_{\Omega_\eta}\varrho\psi\dx+\int_{\partial\Omega_\zeta}\varrho\psi(\bfu-(\partial_t\zeta\nu)\circ\bfPsi_\zeta)\cdot\nu_\zeta \dd\mathcal{H}^2\dt
-\int_{\Omega_\eta}\Big(\varrho \partial_t\psi +\varrho \bfu\cdot\nabla \psi\Big)\dx.
\end{align*}
In the case of our consideration we find that $\tr_\zeta(\bfu)=\partial_t\eta\nu$ and have
\begin{align}\label{eq:2003weak}
\int_I\frac{\dd}{\dt}\int_{\Omega_\eta}\varrho\psi\dxt
-\int_I\int_{\Omega_\eta}\Big(\varrho \partial_t\psi +\varrho \bfu\cdot\nabla \psi\Big)\dxt=0.
\end{align}
for all $\psi\in C^\infty(\overline{I}\times\R^3)$.
Equation \eqref{eq:2003weak} will serve as a weak formulation of \eqref{eq:2003strong}.
It is worth mentioning, that by taking $\psi\equiv \chi_{[0,t]}$ in \eqref{eq:2003weak} we find that the total mass is conserved in the sense that
\begin{align*}
\int_{\Omega_{\eta(t)}}\varrho(t)\dx=\int_{\Omega_{\eta(0)}}\varrho(0)\dx
\end{align*}
for all $t\in I$.
Following the DiPerna-Lions \cite{DL}
we will introduce a renormalized formulation which will be of crucial importance for the reminder of the paper.
A crucial observation is that the formulation in \eqref{eq:2003weak}
can be extended to the whole space despite the fact that $\bfu$ does not have zero boundary values (this will be essential to prove strong convergence of the density, see Subsection \ref{sec:strongrhoep}). In fact, we have
\begin{align}\label{eq:2003weak1}
\int_I\frac{\dd}{\dt}\int_{\R^3}\varrho\psi\dxt
-\int_I\int_{\R^3}\Big(\varrho \partial_t\psi +\varrho \bfu\cdot\nabla \psi\Big)\dxt=0.
\end{align}
for all $\psi\in C^\infty(\overline{I}\times\R^3)$ provided we extend $\varrho$ by zero.
It is essential to introduce the principle of normalized solutions on variable domains. In explicit we wish to study the family of solutions $\theta(\varrho)$, where $\theta\in C^{2}(\setR^+;\setR^+)$, such that $\theta''=0$ for large values and $\theta(0)=0$. In the following we only argue formally. For a rigorous derivation of the renormalized continuity equation we refer to the next subsection and the derivations before \eqref{eq:ren'}.
We may use the test function $
\theta'(\varrho)\psi$ with $\psi\in C^\infty(\overline{I}\times\R^3$) in \eqref{eq:2003weak1} and hence find that
\begin{align*}
0&=\int_I\frac{\dd}{\dt}\int_{\R^3}\varrho\theta'(\varrho)\,\psi\dxt
-\int_I\int_{\R^3}\Big(\partial_t(\varrho\theta'(\varrho)\,\psi)-\partial_t\varrho\theta'(\varrho)\psi+\diver(\varrho \theta'(\varrho) \bfu\, \psi)\Big)\dxt
\\
&+\int_I\int_{\R^3}\Big(\nabla\varrho \theta'(\varrho)\cdot \bfu\, \psi+\varrho \theta'(\varrho)\diver \bfu\Big) \dxt
\\
&=\int_I\frac{\dd}{\dt}\int_{\R^3}\varrho\theta'(\varrho)\psi\dx
-\int_{\R^3}\Big(\partial_t(\varrho\theta'(\varrho)\,\psi)+\diver(\varrho \theta'(\varrho) \bfu\, \psi)\Big)\dx
\\
&\quad +\int_I\int_{\R^3}\Big(\partial_t\theta(\varrho)\psi +\diver(\theta(\varrho) \bfu) \,\psi\Big)\dx
 + \int_{\R^3}(\varrho\theta'(\varrho)-\theta(\varrho)) \diver\bfu\, \psi\dx.
\end{align*}
Now partial integration and Reynolds transport theorem imply that the first line vanishes. Again partial integration and Reynolds transport theorem transfer the second line in the appropriate weak formulation. Hence we find the renormalized formulation as
\begin{align}
\label{eq:renorm}
\begin{aligned}
0=&\int_I\frac{\dd}{\dt}\int_{\Omega_\eta}\theta(\varrho)\,\psi\dxt-\int_I\int_{\Omega_\eta}\Big(\theta(\varrho)\partial_t\psi +\theta(\varrho) \bfu\cdot \nabla \psi\Big)\dxt
\\
&\quad + \int_I\int_{\Omega_\eta}(\varrho\theta'(\varrho)-\theta(\varrho)) \diver\bfu\, \psi\dxt.
\end{aligned}
\end{align}

\subsection{The damped continuity equation in time dependent domains}
\label{ssec:damp}
We will need very explicit a-priori information about our approximation of the density $\varrho$. The necessary result is collected in Theorem~\ref{lem:warme} below. For the analogous results for fixed in time domains see~\cite[section 2.1]{feireisl1}.
 We will assume that the moving boundary is prescribed by a function $\zeta$ of class $C^{3}(\overline{I}\times M)$. 
 For a given function $\bfw\in L^2(I;W^{1,2}(\Omega_\zeta))$ and $\varepsilon>0$ we consider the equation
\begin{align}\label{eq:warme}\begin{aligned}
\partial_t \varrho&+\Div(\varrho\bfw)=\varepsilon\Delta\varrho\quad\text{in}\quad I\times\Omega_\zeta,\\
\varrho(0)&=\varrho_0\text{ in }\Omega_{\zeta(0)},\quad\partial_{\nu_\zeta}\varrho\big|_{\partial\Omega_\zeta}=\tfrac{1}{\varepsilon}\varrho(\bfw-(\partial_t\zeta\nu)\circ\bfPsi^{-1}_\zeta)\cdot\nu_\zeta\quad\text{on}\quad I\times\partial\Omega_\zeta
\end{aligned}
\end{align}
 A solution to \eqref{eq:warme} satisfies (formally) for any $\psi\in C^\infty(\overline{I}\times \R^3)$
\begin{align*}
&\int_{\Omega_\zeta}\Big(\partial_t \varrho \psi +\divergence (\varrho \bfw)\psi\Big)\dx
=\int_{\Omega_\zeta}\Big(\partial_t (\varrho \psi) +\divergence (\varrho \bfw\psi)\Big)\dx-\int_{\Omega_\zeta}\Big( \varrho \partial_t\psi +\varrho \bfw\cdot\nabla \psi\Big)\dx\\
&\quad=\frac{\dd}{\dt}\int_{\Omega_\zeta}\varrho\psi\dx+\int_{\partial\Omega_\zeta}\varrho\psi(\bfw-(\partial_t\zeta\nu)\circ\bfPsi_\zeta)\cdot\nu_\zeta \dd\mathcal{H}^2\dt
-\int_{\Omega_\zeta}\Big(\varrho \partial_t\psi +\varrho \bfw\cdot\nabla \psi\Big)\dx.
\end{align*}
On the other hand we have
\begin{align*}
\int_{\Omega_\zeta}\varepsilon\Delta\varrho\,\psi\dx&=\int_{\Omega_\zeta}\varepsilon\partial_\nu\varrho\,\psi\dH-\int_I\int_{\Omega_\zeta}\varepsilon\nabla\varrho\cdot\nabla\psi\dxt\\
&=\int_{\partial\Omega_\zeta}\varrho\psi(\bfw-(\partial_t\zeta\nu)\circ\bfPsi_\zeta)\cdot\nu_\zeta \dd\mathcal{H}^2\dt-\int_I\int_{\Omega_\zeta}\varepsilon\nabla\varrho\cdot\nabla\psi\dxt.
\end{align*}
This motivates the choice of the Neumann boundary values in \eqref{eq:warme} which imply following neat weak formulation:
\begin{align}
\int_I\frac{\dd}{\dt}\int_{\Omega_\zeta}\varrho\psi\dx\dt
-\int_I\int_{\Omega_\zeta}\big(\varrho \partial_t\psi +\varrho \bfw\cdot\nabla \psi\big)\dxt=-\int_I\int_{\Omega_\zeta}\varepsilon\nabla\varrho\cdot\nabla\psi\dxt\label{eq:weak-masse}
\end{align}
for all $\psi\in C^\infty(\overline{I}\times \R^3)$. We wish to emphasize that this weak formulation is canonical with respect to the moving boundary. It is the only formulation which preserves mass. This turns out to be the essential property to gain the necessary estimates and correlations.

\begin{theorem}\label{lem:warme}Let $\zeta\in C^{3}(\overline{I}\times M,[\frac{L}2,\frac{L}2])$ be the function describing the boundary. Assume that $\bfw\in L^2(I;W^{1,2}(\Omega_\zeta))$ and
$\varrho_0\in L^{2}(\Omega_{\zeta(0)})$.
\begin{enumerate}
\item[a)] There is a unique weak solution $\varrho$ to \eqref{eq:weak-masse} such that
$$\varrho\in L^\infty(0,T;L^2(\Omega_{\zeta}))\cap L^2(0,T;W^{1,2}(\Omega_{\zeta})).$$
\item[b)] Let $\theta\in C^2(\setR_+;\setR_+)$ be such that $\theta''(s)\equiv 0$
 for large values of $s$ and $\theta(0)=0$.\footnote{The restriction $\theta(0)=0$ is only needed if we extend the equation to the whole space.} Then the following holds, for the canonical zero extension of $\varrho\equiv\varrho\chi_{\Omega_{\zeta}}$:
\begin{align}
\label{eq:renormz}
\begin{aligned}
\int_I\partial_t&\int_{\R^3} \theta(\varrho)\,\psi\dxt-\int_{I\times \R^3}\theta(\varrho)\,\partial_t\psi\dxt
\\
=&-\int_{I\times \R^3}\chi_{\Omega_{\zeta}}\big(\varrho\theta'(\varrho)-\theta(\varrho)\big)\Div\bfw\,\psi\dx+\int_{I\times \R^3}\theta(\varrho) \bfw\cdot\nabla\psi\dxt\\ &-\int_{I\times\R^3}\varepsilon\chi_{\Omega_{\zeta}}\nabla \theta(\varrho)\cdot\nabla\psi\dxt-\int_{I\times \R^3}\varepsilon\chi_{\Omega_{\zeta}}\theta''(\varrho)|\nabla\varrho|^2\psi\dxt
\end{aligned}
\end{align}
for all $\psi\in C^\infty(\overline{I}\times\R^3)$.
\item[c)] Assume that $\varrho_0\geq0$ a.e. in $\Omega_{\zeta(0)}$. Then we have $\varrho\geq0$ a.e. in $I\times\Omega_\zeta$. 
\end{enumerate}
\end{theorem}
\begin{proof}
In order to find a solution to \eqref{eq:weak-masse} we discretise
the system.
It is standard to find a smooth orthonormal basis $(\tilde \omega_k)_{k\in\N}$
of $W^{1,2}(\Omega)$. Now define pointwise in $t$
\begin{align*}
\omega_k:=\tilde \omega_k\circ \bfPsi_{\zeta}^{-1}.
\end{align*}
By Lemma \ref{lem:diffeo} we still know that $\omega_k$ belongs to the class $C^3(\overline{\Omega}_{\zeta}(t))$.
Obviously, $(\omega_k)_{k\in\N}$ forms a basis of $W^{1,2}(\Omega_{\zeta}(t))$.  We fix the initial values as the $L^{2}(\Omega_{\zeta}(0))$-projection onto $W_N=\Span\{\omega_1,...\omega_N\}$ such that
\[
\varrho^N_0\to \varrho_0\quad\text{in}\quad L^{2}(\Omega_{\zeta(0)}).
\]
 We are looking for a function $\varrho_N=\sum\beta_k\omega_k$ satisfying for all $l=1,...,N$
\begin{align}
\frac{\dd}{\dt}\int_{\Omega_\zeta}\varrho_N\,\omega_l\dx
-\int_{\Omega_\zeta}\big(\varrho_N \,\partial_t\omega_l +\varrho_N \bfw\cdot\nabla \omega_l\big)\dx
=-\int_{\Omega_\zeta}\varepsilon\nabla\varrho_N\cdot\nabla\omega_l\dx,\label{eq:0102}
\end{align}
and respective initial values $\varrho_0^N$.
This is equivalent to
\begin{align}\nonumber
\frac{\dd\beta_k}{\dt}\int_{\Omega_\zeta}\omega_k\omega_l\dx&=-\beta_k\frac{\dd}{\dt}\int_{\Omega_\zeta}\omega_k\omega_l\dx+\beta_k\int_{\Omega_\zeta}\big(\omega_k \partial_t\omega_l +\omega_k \bfw\cdot\nabla \omega_l\big)\dx
\\
&\quad -\beta_k\int_{\Omega_\zeta}\varepsilon\nabla\omega_k\cdot\nabla\omega_l\dx,\label{eq:weak-masseN}
\end{align}
and $\beta^l(0)=\beta_0^l$.
Her the $\beta_k$'s are the unknowns (as functions only on time). Now, we define
the matrices $\mathcal A,\mathcal B\in\R^{N\times N}$ by
\begin{align*}
\mathcal A_{k,l}&=\int_{\Omega_\zeta}\omega_k\omega_l\dx,\\
\mathcal B_{k,l}&=-\frac{\dd}{\dt}\int_{\Omega_\zeta}\omega_k\omega_l\dx
+\int_{\Omega_\zeta}\big(\omega_k \partial_t\omega_l +\omega_k \bfw\cdot\nabla \omega_l\big)\dx
\\
&\quad
-\int_{\Omega_\zeta}\varepsilon\nabla\omega_k\cdot\nabla\omega_l\dx.
\end{align*}
Because $(\omega_k)$ is a basis of $(\omega_k)_{k\in\N}$ the matrix $\mathcal A$ is positive definite. Hence \eqref{eq:weak-masseN} can be written as
$\bfbeta'=\mathcal A^{-1}\mathcal B\bfbeta$ where $\bfbeta$ is the vector containing the $\beta_k$'s. This is a linear system of ODEs which has a unique solution. In order to pass to the limit $N\rightarrow\infty$ we need uniform a priori estimates. So, we multiply \eqref{eq:weak-masseN} by $\beta_l$ and sum over $l$. \begin{align*}
\frac{\dd}{\dt}&\int_{\Omega_\zeta}\frac{|\varrho_N|^2}{2}\dx+\int_{\Omega_\zeta}\varepsilon|\nabla\varrho_N|^2\dx\\&=\int_{\partial\Omega_\zeta}\frac{|\varrho_N|^2}{2}(\partial_t\zeta\nu)\circ\bfPsi_\zeta^{-1}\cdot\nu_\zeta \dd\mathcal{H}^2
+\int_{\Omega_\zeta}\varrho_N \bfw\cdot\nabla \varrho_N\dx\\
&\leq\,c\,\int_{\partial\Omega_\zeta}|\varrho_N|^2\dd\mathcal{H}^2+c\,\int_{\Omega_\zeta}|\varrho_N||\nabla\varrho_N|\dx=:(I)_N+(II)_N.
\end{align*}
We use Lemma~\ref{lem:sobdif} and the trace theorem $W^{\frac{1}{2},2}(\Omega_\zeta)\rightarrow L^2(\partial\Omega_\zeta)$ (note that $\zeta$ is Lipschitz continuous uniformly in time) to conclude that (see \cite[Chapter 7.7]{AdFo})
\[
(I)_N=\|\varrho_N\|_2^2\leq c(\zeta)\big\|\varrho_N\big\|^2_{\frac{1}{2},2}= c(\zeta)\big(\big\|\varrho_N\big\|_2^2+\big|\varrho_N\big|_{\frac{1}{2},2}^2\big),
\]
where $|\cdot|_{\frac{1}{2},2}$ is given by \eqref{frac}.
Interpolating $W^{\frac{1}{2},2}$ between $L^2$ and $W^{1,2}$ (see see \cite[Chapter 7.3]{AdFo}) we obtain for $\kappa>0$ arbitrary
\begin{align}\label{eq:0102b}
(I)_N=\int_{\partial\Omega_\zeta}|\varrho_N|^2\dd\mathcal{H}^2
\leq \kappa \int_{\Omega_\zeta}\abs{\nabla\varrho_N}^2\dx + c(\kappa)  \int_{\Omega_\zeta}|\varrho|^2\dx.
\end{align}
The same estimate holds for $(II)_N$ by a simple application of Young's inequality. Combining both and apply Gronwall's lemma we have shown
\begin{align*}
\sup_{I}&\int_{\Omega_\zeta}|\varrho_N|^2\dx+\int_I\int_{\Omega_\zeta}\varepsilon|\nabla\varrho_N|^2\dxt\leq\,C\int_{\Omega_{\zeta(0)}}|\varrho_0|^2\dx
\end{align*}
uniformly in $N$, where $C$ depends on $ \xi$, $\norm{\bfw}_\infty$ and $\abs{I}$ only. 
 Hence we obtain the existence of a positive limit functions
$$\varrho\in L^\infty(I;L^2(\Omega_{\zeta}))\cap L^2(I;W^{1,2}(\Omega_{\zeta}))$$ 
using \eqref{0903}. Moreover, $\varrho^N$ converges weakly (weakly$^\ast$) to $\varrho$.
The passage to the limit in \eqref{eq:0102} is obvious as it is a linear equation. The uniqueness is shown in the following way. Assume that we have two solutions $\rho_1, \rho_2$. The differences of the two solutions $\varrho_1-\varrho_2$ and $\rho_2-\rho_1$ are both solutions with zero initial datum. Now we may take $\phi\equiv1$ as a test-function for both solutions and find that
\[
0\leq \int_{\Omega_\zeta(t)}\big(\varrho_1(t)-\varrho_2(t)\big)\dx\leq 0\quad\text{for all}\quad t\in I.
\]
 Hence, a) is shown.\\
 Next we show b). We extend $\varrho$ by zero to $I\times\R^3$ and obtain
\begin{align*}
\int_I\frac{\dd}{\dt}\int_{\R^3}\varrho\,\psi\dx\dt
-\int_I\int_{\R^3}\big(\varrho\, \partial_t\psi +\varrho \bfw\cdot\nabla \psi\big)\dxt=-\int_I\int_{\R^3}\varepsilon\chi_{\Omega_{\zeta}}\nabla\varrho\cdot\nabla\psi\dxt
\end{align*}
for all $\psi\in C^\infty(\overline I\times \R^3)$. 
Now, we mollify the equation in space
using a standard convolution with parameter $\kappa>0$. Then we find that the following PDE is satisfied:
\begin{align}\label{eq:2002a}
\partial_t\varrho_\kappa
+\Div\big( \varrho \bfw\chi_{\Omega_\zeta}\big)_\kappa=\Div\big(\chi_{\Omega_{\zeta}}\nabla\varrho\big)_\kappa\quad \text{in}\quad I\times \setR^3.
\end{align}
We observe that this equation implies in particular, that $\partial_t\varrho_
\kappa$ is a smooth function in space. To proceed we need to use an extension operator on $\bfw$. Since $\Omega_\zeta$ is uniformly in $C^2$ there exists a continuous linear extension operator
$$\mathscr E_{\zeta}:W^{1,2}(\Omega_\zeta)\rightarrow W^{1,2}(\R^3),$$ see, for instance, \cite[Thm. 5.28]{AdFo}. Using this operator, we can reformulate \eqref{eq:2002a} by:
\begin{align}\label{eq:2002}
\partial_t\varrho_\kappa
+\Div\big(\varrho_\kappa \mathscr E_\zeta \bfw \big)= \bfr_\kappa+\varepsilon\Div\big(\chi_{\Omega_{\zeta}}\nabla\varrho\big)_\kappa\quad \text{in}\quad I\times \setR^3,
\end{align}
where $\bfr_\kappa=\Div(\varrho_\kappa \mathscr E_\zeta \bfw)-\Div(\varrho\mathscr E_\zeta\bfw)_\kappa$.
 We infer from the commutator lemma (see e.g. \cite[Lemma 2.3]{Li1}) that for a.e. $t$
\begin{align*}
\|\bfr_\kappa\|_{L^q(\R^3)}\leq \|\bfw\|_{W^{1,2}(\R^3)}\|\varrho\|_{L^{10/3}(\R^3)},\quad \tfrac{1}{q}=\tfrac{1}{2}+\tfrac{3}{10},
\end{align*}
as well as 
\begin{align}\label{eq:2002b}
\bfr_\kappa\rightarrow0\quad\text{in}\quad L^{q}(\R^3).
\end{align}
a.e. in $I$.
 Note that a) implies that
$\varrho\in L^{10/3}(I\times\Omega_\zeta)$. Now we multiply \eqref{eq:2002} by
$\theta'(\varrho_\kappa)$ and obtain
\begin{align}\label{eq:2002c}
\begin{aligned}
\partial_t \theta(\varrho_\kappa)
&+\Div\big(\theta(\varrho_\kappa)\mathscr E_\zeta \bfw \big)+\big(\varrho_\kappa\theta'(\varrho_\kappa)-\theta(\varrho_\kappa)\big)\Div\mathscr E_\zeta \bfw\\
&= \bfr_\kappa\theta'(\varrho_\kappa)+\Div\big(\varepsilon\big(\chi_{\Omega_{\zeta}}\nabla\varrho\big)_\kappa \theta'(\varrho_\kappa)\big)-
\big(\chi_{\Omega_{\zeta}}\nabla\varrho\big)_\kappa\cdot\theta''(\varrho_\kappa)\nabla\varrho_\kappa.
\end{aligned}
\end{align}
Due to the properties of the mollification and $\theta\in C^2$ we have (at least after taking a subsequence)
\begin{align*}
\theta(\varrho_\kappa)&\rightarrow \theta(\varrho)\quad \text{in}\quad L^q(I\times\R^3),\\
\theta(\varrho_\kappa)&\rightharpoonup^\ast \theta(\varrho)\quad \text{in}\quad L^\infty(\times\R^3),
\end{align*}
for all $q<\infty$.
The same is true for $\theta'(\varrho_\kappa)$ and $\theta''(\varrho_\kappa)$.
Consequently, we have
\begin{align*}
\big(\chi_{\Omega_{\zeta}}\nabla\varrho\big)_\kappa\cdot\theta''(\varrho_\kappa)\nabla\varrho_\kappa
&\rightharpoonup\chi_{\Omega_{\zeta}}\theta''(\varrho)|\nabla\varrho|^2\quad\text{in}\quad L^1(I\times\R^3)
\\
\text{ and }
\theta'(\varrho_\kappa)\nabla\varrho_\kappa&\rightharpoonup\chi_{\Omega_{\zeta}}\theta'(\varrho)\nabla\varrho\quad\text{in}\quad L^2(I\times\R^3).
\end{align*}
Hence multiplying \eqref{eq:2002c} by $\phi\in C^\infty(\overline{I}\times \R^3)$ and integrating over $I\times\R^3$ this implies
\begin{align*}
\int_{I}\partial_t&\int_{\R^3} \theta(\varrho)\,\psi\dxt-\int_{I\times\R^3}\theta(\varrho)\,\partial_t\psi\dxt
+\int_{I\times\R^3}\big(\varrho\theta'(\varrho)-\theta(\varrho)\big)\Div\mathscr E_\zeta \bfw\,\psi\dxt\\
&=\int_{I\times\R^3}\theta(\varrho) \mathscr E_\zeta \bfw\cdot\nabla\psi -\int_{I\times\R^3}\varepsilon\chi_{\Omega_{\zeta}}\nabla \theta(\varrho)\cdot\nabla\psi\dxt-\int_{I\times\R^3}\chi_{\Omega_{\zeta}}\theta''(\varrho)|\nabla\varrho|^2\psi\dxt.
\end{align*}
This proves b) since $\mathscr E_\zeta \bfw\equiv \bfw$ on $\Omega_\zeta$. 
In order to prove c). We use \eqref{eq:renormz} for $\varphi=\chi_{[0,t]}$ and $\theta=\theta_n$
where $\theta_n$ is a smooth approximation to $\theta(z)=z^-=-\min\{z,0\}$.
It is possible to define $\theta_n$ as a convex function such that
\begin{align}\label{0203a}
\theta_n\rightarrow\theta,\quad\theta_n'\rightarrow\theta',
\end{align}
pointwise as $n\rightarrow\infty$ as well as
\begin{align}\label{0203b}
 |\theta_n(z)|\leq\,c\,(1+|z|),\quad  |\theta_n'(z)|\leq\,c,
\end{align}
uniformly in $n$ and $z$.
This yields
\begin{align*}
\int_{\R^3} &\theta_n(\varrho(t))\dx-\int_{\R^3} \theta_n(\varrho_0)\dx
\\
=&-\int_0^t\int_{\R^3}\big(\varrho\theta_n'(\varrho)-\theta_n(\varrho)\big)\Div\bfw\,\dx-\int_0^t\int_{\R^3}\varepsilon\chi_{\Omega_{\zeta}}\theta_n''(\varrho)|\nabla\varrho|^2\dxt\\
\leq&-\int_0^t\int_{\R^3}\big(\varrho\theta_n'(\varrho)-\theta_n(\varrho)\big)\Div\bfw\dx.
\end{align*}
On account of \eqref{0203a} and \eqref{0203b} we can pass to the limit by dominated convergence. So, we have
\begin{align*}
\int_{\R^3} \theta(\varrho(t))\dx-\int_{\R^3} \theta(\varrho_0)\dx
\leq-\int_0^t\int_{\R^3}\big(\varrho\theta'(\varrho)-\theta(\varrho)\big)\Div\bfw\,\dx=0
\end{align*}
which implies $\varrho=0$ a.e. By the definition of $\theta$ and the nonnegative assumption
in $\varrho_0$. 
 \end{proof}

\section{The regularized system}
\label{sec:3}
The aim of this section is to prepare the existence of a weak solution to the regularized system with artificial viscosity and pressure. In order to do so we have to regularize the convective terms and the variable domain.
 We start by introducing a suitable regularization. Here and in the following we will use, whenever necessary zero-extension two the whole space for quantities which we wish to regularize via convolution without further reference.
 
\subsection{Definition of the regularized system}
We will construct a mollification of both $\zeta$ and $\bfv$. 
At first, for any
\[
\zeta\in C\Big(\overline I\times \partial\Omega;\Big[-\frac{L}2,\frac{L}2\Big]\Big),
\] 
we introduce a standard regularizer. Since we can not extend $\zeta$ to $\setR$ in time, we use convolution with half intervals. Firstly, we take $\tau^-_\kappa\in C^\infty_0((-\kappa,0],\setR_+)$ and $\tau^+_\kappa\in C^\infty_0([0,\kappa),\setR_+)$
with $\int\tau_\kappa^\pm=1$. Secondly, we take $\psi\in C^\infty([0,T],[0,1])$
such that $\psi=0$ on $[0,T/4]$, $\psi=1$ on $[3/4T,T]$
Then we define $\tau_\kappa=\psi\tau_\kappa^++(1-\psi)\tau^-_\kappa$.
Now convolute $\zeta$ with the standard mollification kernel $\phi_\kappa$ on $\partial\Omega$ and define
$
\mathscr{R}_\kappa\zeta(t,q)=(\tau_\kappa\phi_\kappa * \zeta)(t,q).
$
By classical theory we have the following properties.
\begin{lemma}\label{lem:0803}
\begin{itemize}
\item[a)] We have $\mathscr R_\kappa\zeta\in C^4(\overline I\times M)$. 
\item[b)] If $\kappa\rightarrow0$ we have $\mathscr R_\kappa\zeta\rightarrow\zeta$
uniformly.
\item[c)] If $\zeta\in L^2(I;W^{2,2}_0(M))$ then we have $\mathscr R_\kappa\zeta\to\zeta$ in  $L^2(I;W^{2,2}_0(M))$
for $\kappa\rightarrow0$.
\item[d)] If $\partial_t\zeta\in L^p(I\times M)$ we have $\partial_t\mathscr R_\kappa\zeta=\mathscr R_\kappa(\partial_t\zeta)\rightarrow\partial_t\zeta$ in  $L^p(I\times M)$
for $\kappa\rightarrow0$.
\item[e)] If $\zeta\in C^{\gamma}(I\times M)$ for some $\gamma\in(0,1)$ we have $\mathscr R_\kappa\zeta\to\zeta$ in $C^\gamma  (I\times M)$
as $\kappa\rightarrow0$.
\item[f)] $\max \abs{\mathscr R_\kappa\zeta}\leq \max\abs{\zeta}$.
\end{itemize}
\end{lemma}

On the other hand, for functions belonging to $L^2(I;L^2(\R^3))$ we define
$\psi_\kappa$ to be the standard space-time mollification kernel with parameter $\kappa$. Note that functions defined on the variable domain can be extended to the whole space by zero. To be precise we will use the definition for the regularization
\[
(\mathscr R_\kappa v)(x):=\int_{\setR^{n+1}}\psi_\kappa(t-s,y-x)\chi_{I\times \Omega_{\regkap\zeta}}(s,y)v(s,y) \dd y\ds.
\]
Since we may assume that $\psi_\kappa$ is an even function, we find that, for $u,v\in L^1_\loc(\setR^{n+1})$
\[
\int_{I\times \Omega_{\regkap\zeta}}\regkap v u\dx \dt=\int_{I\times \Omega_{\regkap\zeta}} v \regkap u\dx \dt
\]
 With no loss of generality we assume, that $\varrho_0,\bfq_0$ are defined in the whole space $\setR^3$. We also set $\bfu_0=\frac{\bfq_0}{\varrho_0}$ and assume that $\bfu_0\in L^2(\R^3)$. Finally, in accordance with \eqref{eq:compa}, we assume that
\begin{align*}
\mathrm{tr}_{\regkap\eta_0}\bfu_0=\eta_1\nu
\end{align*}
This can be achieved as done in \cite{LeRu}[p. 234, 235] (in fact, our situation is easier as we do not have to take into account the divergence-free constraint).

The aim is therefore to get a solution to the following system. We are looking for a triple $(\eta,\varrho,\bfu)$ such that
\begin{align}
\label{eq:fixpoint}
\begin{aligned}
 \partial_t\varrho+\diver(\varrho \regeta\bfu) &=\epsilon\Delta \varrho&\text{ in } &I\times\Omega_{\regkap\eta},
 \\ 
 \partial_t((\varrho+\kappa)\bu)+\diver(\varrho\regeta\bfu\otimes \bfu)&=\epsilon \Delta(\rho\bfu) +\mu\Delta\bu+(\lambda+\mu)\nabla\diver\bfu
 \\
 &\quad-\regeta\nabla ( a\varrho^{\gamma}+\delta\varrho^\beta)+\varrho\bff 
 &\text{ in } &I\times\Omega_{\regkap\eta},
 \\
  \partial_{\nu_{\regkap\eta}} \varrho(\cdot,\cdot+\regkap\eta\nu)=0&,\qquad\bfu(\cdot,\cdot+\regkap\eta\nu)=\partial_t\eta\nu
 &\text{ in } &I\times \partial\Omega
 \\
  \varrho(0)=\varrho_0&,\qquad (\varrho\bu)(0)=\bq_0
  &\text{ in }& \Omega_{\regkap\eta(0)}\\
  \partial_t^2\eta
  +K'(\eta)&= -\nu \cdot\big(-\bftau\nu\big)\circ \bfPsi_{\regkap\eta}^{-1}|\det D\bfPsi_{\regkap\eta}|&\text{ in } &I\times M,\\
\bftau&=\varepsilon\nabla(\rho\bfu)-2\mu\ep^D(\bfu)-\lambda\Div\bfu\, \mathbb I 
\\
&\quad+\mathscr R_\kappa(a\varrho^\gamma+ \delta\varrho^\beta)\mathbb I
\\
\eta=0&,\qquad \nabla \eta=0&\text{ in }&I\times \partial M
\\
  \eta(0)=\eta_0&,\qquad 
  \partial_t \eta(0)=\eta_1&\text{ in }& M.
  \end{aligned}
\end{align}
The choice of the regularization of the above system will be clear by defining the weak formulation. In fact, the weak form of the above system   can be written in two equations. Every other information will be imposed by the choice of convenient function spaces. For this reason we define the following function spaces. We set
\begin{align*}
Y^I:=W^{1,\infty}(I;L^2(M))\cap L^\infty(I;W^{2,2}_0(M))
\end{align*}
and for $\zeta\in Y^I$ with $\|\zeta\|_\infty<L$ we define
\begin{align*}
X_\zeta^I:=L^2(I;W^{1,2}(\Omega_{\zeta   (t)})).
\end{align*}
The weak concept of solution is $(\bfu,\varrho,\eta)\in X_{\regkap\eta}^I\times X_{\regkap\eta}^I\times Y^I$ that satisfies the following.
\begin{enumerate}[label={(K\arabic{*})}]
 \item\label{K1} The regularized weak momentum equation
\begin{align}\label{eq:regu}
\begin{aligned}
&\int_I\frac{\dd}{\dt}\int_{\Omega_{\regkap \eta}}(\varrho+\kappa)\bfu \cdot\bfphi\dxt-\int_I\int_{\Omega_{\regkap \eta}}\varrho\bfu \cdot\partial_t\bfphi\dx\dt\\
&-\int_I\int_{\Omega_{\regkap \eta}}\varrho\regeta\bfu\otimes \bfu:\nabla \bfphi\dx\dt+\int_I\int_{\Omega_{\regkap \eta }}\mu\nabla\bfu:\nabla\bfphi \dxt
\\&+\int_I\int_{\Omega_{\regkap \eta }}(\lambda+\mu)\Div\bfu\,\Div\bfphi\dxt-\int_I\int_{\Omega_{\regkap \eta }}
(a\varrho^\gamma+\delta\varrho^\beta)\,\Div\regeta\bfphi\dxt
\\&+\int_I\bigg(\frac{\dd}{\dt}\int_M \partial_t \eta\, b\dH-\int_M \partial_t\eta\,\partial_t b\dH + \int_M K'(\eta)\,b\dH\bigg)\dt
\\&=\int_I\int_{\Omega_{\regkap \eta }}
\varepsilon\nabla(\varrho\bfu):\nabla\bfphi\dxt+\int_I\int_{\Omega_{\regkap \eta}}\varrho\bff\cdot\bfphi\dxt+\int_I\int_M g\,b\,\dd x\dt
\end{aligned}
\end{align} 
for all test-functions $(b,\bfphi)\in C^\infty_0(M)\times C^\infty(\overline{I}\times \R^3)$ with $\mathrm{tr}_{\regkap \eta}\bfphi=b\nu$.
Moreover, we have $(\varrho\bfu)(0)=\bfq_0$, $\eta(0)=\eta_0$ and $\partial_t\eta(0)=\eta_1$. 
\item\label{K2} The regularized continuity equation 
\begin{align}\label{eq:regvarrho}
\begin{aligned}
\int_I\bigg(\frac{\dd}{\dt}\int_{ \mathscr R_\kappa\eta}\varrho \psi\dx&-\int_{ \mathscr R_\kappa\eta}\Big(\varrho\,\partial_t\psi
+\varrho\mathscr R_\kappa\bfu\cdot\nabla\psi\Big)\dx\bigg)\dt\\&
+\varepsilon\int_I\int_{\Omega_{ \mathscr R_\kappa\eta}}\nabla\varrho\cdot\nabla\psi\dxt=0
\end{aligned}
\end{align}
for all $\psi\in C^\infty\big(\overline{I}\times\R^3\big)$ and we have $\varrho(0)=\varrho_0$.
\item\label{K3}  The boundary condition $\mathrm{tr}_{\regkap\eta}\bfu=\partial_t\eta\nu$ holds in the sense of Lemma \ref{lem:2.28}
\end{enumerate}
It is a consequence of Reynold's transport theorem. For more details on the interplay of the convective term and the time derivative on the boundary we refer to the next subsection.
\subsection{Formal a priori estimates for the regularized system}
\label{ssec:formal1}
To understand the particular regularization we briefly discuss how to obtain formal a priori estimates.
By taking $\frac{\abs{\bfu}^2}{2}$ in the continuity equation and subtracting it from the momentum equation tested by the couple $(\bfu,\partial_t\eta)$ we find 
\begin{align}
\label{eq:mom}
\begin{aligned}
\int_{\Omega_{\regkap \eta}}&\Big(\frac{\varrho(t)}{2}+\kappa\Big)|\bfu(t)|^2\dx+\int_0^t\int_{\Omega_{\regkap \eta}}\big(\mu|\nabla\bfu|^2+(\lambda +\mu)\abs{\diver \bfu}^2\big)\dxs\\
&+\int_0^t\int_{\Omega_{\regkap \eta}}\varepsilon\varrho|\nabla\bfu|^2\dxt+\int_M\frac{|\partial_t\eta(t)|^2}2\dH+\frac{{K(\eta(t))}}{2}\\
&-\int_0^t\int_{\Omega_{\regkap \eta }}
(\varrho^\gamma+\delta\varrho^\beta)\,\Div\regeta\bfu\dxs
\\
&=\int_{\Omega_{\regkap \eta(0)}}\frac{|\bfq_0|^2}{2}\dx+\int_M\frac{|\eta_0|^2}2\dH+\int_M\frac{|\eta_1|^2}2\dH+\frac{{K(\eta_0)}}{2}\\
&+\int_0^t\int_{\Omega_{\regkap \eta}}\varrho\bff\cdot\bfu\dxs+\int_0^t\int_M g\,\partial_t \eta\dH\ds.
\end{aligned}
\end{align}
The right hand side of the inequality is as wanted, since all dependencies on $\eta,\bfu$ can be absorbed to the left hand side. Therefore, the only term that needs an extra treatment is the pressure term. We multiply the continuity equation by $\varrho^{\gamma-1}$ to obtain
\begin{align*}
0=&\frac{\dd}{\dt}\int_{\Omega_{\regkap\eta}}\varrho^\gamma\dx
+ \int_{\Omega_{\regkap\eta}}(\gamma-1)\varrho^\gamma \diver\regeta\bfu \dx+\epsilon\int_{\Omega_{\regkap\eta}}\gamma(\gamma-1)\varrho^{\gamma-2}\abs{\nabla \varrho}^2\dx.
\end{align*} 
Repeating, the above for $\theta(\varrho)=\varrho^\beta$, we can estimate the pressure term in \eqref{eq:mom} accordingly and deduce the following a priori estimate
\begin{align}
\label{apri:reg}
\begin{aligned}
\sup_{t\in I}\int_{\Omega_{\regkap \eta}}&(\varrho+\kappa)|\bfu|^2\dx+\sup_{t\in I}\int_{\Omega_{\regkap \eta}}\big(a\varrho^\gamma+\delta\varrho^\beta\big)\dx+\int_I\int_{\Omega_{\regkap \eta}}|\nabla\bfu|^2\dxt\\
&+\varepsilon\int_I\int_{\Omega_{\regkap \eta}}\varrho^{\gamma-2}|\nabla\varrho|^2+\varrho\abs{\nabla \bfu}^2\dxt+\sup_{t\in I}\int_M|\partial_t\eta|^2\,\dd\mathcal H^2+\sup_{t\in I}K(\eta)\\
&\leq\,c\,\bigg(\int_{\Omega_{\regkap \eta(0)}}\frac{|\bfq_0|^2}{\varrho_0}\dx+\int_{\Omega_{\regkap \eta(0)}}\big(\varrho_0^\gamma    
+\delta\varrho_0^\beta\big)\dx+\int_I\|\bff\|_{L^2(\Omega_{\regkap \eta})}^2\dt\bigg)
\\
&\quad+c\bigg(\int_M\abs{\eta_0}^2\dH +\int_M\abs{\eta_1}^2\dH+K(\eta_0)+\int_I\norm{g}_{L^2(M)}^2\dt\bigg),
\end{aligned}
\end{align}
with a constant $c$ that is independent of $\kappa,\delta,\varepsilon$. 
The rest of this section is now dedicated to prove the following existence theorem.
\begin{theorem}\label{thm:regu}
Suppose that $\eta_0,\eta_1,\varrho_0,\bfq_0,\bff$ and $g$ are regular enough to give sense to the right-hand side of \eqref{apri:reg}, that $\varrho_0\geq0$ a.e. and \eqref{eq:compa} is satisfied. Then there is a weak solution $(\eta,\bfu,\varrho)\in Y^I\times X_{\regkap\eta}^I\times X_{\regkap\eta}^I$, that satisfies \ref{K1}--\ref{K3} as well as \eqref{apri:reg}. Here, we have $I=(0,T_*)$, where
$T_*<T$ only if $\lim_{t\rightarrow T^\ast}\|\eta(t,\cdot)\|_{L^\infty(\partial\Omega)}=\frac{L}{2}$.
\end{theorem}
 
\subsection{Definition of the decoupled system}
\label{subsec:dec}
The strategy is to first construct a weak solution to a decoupled system. Let us consider a given deformation $\zeta\in C(\overline I\times M)$ and a given function $\bfv\in L^2(I;\R^3)$. We will decouple \eqref{eq:fixpoint}, by replacing there $\regkap\eta$, with $\regkap \zeta$ and $\regkap \bfu$ by $\regkap\bfv$.
Firstly, we find from Theorem~\ref{lem:warme}, that there exists a unique $\varrho\in X_{\regkap\zeta}$ that satisfies the following
\begin{align}\label{eq:regrhodc}
\begin{aligned}
\int_I\frac{\dd}{\dt}\int_{\Omega_{\regkap\zeta}}\varrho\, \psi\dx&-\int_{\Omega_{\regkap\zeta}}\Big(\varrho\,\partial_t\psi\dxt
+\varrho\regzet\bfv\cdot\nabla\psi\Big)\dxt\\&+\varepsilon\int_I\int_{\Omega_{\regkap \zeta}}\nabla\varrho\cdot\nabla\psi\dxt=0
\end{aligned}
\end{align}
for all $\psi\in C^\infty\big(\overline{I}\times\R^3\big)$. Observe, that $\varrho$ exists independently of $\bfu,\eta$.

Secondly, we repeat the interplay of the boundary deformation with the convective term  for the momentum equation and find that smooth functions satisfy
\begin{align}
\label{pi2}
\begin{aligned}
&\int_{\Omega_{\regkap\zeta}}\Big(\partial_t ((\varrho+\kappa) \bfu)\cdot \bfphi +\divergence (\varrho \regzet\bfv\otimes \bfu)\cdot\bfphi\Big)\dx
\\&
=\int_{\Omega_{\regkap\zeta}}\Big(\partial_t ((\varrho+\kappa) \bfu\cdot \bfphi) 
+\divergence (\varrho \regzet\bfv\otimes \bfu\, \bfphi)\Big)\dx\\&-\int_{\Omega_{\regkap \zeta}}\Big((\varrho+\kappa)\bfu \cdot\partial_t\bfphi +\varrho \regzet\bfv\otimes \bfu:\nabla \bfphi\Big)\dx\\
&=\frac{\dd}{\dt}\int_{\Omega_{\regkap \zeta}}\varrho \bfu \cdot\bfphi\dx-\int_{\Omega_{\regkap \zeta}}\Big(\varrho\bfu\cdot \partial_t\bfphi +\varrho\regzet\bfv\otimes \bfu:\nabla \bfphi\Big)\dx\\&+\int_{\partial\Omega_{\regkap \zeta}}\varrho\bfu\cdot\bfphi\Big(\regzet\bfv\cdot\nu_{\regkap\zeta}
-(\partial_t\regkap\zeta\nu)\circ\bfPsi_{\regkap\zeta}^{-1}\Big)\dH-\int_{\Omega_{\regkap \zeta}}\kappa\bfu \cdot\partial_t\bfphi\dxt
\end{aligned}
\end{align}
Observe that in case of a fixed point $\bfu\equiv \bfv,\, \eta\equiv \zeta$, we find that 
\begin{align*}
\Big(\regeta\bfu-(\partial_t\regkap\eta\nu)\circ\bfPsi_{\regkap\eta}^{-1}\Big)\cdot\nu_{\regkap\eta}\equiv 0\text{ on }\partial\Omega_{\regkap\eta},
\end{align*}
which implies that the boundary integrals will vanish. For this reason, we will  solve the decoupled momentum equation with  boundary values of $\bfu$, which are implicitly defined by removing the first boundary term (this is analogous to the Neumann boundary data of the decoupled continuity equation, see Section~\ref{subsec:continuity}). for the same reason we neglect the very last integral.
Concerning the other terms of the momentum equation, when adapting partial integration we get force terms acting on the boundary in normal direction (pressure, diffusion, exterior forces). These are then assumed to be equalized by the elastic forces of the shell. Observe here, that $\bftau$ is identical for the decoupled system and the coupled system.
 
All together we require from $(\bfu,\eta)\in X_{\regkap\zeta}^I\times Y^I$ that it satisfies the following.
\begin{enumerate}[label={(N\arabic{*})}]
\item\label{N1} The regularized decoupled momentum equation
\begin{align}\label{eq:regudc}
\begin{aligned}
&\int_I\frac{\dd}{\dt}\int_{\Omega_{\regkap \zeta}}(\varrho+\kappa) \bfu \cdot\bfphi\dxt
-\int_I\int_{\Omega_{\regkap \zeta}} \Big(\varrho\bfu\cdot \partial_t\bfphi +\varrho\regzet\bfv\otimes \bfu:\nabla \bfphi\Big)\dxt
\\
&+\int_I\int_{\Omega_{\regkap \zeta }}\Big(\mu\nabla\bfu:\nabla\bfphi +(\lambda+\mu)\Div\bfu\,\Div\bfphi\Big)\dxt
\\
&-\int_I\int_{\Omega_{\regkap \zeta }}
\Big((\varrho^\gamma+\delta\varrho^\beta)\,\Div\regzet\bfphi+\varepsilon\nabla(\varrho\bfu):\nabla\bfphi\Big)\dxt
\\
&+\int_I\bigg(\frac{\dd}{\dt}\int_M \partial_t \eta b\dH-\int_M \partial_t\eta\,\partial_t b\dH + \int_M K'(\eta)\,b\dH\bigg)\dt
\\&=\int_I\int_{\Omega_{\regkap \zeta}}\varrho\bff\cdot\bfphi\dxt+\int_I\int_M g\,b\,\dd x\dt
\end{aligned}
\end{align} 
 holds for all test-functions  $(b,\bfphi)\in C^\infty_0(M)\times C^\infty(\overline{I}\times \R^3)$ with $\mathrm{tr}_{\Omega}(\bfphi\circ \bfPsi_{\regkap \zeta})=b\nu$. Moreover, we have $(\varrho\bfu)(0)=\bfq_0$, $\eta(0)=\eta_0$ and $\partial_t\eta(0)=\eta_1$. 
\item\label{N2} The decoupled regularized continuity equation \eqref{eq:regrhodc} is satisfied with initial datum $\varrho(0)=\varrho_0$.
\item\label{N3}  The boundary condition $\mathrm{tr}_{\mathscr R_\kappa\zeta}\bfu=\partial_t\eta\nu$ holds in the sense of Lemma \ref{lem:2.28}
\end{enumerate}
The a priori estimates are formally available as before for the regularized system in Section \ref{ssec:formal1}: First, one uses $\bfu,\partial_t\eta$ as test-function in the momentum equation and subtract the continuity equation tested with $\frac{\abs{\bfu}^2}{2}$. Secondly, one uses the renormalized formulation~\eqref{eq:renorm} to estimate the pressure term.

\begin{theorem}\label{thm:decu}
For any $\zeta\in C(\overline I\times M;[-\frac{L}2,\frac{L}{2}]))$ and $\bfv\in L^2(I;L^{2}(\R^3))$ there exists a solution $(\eta,\bfu)\in Y^I\times X_{\regkap\zeta}^I$ to \ref{N1}--\ref{N3}. Here, we have $I=(0,T_*)$, where $T_*<T$ only if $\lim_{t\rightarrow T^\ast}\|\eta(t,\cdot)\|_{L^\infty(M)}=\frac{L}{2}$.
The solution satisfies the energy estimate
\begin{align*}
&\int_{\Omega_{\regkap \zeta}}\Big( \frac{\varrho(t)}{2}+\kappa\Big)|\bfu(t)|^2\dx+\int_0^t\int_{\Omega_{\regkap \zeta}}\bigg((\mu+\varepsilon\varrho)|\nabla\bfu|^2+(\lambda+\mu)\abs{\diver \bfu}^2\bigg)\dxs
\\ 
&+\int_{\Omega_{\regkap \zeta}}\Big(\frac{a}{\gamma-1}\varrho^\gamma(t)+\frac{\delta}{\beta-1}\varrho^\beta(t)\Big)\dx+\int_M\frac{|\partial_t\eta(t)|^2}2\,\dd\mathcal H^2+\frac{K(\eta(t))}{2}\\
&\leq\int_I\int_{\Omega_{\regkap\zeta}}\rho\bff\cdot\bfu\dxs+\int_I\int_M g\partial_t \eta\,\dH\ds+\int_{\Omega_{\regkap \zeta(0)}}\frac{|\bfq_0|^2}{\varrho_0}\dx
\\&+\int_{\Omega_{\regkap \zeta}(0)}\Big(\frac{a}{\gamma-1}\varrho_0^\gamma+\frac{\delta}{\beta-1}\varrho^\beta_0\Big)\dx+\frac{K(\eta_0)}{2}+\int_M\frac{|\eta_0|^2}2\,\dd\mathcal H^2+\int_M\frac{|\eta_1|^2}2\,\dd\mathcal H^2
\end{align*}
 for all $t\in[0,T^*]$,
provided that $\eta_0,\eta_1,\varrho_0,\bfq_0,\bff$ and $g$ are regular enough to give sense to the right-hand side, that $\varrho_0\geq0$ a.e and \eqref{eq:compa} is satisfied. Here the constant $c$ is independent of all involved quantities. In particular, it is independent of $\bfu$ and $\zeta$.
\end{theorem}
\begin{proof}
In order to find a solution to \eqref{eq:regrhodc}--\eqref{eq:regudc} we discretise
the system.
It is standard to find a smooth orthonormal basis $(\tilde\bfX_k)_{k\in\N}$
of $W^{1,2}_0(\Omega)$ and a smooth orthonormal basis $(\tilde Y_k)_{k\in\N}$ of $W^{2,2}_0(M)$. We define vector fields $\tilde\bfY_k$ by solving the homogeneous Laplace equation on $\Omega$ with boundary datum $\tilde Y_k\nu$ (which is extended by zero to $\partial\Omega$). Note that standard results on the inverse Laplace operator guarantee that $\tilde\bfY_k$ is smooth. Now define pointwise in $t$
\begin{align*}
\bfX_k:=\tilde\bfX_k\circ \bfPsi_{\mathscr R_\kappa\zeta}^{-1},\quad \bfY_k:=\tilde\bfY_k\circ\bfPsi_{\mathscr R_\kappa\zeta}^{-1}.
\end{align*}
By Lemma \ref{lem:diffeo} we still know that $\bfX_k$ and $\bfY_k$ belong to the class $C^3(\overline{\Omega}_{\mathscr R_\kappa\zeta}(t))$.
Obviously, $(\bfX_k)_{k\in\N}$ forms a basis of $W^{1,2}_0(\Omega_{\regkap \zeta}(t)).$
Now we choose an enumeration $(\bfomega_k)_{k\in\N}$ of $\bfX_k\oplus\bfY_k$. In return we associate
$w_k:=\bfomega_k\circ\bfPsi_{\regkap\zeta}|_{\partial\Omega_{\regkap\zeta}}\cdot \nu$. Analogous to the arguments
in \cite[p. 237]{LeRu} we find that
\begin{align*}
Z:=\mathrm{span}\Big\{(\varphi w_k,\varphi\bfomega_k)\,|\varphi\in C^1(I),\,k\in\N\Big\}
\end{align*}
is dense in the solution space
\begin{align*}
Z_{\regkap\zeta}:=\Big\{(\xi,\bfvarphi)\in Y^I\times X^I_{\regkap\zeta}:\,\,\partial_t \xi\nu_{\regkap\zeta}=\mathrm{tr}_{{\regkap\zeta}}\bfvarphi\Big\},
\end{align*}
and in the space of test-functions
\begin{align*}
Z_{\regkap\zeta}^*:=\Big\{(\xi,\bfvarphi)\in C(\overline I;W^{2,2}_0(M))\times L^2(I,W^{1,2}(\Omega_{\regkap\zeta}))\cap C(\overline I;L^2(\Omega_{\regkap\zeta})):\,\,\xi\nu=\mathrm{tr}_{{\regkap \zeta}}\bfvarphi\Big\}.
\end{align*}
Now we can begin with the construction of the solution. 
First, we fix $\varrho=\varrho(\regkap\zeta,\regkap\bfv)$ as the unique solution from Theorem~\ref{lem:warme}, where $\zeta\equiv \regkap \zeta$ and $\bfw\equiv \regkap\bfv$ subject to the initial datum $\varrho_0$.
Next we seek for a  couple of discrete solutions $(\eta^N,\bfu^N)\in \bfZ_{\regkap\zeta}$ of the form
\begin{align*}
\eta_N=\eta_0+\sum\nolimits_{k=1}^N\int_0^t\alpha_{kN} w_k\ds
,\quad \bfu_N=\sum\nolimits_{k=1}^N\alpha_{kN} \bfomega_k,
\end{align*}
which solve the following discrete version of \eqref{eq:regudc}:
\begin{align}\label{eq:decuN}
\begin{aligned}
&\int_{\Omega_{\regkap \zeta}}(\varrho(t)+\kappa)  \bfu_N(t)\cdot \bfomega_k(t)\dx
\\
&-\int_0^t\int_{\Omega_{\regkap \zeta}}\Big( \varrho\bfu_N\cdot \partial_t\bfomega_k +\varrho\regzet\bfv\otimes \bfu^N:\nabla \bfomega_k\Big)\dxt
\\
&+\int_0^t\int_{\Omega_{\regkap \zeta }}\Big(\mu\nabla\bfu_N:\nabla\bfomega_k +(\lambda+\mu)\Div\bfu_N\,\Div\bfomega_k\Big)\dx\ds
\\
&-\int_0^t\int_{\Omega_{\regkap \zeta }}\Big(
(\varrho^\gamma+\delta\varrho^\beta)\,\Div\regzet\bfomega_k+\varepsilon\nabla(\varrho\bfu_N):\nabla\bfomega_k\Big)\dx\ds
\\
&+\int_0^t\int_M \Big(K'(\eta_N)w_k-\partial_t\eta_N\,\partial_t w_k\Big)\dH\ds+\int_M \partial_t \eta_N(t) w_k(t)\dH
\\&=\int_0^t\int_{\Omega_{\regkap \zeta}}\varrho\bff\cdot\bfomega_k\dxt+\int_I\int_M g\,w_k\,\dd x\dt\\
&+\int_{\Omega_{\regkap \zeta(0)}}\bfq_0\cdot\bfomega_k(0,\cdot)\dx+\int_M\eta_1\,w_k\,\dd\mathcal H^2
\end{aligned}
\end{align} 
We can choose $\alpha_{kN}(0)$ in a way that
$\bfu_N(0)$ converges to $\bfq_0/\varrho_0$.\\
 The system \eqref{eq:decuN} is equivalent to
a system of integro-differential equations for the vector $\bfalpha_N=(\alpha_{kN})_{k=1}^N$. It reads as
\begin{align}\label{eq:ide1}
\mathcal A(t)\bfalpha_N(t)&=\int_0^t\mathcal B(\sigma)\bfalpha_N(\sigma)\ds+\int_0^t\int_0^\sigma \tilde{\mathcal B}(s,\sigma)\bfalpha_N(s)\dd s\ds+\int_0^t\bfc(\sigma)\ds+\tilde{\bfc},
\end{align}
with
\begin{align*}
\mathcal A_{ij}&=\int_{\Omega_{\regkap \zeta}}(\varrho(t)+\kappa)  \bfomega_i(t) \cdot\bfomega_j(t)\dx
+\int_M w_i(t) w_j(t)\dH
\\
\mathcal{B}_{ij}
&=\int_{\Omega_{\regkap \zeta}}\Big( \varrho\bfomega_i\cdot \partial_t\bfomega_j +\varrho\regzet\bfv\otimes \bfomega_i:\nabla \bfomega_j\Big)\dx
\\
&+\int_{\Omega_{\regkap \zeta }}\Big(\mu\nabla\bfomega_i:\nabla\bfomega_j +(\lambda+\mu)\Div\bfomega_i\,\Div\bfomega_j\Big)\dx
\\
&-\int_{\Omega_{\regkap \zeta }}
\varepsilon(\nabla\varrho\otimes\bfomega_i+\varrho \nabla \bfomega_i):\nabla\bfomega_j\dx\ds
-\int_{\partial\Omega_{\regkap \zeta }}w_i\,\partial_t w_j\dH
\\
\tilde{\mathcal B}_{ij}&=\int_M K'(w_i(s))w_j(\sigma))\dH,
\\
c_i&=\int_{\Omega_{\regkap \zeta}}(\varrho^\gamma+\delta\varrho^\beta)\,\Div\regzet\bfomega_i\dx+\int_{\Omega_{\regkap \zeta}}\varrho\bff\cdot\bfomega_i\dxt+\int_M g\,w_i\,\dd\mathcal H^2
\\
\tilde{c}_i&=\int_{\Omega_{\regkap \zeta(0)}}\bfq_0\cdot\bfomega_i(0,\cdot)\dx+\int_M\eta_1\,w_i\,\dd\mathcal H^2.
\end{align*}
As $(\varrho+\kappa)$ is strictly positive (recall Theorem \ref{lem:warme}) and the $\bfomega_k$ and $w_k$ from a basis the matrix $\mathcal A$ is bounded (by the integrability of $\varrho$) and positive definite (due to $\kappa>0$). Hence the inverse $\mathcal A^{-1}$ exists and is bounded as well. We find a continuous solution $\bfalpha_N$ to \eqref{eq:ide1}
by standard arguments for ordinary integro-differential equations. Since we wish to use it as a testfunction in the momentum equation we have to show, that $\partial_t\bfalpha_N\in L^2(I)$, for some $s>1$. The difficulty here is that $\varrho$ is not weakly differentiabel in time. This has to be circumvented. First observe that by Leipnitz rule, we find that
 \[
\partial_t \bfalpha=\mathcal A^{-1}\Big(\partial_t(\mathcal A \bfalpha)-\partial_t\mathcal A \bfalpha\Big).
 \]
Due to \eqref{eq:ide1} and the integrability of $\varrho$ from Theorem~\ref{lem:warme}
we have $\partial_t(\mathcal A \bfalpha_N)\in L^\infty(I)$.
Moreover, $\mathcal A^{-1}$ is uniformly bounded (due to $\kappa>0$). Consequently it suffices to prove that $\partial_t\mathcal A_{i,j}\in L^{2}(I)$ to conclude differentiability of $\alpha_N$. 
By taking the test function $\bfomega_i\cdot \bfomega_j$ in \eqref{eq:regrhodc} we find that
\begin{align*}
\partial_t \bigg(\mathcal A_{i,j}&-\int_M w_iw_j\dH\bigg)=\frac{\dd}{\dt}\int_{\Omega_{\regkap\zeta}}\varrho\, \bfomega_i\cdot \bfomega_j\dx+\frac{\dd}{\dt}\int_{\Omega_{\regkap\zeta}}\kappa\, \bfomega_i\cdot \bfomega_j\dx
\\
&=\int_{\Omega_{\regkap\zeta}}\varrho\,\Big(\partial_t(\bfomega_i\cdot \bfomega_j)
+\regzet\bfv\cdot\nabla(\bfomega_i\cdot \bfomega_j)+\varepsilon\nabla\varrho\cdot\nabla(\bfomega_i\cdot \bfomega_j)\dx\\
&\quad +\kappa\int_{\Omega_{\regkap\zeta}}\big(\partial_t\bfomega_i\cdot \bfomega_j+\bfomega_i\cdot \partial_t\bfomega_j\big)\dx+\kappa\int_{\partial\Omega_{\regkap\zeta}}\partial_t\regkap\zeta\bfomega_i\cdot \bfomega_j\dx.
\end{align*}
Since the right hand side is in $L^2(I)$ (note that the $\bfomega_i$ are smooth also in time) we find that $\partial_t \alpha_N\in L^s(I)$ and hence $\partial_t \bfu_N\in L^s(I\times \Omega_{\regkap\zeta})$.
The a priori estimates are now achieved by differentiating \eqref{eq:decuN} in time, testing with $(\partial_t\eta_N,\bfu_N)$
and subtracting \eqref{eq:regrhodc} tested by $\frac{1}{2}|\bfu_N|^2$. The terms with time derivative and the convective terms cancel and we obtain
\begin{align}
\label{eq:future}
\begin{aligned}
&\int_{\Omega_{\regkap \zeta}}\Big(\frac{\varrho(t)}{2}+\kappa\Big)|\bfu_N(t)|^2\dx+\int_M\frac{|\partial_t\eta_N(t)|^2}2\,\dd\mathcal H^2+\frac{{K(\eta_N(t))}}{2}
\\ 
&\quad +\int_0^t\int_{\Omega_{\regkap \zeta}}\Big((\mu+\varepsilon\varrho)|\nabla\bfu_N|^2+(\lambda+\mu)\abs{\diver \bfu_N}^2\Big)\dxs\\
&=\int_I\int_{\Omega_{\regkap\zeta}}\rho\bff\cdot\bfu_N\dxs+\int_I\int_M g\partial_t \eta_N\,\dH\ds
+\int_{\Omega_{\regkap \zeta(0)}}{|\bfq_0|^2}\dx+\int_M\frac{|\eta_0|^2}2\,\dd\mathcal H^2\\&\quad+\int_M\frac{|\eta_1|^2}2\,\dd\mathcal H^2+\frac{{K(\eta_0)}}{2}
 +\int_0^t\int_{\Omega_{\regkap \zeta}}(a\varrho^\gamma+\delta\varrho^\beta)\,\Div\regzet\bfu_N\dxs.
\end{aligned}
\end{align}
Finally, we use Theorem~\ref{lem:warme} b) in order to rewrite the last integral.
Choosing $\varphi=\chi_{[0,t]}$ yields
\begin{align*}
\int_0^t\int_{\Omega_{\regkap \zeta}}&\big(\varrho\theta'(\varrho)-\theta(\varrho)\big)\Div\regzet\bfu_N\dxt
\\\leq&\int_{\Omega_{\regkap \zeta}(0)} \theta(\varrho(0))\dxt -\int_{\Omega_{\regkap \zeta}} \theta(\varrho(t))\dxt
\end{align*}
for any convex $\theta\in C^2(\setR_+;\setR_+)$ be such that $\theta''(s)\equiv 0$ for large values of $s$ and $\theta(0)=0$. We approximate the function
$s\mapsto \frac{as^\gamma}{\gamma-1}+\frac{\delta s^\beta}{\beta-1}$ by a sequence of such functions and obtain
\begin{align*}
\int_0^t&\int_{\Omega_{\regkap \zeta}}(a\varrho^\gamma+\delta\varrho^\beta)\,\Div\regzet\bfu_N\dxs\\&\leq\int_{\Omega_{\regkap \zeta}(0)}\Big(\frac{a}{\gamma-1}\varrho_0^\gamma+\frac{\delta}{\beta-1}\varrho^\beta_0\Big)\dx
-\int_{\Omega_{\regkap \zeta}}\Big(\frac{a}{\gamma-1}\varrho_0^\gamma(t)+\frac{\delta}{\beta-1}\varrho^\beta(t)\Big)\dx
\end{align*}
By Young's inequality we can absorb the terms that depend on $\bfu_N$ or $\partial_t\eta_N$ in the left hand side of \eqref{eq:future}, such that
\begin{align}
\label{apri:regdc}
\begin{aligned}
&\sup_{I}\int_{\Omega_{\regkap \zeta}}(\varrho+\kappa)\frac{|\bfu_N|^2}{2}\dx+\sup_I\int_{\Omega_{\regkap \zeta}}\Big(\frac{a}{\gamma-1}\varrho^\gamma+\frac{\delta}{\beta-1}\varrho^\beta\Big)\dx\\&+\int_I\int_{\Omega_{\regkap \zeta}}|\nabla\bfu_N|^2\dxs+\int_I\int_{\Omega_{\regkap \zeta}}(\mu+\varepsilon\varrho)|\nabla\bfu_N|^2)\dxs\\&+\int_I\int_{\Omega_{\regkap \zeta}}(\lambda+\mu)\abs{\diver \bfu_N}^2\dx+\sup_{I}\int_M|\partial_t\eta_N|^2\,\dd\mathcal H^2+\sup_I\frac{K(\eta_N)}{2}\\
&\leq\,c\bigg(\int_I\int_{\Omega_{\regkap\zeta}}\abs{\bff}^2\dxs+\int_I\int_M \abs{g}^2\,\dH\ds+\int_{\Omega_{\regkap \zeta(0)}}\frac{|\bfq_0|^2}{\varrho_0}\dx+\int_M|\eta_0|^2\,\dd\mathcal H^2\bigg)
\\&\quad+c\,\bigg(\int_M|\eta_1|^2\,\dd\mathcal H^2+K(\eta_0)+\int_{\Omega_{\regkap \zeta}(0)}\Big(\frac{a}{\gamma-1}\varrho_0^\gamma+\frac{\delta}{\beta-1}\varrho^\beta_0\Big)\dx\bigg).
\end{aligned}
\end{align}
This implies that there is a subsequence, such that
\begin{align*}
\eta_N&\rightharpoonup^\ast \eta\quad\text{in}\quad Y^I,\quad
\bfu_N\rightharpoonup \bfu\quad\text{in}\quad X^I_{\regkap\zeta},
\end{align*}
for some limit function $(\eta,\bfu)$.
As \eqref{eq:decuN} is linear in $(\eta_N,\bfu_N)$ we can pass to the limit
and see that $(\eta,\bfu)$ solves \eqref{eq:regudc}. 
\end{proof}

\subsection{A fixed point argument.}
Now we are seeking for a fixed point of the solutions map $(\bfv,\zeta)\mapsto (\bfu,\eta)$ on $L^2(I,L^2(\setR^3))\times C(\overline I\times\partial\Omega)$ from Theorem \ref{thm:decu}. 
As we do not know about uniqueness of the solutions constructed in Theorem~\ref{thm:decu} we will use the following version of set-valued mappings
\begin{theorem}[\cite{GD}]\label{lem:fix}
Let $C$ be a convex subset of a normed
vector space $Z$ and let $F:C\rightarrow \mathfrak P (C)$ be an upper-semicontinuous set-valued
mapping, that is, for every open set $W\subset C$ the set $\{c\in C:\,\, F(c)\in W\}\subset C$ is
open. Moreover, let $F(C)$ be contained in a compact subset of $C$, and let $F(z)$ be
non-empty, convex, and compact for all $c\in C$. Then F possesses a fixed point, that
is, there exists a $c_0\in C$ with $c_0\in F(c_0)$.
\end{theorem}

\subsection{Proof of theorem~\ref{thm:regu}}
We will prove Theorem \ref{thm:regu} by finding a fixed point of a suitable mapping defined below.
We denote $I_*=[0,T_*]$ with $T_*$ sufficiently small. We do not know about the uniqueness of solutions. Hence we apply Theorem~\ref{lem:fix} to get a fixed point. We consider the sets
\begin{align*}
D:=\Big\{(\zeta,\bfv)\in C(\overline{I}_*\times \partial\Omega)\times L^2(I_*,L^2(\R^3)):\,\,\zeta(0)=\eta_0,\,\,\|\zeta\|_{L^\infty}\leq M,\,\,\|\bfv\|_{L^2(I_\ast;L^2(\R^3))}\leq K\Big\}
\end{align*}
for $M=(\|\eta_0\|_\infty+L)/2$ and $K>0$ to be chosen later.
\begin{align*}
F:D\rightarrow \mathfrak P(D).
\end{align*}
with 
\[
F:(\bfv, \zeta)\mapsto \Big\{(\bfu,\eta):\,(\bfu,\eta)\text{ solves }\eqref{eq:regudc}\text{ with $(\bfv, \zeta)$ and satisfies the energy bounds}\Big\}.
\]
Note that we extend $\bfu$ and $\eta$ by zero to $\R^3$ and $\partial\Omega$
respectively.
First we have to check that $F(D)\subset D$. We will use the a priori estimate from Theorem \ref{thm:decu} to conclude
 \begin{align*}
&\sup_{I_\ast}\int_{\Omega_{\regkap \zeta}}(\varrho+\kappa)\frac{|\bfu|^2}{2}\dx+\sup_{I_\ast}\int_{\Omega_{\regkap \zeta}}\Big(a\varrho^\gamma+\delta\varrho^\beta\Big)\dx\\&+\int_{I_\ast}\int_{\Omega_{\regkap \zeta}}|\nabla\bfu|^2\dxs+\sup_{I_\ast}\int_M\Big(|\partial_t\eta|^2+\abs{\nabla^2\eta}^2\Big)\,\dd\mathcal H^2\\
&\leq\,c(\bff,g,\bfq_0,\eta_0,\eta_1,\varrho_0)
\end{align*}
independently of $L$, $K$ and the size of $I_\ast$. This implies that $\eta\in C^\alpha(\overline I\times M)$, by Sobolev embedding for some $\alpha>0$, with H\"older norm independently of $L,K$. This implies that
\begin{align}
\label{constarint-eta}
\abs{\eta(t,x)}\leq \abs{\eta(t,x)-\eta_0(0,x)}+\abs{\eta_0(0,x)}\leq c (T^*)^\alpha+\|\eta_0\|_\infty\leq M.
\end{align}
Therefore, we find for $T^*$ small enough but independent of $\bfv, \zeta$, that
\[
\norm{\eta}_{L^\infty(I_*\times \partial\Omega)}\leq M.
\]
Hence we gain $F(D)\subset D$ for an appropriate choice of $K\in \setR_+$.

Next, since the problem is linear and the left-hand side of the energy inequality is convex, we find that $F(\zeta,\bfv)$ is a convex and closed subset of $\bfZ$. It remains to show that $F(D)$ is relatively compact.
 Consider $(\bfu_n,\eta_n)_n\subset F(D)$. Then there exists a corresponding sequence $(\bfv_n,\zeta_n)_n\subset D$, such that $(\eta_n,\bfu_n)$ solve \eqref{eq:regudc}, with respect to $(\bfv_n,\zeta_n)$. 
We may choose subsequences such that
\begin{align}\label{conv1}
\eta_n&\rightharpoonup^\ast\eta\quad\text{in}\quad L^\infty(I_*,W^{1,2}_0(M)),\\
\label{conv2}\partial_t\eta_n&\rightharpoonup^\ast\partial_t\eta\quad\text{in}\quad L^\infty(I_*,L^{2}(M)),\\
\label{conv3orig}\bfu_n&\rightharpoonup^{\ast,\eta}\bfu\quad\text{in}\quad L^\infty(I_*;L^2(\Omega_{\regkap\zeta})),\\
\label{conv4orig}\nabla\bfu_n&\rightharpoonup^\eta\nabla\bfu\quad\text{in}\quad L^2(I_*;L^2(\Omega_{\regkap\zeta}))).
\end{align}
Note also that we can extend $\bfu_n$ and $\bfu$ by zero to the whole space
and gain
\begin{align}
\label{conv3origneu}\bfu_n&\rightharpoonup^{\ast}\bfu\quad\text{in}\quad L^\infty(I_*;L^2(\R^3)).
\end{align}
The compactness of $\eta_n$ in $C(\overline I_*\times \partial\Omega)$ follows immediately by Arcela-Ascoli, since we know that $\eta_n$ is uniformly H\"older continuous. To show the compactness of the $\bfu_n$ is much more sophisticated.
We first need to show compactness of $\varrho_n$,
 where $\varrho_n$ is the unique solution to \eqref{eq:regrhodc} with $\bfv=\bfv_n$. A direct application
of Theorem~\ref{lem:warme} a) shows
\begin{align}\label{eq:1306a}
\begin{aligned} \varrho_n\weaktoregzetan\varrho\quad\text{in}\quad L^2(I_*;W^{1,2}(\Omega_{\regkap\zeta_n})),\\
\varrho_n\weakto^{\ast,\eta}\varrho\quad\text{in}\quad L^\infty(I_*;L^{2}(\Omega_{\regkap\zeta_n})),
\end{aligned}
\end{align}
at least after taking a subsequence. Moreover, \eqref{eq:1306a} holds uniformly in $n$.
Firstly, we find for all $k,l\in\N$ that $\norm{\partial_t^l\nabla^k\regkap\zeta_n}_{L^\infty(I\times\partial\Omega)}\leq c(\kappa,k,l)$. Hence, there is a (not relabeled) subsequence
 \begin{align}\label{eq:2410}
\regkap\zeta_n\to \regkap\zeta \quad\text{in}\quad C^2(I_*\times\partial\Omega).
\end{align}
Next, we claim that 
\begin{align}\label{0103a}
\varrho_n\toregzetan\varrho\quad\text{in}\quad L^q(I_*;L^q\Omega_{\mathscr R_\kappa\zeta_n}).
\end{align}
for any $q<\frac{10}{3}$.
In fact, the assumptions of Lemma \ref{thm:weakstrong} are satisfied
due to \eqref{eq:regrhodc}. In particular \ref{A3} holds
with $\bfH_n^1=0$, $\bfH^2_n=\mathscr R_\kappa\bfv_n+\varepsilon\nabla\varrho_n$ and $h_n=0$.
Due to the uniform bounds on $\varrho_n$ in \eqref{eq:1306a} and the bounds on $\bfv_n$ we gain strong convergence in $L^2$ by Remark~\ref{rem:strong} at least for a subsequence. Combining this with \eqref{eq:1306a} proves \eqref{0103a}.

Fourthly, again by Lemma~\ref{thm:weakstrong} we find for the couple $(\kappa+\varrho_n)\bfu_n$ and $\bfu_n$, that
\begin{align}\label{eq:0103c}
(\kappa+\varrho_n)\abs{\bfu_n}^2\weakto (\kappa+\varrho)\abs{\bfu}^2\text{ in }L^s(I_*\times \Omega_{\regkap\zeta_n}),
\end{align}
for any $s<\frac{10}{3}$. To be precise we infer from \eqref{eq:regudc} that
 \begin{align*}
\partial_t((\varrho_n+\kappa)\bu_n)&=-\diver(\varrho_n\mathscr R_\kappa\bfv_n\otimes \bfu_n)+\epsilon \Delta(\varrho_n\bfu_n) +\mu\Delta\bu_n+(\lambda+\mu)\nabla\diver\bfu_n\\&-\mathscr R_\kappa\nabla (\varrho^{\gamma}_n+\delta\varrho_n^\beta)+\varrho_n\bff 
\end{align*}
holds in the interior in the sense of distributions. In particular, \ref{A3} is satisfied with
\begin{align*}
\bfH_n^1&=\varepsilon\varrho_n\bfu_n+\mathcal R\bfu_n,\quad
\bfH_n^2=-\varrho_n\mathscr R_\kappa\bfv_n\otimes \bfu_n-\nabla\mathscr R_\kappa (\varrho^{\gamma}_n+\delta\varrho_n^\beta)I,\quad\bfh_n=\varrho_n\bff,
\end{align*}
choosing $p=s=2$ and $m$ arbitrary and $b\in\big(\frac{6}{5},\frac{10}{3}\big)$.
Here $\mathcal R\in\R^{3\times 3}$ suitably.
In order to obtain \eqref{eq:0103c} for any $s<\frac{10}{3}$ we finally combine
Lemma~\ref{thm:weakstrong} with \eqref{eq:1306a}.
 On account of \eqref{0103a} and \eqref{eq:0103c}
we conclude (extending $\varrho$ with $0$ outside of $\Omega_{\regkap\zeta}$) 
\begin{align*}
\lim_{n\to\infty}\int_{I^*}\int_{\R^3}\abs{\bfu_n}^2\dxt
&=\lim_{n\to\infty}\int_{I^*}\int_{\R^3}\frac{\kappa +\varrho_n}{\kappa+\varrho}\abs{\bfu_n}^2\dxt
+\lim_{n\to\infty}\int_{I^*}\int_{\R^3}\frac{\varrho -\varrho_n}{\kappa+\varrho}\abs{\bfu_n}^2\dxt
\\
&=\int_{I^*}\int_{\R^3}\abs{\bfu}^2\dxt.
\end{align*}
Since strong norm convergence and weak convergence imply strong convergence the compactness is shown and the existence of a fixpoint follows by Theorem~\ref{lem:fix}.

This gives the claim of Theorem \ref{thm:regu}. 


\section{The viscous approximation}
\label{sec:4}
In this section we want to get rid of the regularization operator $\regkap$ in order to find a solution $(\bfu,\varrho,\eta)\in X_\eta^I\times X_\eta^I \times Y^I$
to the viscous approximation satisfying the following.
\begin{enumerate}[label={(E\arabic{*})}]
\item\label{E1} The regularized momentum equation in the sense that
\begin{align}\label{eq:visu}
\begin{aligned}
&\int_I\frac{\dd}{\dt}\int_{\Omega_{ \eta}}\varrho\bfu\cdot \bfphi\dxt-\int_I\int_{\Omega_{\eta}} \Big(\varrho\bfu\cdot \partial_t\bfphi +\varrho\bfu\otimes \bfu:\nabla \bfphi\Big)\dxt\\&+\int_I\int_{\Omega_{ \eta}}\Big(\mu\nabla\bfu:\nabla\bfphi +(\lambda+\mu)\Div\bfu\,\Div\bfphi\Big)\dxt
\\
&-\int_I\int_{\Omega_{ \eta }}
(\varrho^\gamma+\delta\varrho^\beta)\,\Div\bfphi\dxt+\int_I\int_{\Omega_{ \eta }}\varepsilon\nabla(\varrho\bfu):\nabla\bfphi\dxt
\\
&+\int_I\frac{\dd}{\dt}\int_M \partial_t \eta b\dH\dt-\int_I\int_M \partial_t\eta\,\partial_t b\dH\dt + \int_I\int_M K'(\eta)\,b\dH\dt
\\&=\int_I\int_{\Omega_{\eta}}\varrho\bff\cdot\bfphi\dxt+\int_I\int_M g\,b\,\dd x\dt
\end{aligned}
\end{align} 
for all test-functions $(b,\bfphi)\in C^\infty_0(M)\times C^\infty(\overline{I}\times \R^3)$ with $\mathrm{tr}_\eta\bfphi=b\nu$. Moreover, we have $(\varrho\bfu)(0)=\bfq_0$, $\eta(0)=\eta_0$ and $\partial_t\eta(0)=\eta_1$. 
\item\label{E2}  The regularized continuity equation in the sense that
\begin{align}\label{eq:visvarrho}
\begin{aligned}
&\int_I\bigg(\frac{\dd}{\dt}\int_{\Omega_{\eta}}\varrho \psi\dx-\int_I\int_{\Omega_{\eta}}\Big(\varrho\,\partial_t\psi
+\varrho\bfu\cdot\nabla\psi\Big)\dx\bigg)\dt
+\varepsilon\int_I\int_{\Omega_{ \eta}}\nabla\varrho\cdot\nabla\psi\dxt=0
\end{aligned}
\end{align}
for all $\psi\in C^\infty(\overline{I}\times\R^3)$ and we have $\varrho(0)=\varrho_0$. 
\item\label{E3}  The boundary condition $\mathrm{tr}_\eta\bfu=\partial_t\eta\nu$ in the sense of Lemma \ref{lem:2.28}.
\end{enumerate}
\begin{theorem}\label{thm:visu}
There is a solution $(\eta,\bfu,\varrho)\in Y^I\times X_\eta^I\times X_\eta^I$ to \ref{E1}--\ref{E3}. Here, we have $I=(0,T_*)$, where $T_*<T$ only if $\lim_{t\rightarrow T^\ast}\|\eta(t,\cdot)\|_{L^\infty(\partial\Omega)}=\frac{L}{2}$.
The solution satisfies the energy estimate
\begin{align}
\label{apri:eps}
\begin{aligned}
\sup_{t\in I}\int_{\Omega_{  \eta}}&\varrho|\bfu|^2\dx+\sup_{t\in I}\int_{\Omega_{ \eta}}\big(a\varrho^\gamma+\delta\varrho^\beta\big)\dx+\int_I\int_{\Omega_{\eta}}|\nabla\bfu|^2\dxt\\
&+\varepsilon\int_I\int_{\Omega_{\eta}}\varrho^{\gamma-2}|\nabla\varrho|^2+\varrho\abs{\nabla \bfu}^2\dxt+\sup_{t\in I}\int_M|\partial_t\eta|^2\,\dd\mathcal H^2+\sup_{t\in I}K(\eta)\\
&\leq\,c\,\bigg(\int_{\Omega_{\eta}}|\bfq_0|^2\dx+\int_{\Omega_{\eta}}\big(\varrho_0^\gamma    
+\delta\varrho_0^\beta\big)\dx+\int_I\|\bff\|_{L^2(\Omega_{\eta})}^2\dt\bigg)
\\
&+c\bigg(\int_M\abs{\eta_0}^2
\dH+\int_M\abs{\eta_1}^2
\dH+K(\eta_0) +\int_I\norm{g}_{L^2(M)}^2\dt\bigg),
\end{aligned}
\end{align}
provided that $\eta_0,\eta_1,\varrho_0,\bfq_0,\bff$ and $g$ are regular enough to give sense to the right-hand side, that $\varrho_0\geq0$ a.e. and \eqref{eq:compa} is satisfied.
The constant c is independent of $\delta,\varepsilon$.
\end{theorem}
\begin{lemma}
\label{cor:ap}
Under the assumptions of Theorem \ref{thm:visu}, the continuity equation holds in the renormalized sense that is
\begin{align}
\label{eq:Tk01'}
\begin{aligned}
\int_I&\frac{\dd}{\dt}\int_{\R^3}\theta(\varrho)\psi\dxt-\int_I\int_{\R^3}\Big(\theta(\varrho)\partial_t\psi +\theta(\varrho) \bfu \cdot \nabla \psi\Big) \dxt\\
& \leq- \int_I\int_{\R^3}(\varrho\theta'(\varrho)-\theta(\varrho_n)) \diver\bfu\, \psi\dxt-\epsilon\int_I\int_{\R^3}\nabla \varrho\cdot \nabla \psi\dxt\\
\end{aligned}
\end{align}
for all $\psi\in C^\infty(\overline I\times\R^3,[0,\infty))$ and all convex  $\theta\in C^1(\R)$, with 
$\theta'(z)=0$ for $z\geq M_\theta$.
\end{lemma}
\begin{proof}[Proof of Theorem~\ref{thm:visu}]
In Theorem~\ref{thm:regu} we take $\kappa:={1/n}$ where $1/n$
is the regularizing parameter. We call the corresponding solution $(\eta_n,\bfu_n,\varrho_n)$.
If $n\rightarrow\infty$ then $\mathscr R_{1/n}\rightarrow \mathrm{id}$. Now we analyze the convergence of $(\eta_n,\bfu_n,\varrho_n)$.  
The estimate from Theorem~\ref{thm:regu} holds uniformly with respect to $n$. 
Additionally, by testing the continuity equation with $\varrho_\varepsilon$ and using $\beta\geq4$ we find that
\begin{align}\label{est:nablarho'}
\nabla\varrho_n\in L^2(I;L^2(\Omega_{\eta_\epsilon}))\text{ uniformly.}
\end{align}
Hence,
by Lemma~\ref{thm:weakstrong} we find that there is a subsequence such that for some $\alpha\in (0,1)$ fixed we have
\begin{align}
\eta_n&\rightharpoonup^\ast\eta\quad\text{in}\quad L^\infty(I;W^{2,2}(M))\label{eq:conveta'}\\
\eta_n&\rightharpoonup^\ast\eta\quad\text{in}\quad W^{1,\infty}(I;L^2(M)),\label{eq:conetat'}\\
\eta_\epsilon&\to\eta\quad\text{in}\quad C^\alpha(\overline I\times M),\label{eq:conetat''}
\\
\bfu_n&\weakto^\eta\bfu\quad\text{in}\quad L^2(I;W^{1,2}(\Omega_{\mathscr R_{1/n}\eta_n})),\label{eq:convu''}\\
\tfrac{1}{\sqrt{n}}\bfu_n&\to^\eta0\quad\text{in}\quad L^\infty(I;L^{2}(\Omega_{\mathscr R_{1/n}\eta_n})),\label{eq:convu''A}\\
\varrho_n&\weakto^{\ast,\eta}\varrho\quad\text{in}\quad L^\infty(I;L^\beta(\Omega_{\mathscr R_{1/n}\eta_n})).\label{eq:convrho'}
\end{align}
Moreover, Lemma~\ref{thm:weakstrong}, Remark~\ref{rem:strong}, the fact that $\beta\geq4$ and interpolation imply
\begin{align}
\varrho_n&\rightharpoonup^{\ast,\eta}\varrho\quad\text{in}\quad L^\infty(I;L^2(\Omega_{\mathscr R_{1/n}\eta_n})),\label{eq:convrho2}\\
\varrho_n&\rightharpoonup^\eta\varrho\quad\text{in}\quad L^2(I;W^{1,2}(\Omega_{\eta_n})),\label{eq:convrho3'}\\
\varrho_n&\rightarrow^\eta\varrho\quad\text{in}\quad L^2(I;L^2(\Omega_{\mathscr R_{1/n}\eta_n})).\label{eq:convrho2'}
\end{align}
The last convergence, \eqref{eq:convu''} and Lemma~\ref{thm:weakstrong} imply
\begin{align}
\varrho_n\bfu_n&\rightharpoonup^\eta\varrho\bfu\quad\text{in}\quad L^2(I;L^2(\Omega_{\mathscr R_{1/n}\eta_n})),\label{eq:convrhou'}\\
\varrho_n\bfu_n\otimes\bfu_n&\rightharpoonup^\eta\varrho\bfu\otimes\bfu\quad\text{in}\quad L^1(I;L^1(\mathscr R_{1/n}\eta_n)).\label{eq:convrhou}
\end{align}
Therefore we can pass to the limit in the equation and obtain a weak solution
to the viscous approximation. 
 The energy inequality is a consequence of lower semi-continuity.
 \end{proof}
 \begin{proof}[Proof of  Lemma \ref{cor:ap}]
 First, observe that since $\varrho_n$ is a renormalized solution to the continuity equation by Theorem~\ref{lem:warme} c), i.e. we have
\begin{align*}
\int_I&\frac{\dd}{\dt}\int_{\R^3}\theta(\varrho_n)\psi\dxt-\int_I\int_{\R^3}\Big(\theta(\varrho_n)\partial_t\psi +\theta(\varrho_n) \bfu_n \cdot \nabla \psi\Big) \dxt
\\
& =- \int_I\int_{\R^3}\Big((\varrho_n\theta'(\varrho_n)-\theta(\varrho_n)) \diver\bfu_n\, \psi\dxt -\int_I\int_{\R^3}\Big(\varepsilon\theta''(\varrho_n)\abs{\nabla \varrho_n}^2\psi+\epsilon\nabla \varrho_n\cdot \nabla \psi\Big)\dxt.
\end{align*}
As $\theta$ is convex this yields
\begin{align}
\label{eq:Tk0}
\begin{aligned}
\int_I&\frac{\dd}{\dt}\int_{\R^3}\theta(\varrho_n)\psi\dxt-\int_I\int_{\R^3}\Big(\theta(\varrho_n)\partial_t\psi +\theta(\varrho_n) \bfu_n \cdot \nabla \psi\Big) \dxt\\
& \leq- \int_I\int_{\R^3}(\varrho_n\theta'(\varrho_n)-\theta(\varrho_n)) \diver\bfu_n \,\psi\dxt-\epsilon\int_I\int_{\R^3}\nabla \varrho_n\cdot \nabla \psi\dxt\\
\end{aligned}
\end{align}
for all $\psi\in C^\infty(\overline{I}\times\R^3)$ and all $\theta\in C^2(\R)$ with 
$\theta'(z)=0$ for $z\geq M_\theta$. By approximation its is easy to see that the assumption $\theta\in C^1(\R)$ suffices. Due to the convergences
\eqref{eq:conetat''}, \eqref{eq:convu''}, \eqref{eq:convrho3'} and \eqref{eq:convrho2'} we can pass to the limit in \eqref{eq:Tk0}. This implies
\eqref{eq:Tk01'}.
\end{proof}

\section{The vanishing viscosity limit}
\label{sec:5}

The aim of this Section is to study the limit $\varepsilon\rightarrow0$ in the approximate system \eqref{eq:visu}--\eqref{eq:visvarrho} and establish the existence of a weak solution $(\eta,\varrho,\bfu)$ to the system with artificial viscosity in the following sense. We define
$$
\widetilde W_\eta^I= C_w(\overline{I};L^\beta(\Omega_\eta)).
$$
A weak solution is a triple $(\bfu,\varrho,\eta)\in X_{\regkap\eta}^I\times\widetilde W_\eta^I\times Y^I$ that satisfies the following.
\begin{enumerate}[label={(D\arabic{*})}]
\item\label{D1} The momentum equation in the sense that\begin{align}\label{eq:apu}
\begin{aligned}
&\int_I\frac{\dd}{\dt}\int_{\Omega_{ \eta}}\varrho\bfu \cdot\bfphi\dx-\int_{\Omega_{\eta}} \Big(\varrho\bfu\cdot \partial_t\bfphi +\varrho\bfu\otimes \bfu:\nabla \bfphi\Big)\dxt
\\
&+\int_I\int_{\Omega_\eta}\Big(\mu\nabla\bfu:\nabla\bfphi +(\lambda+\mu)\Div\bfu\,\Div\bfphi\Big)\dxt-\int_I\int_{\Omega_{ \eta }}
(\varrho^\gamma+\delta\varrho^\beta)\,\Div\bfphi\dxt\\
&+\int_I\bigg(\frac{\dd}{\dt}\int_M \partial_t \eta b\dH-\int_M \partial_t\eta\,\partial_t b\dH + \int_M K'(\eta)\,b\dH\bigg)\dt
\\&=\int_I\int_{\Omega_{\eta}}\varrho\bff\cdot\bfphi\dxt+\int_I\int_M g\,b\,\dd x\dt
\end{aligned}
\end{align} 
for all $(b,\bfphi)\in C^\infty_0(M)\times C^\infty(\overline{I}\times \R^3)$ with $\mathrm{tr}_{\eta}\bfphi=b\nu$. Moreover, we have $(\varrho\bfu)(0)=\bfq_0$, $\eta(0)=\eta_0$ and $\partial_t\eta(0)=\eta_1$. 
\item\label{D2}  The following continuity equation in the sense that
\begin{align}\label{eq:apvarrho0}
\begin{aligned}
&\int_I\frac{\dd}{\dt}\int_{\Omega_{\eta}}\varrho \psi\dxt-\int_I\int_{\Omega_{\eta}}\Big(\varrho\partial_t\psi
+\varrho\bfu\cdot\nabla\psi\Big)\dxt=0
\end{aligned}
\end{align}
for all $\psi\in C^\infty(\overline{I}\times\R^3)$ and we have $\varrho(0)=\varrho_0$. 
\item\label{D3}  The boundary condition $\mathrm{tr}_\eta\bfu=\partial_t\eta\nu$ in the sense of Lemma \ref{lem:2.28}.
\end{enumerate}
\begin{theorem}\label{thm:ap}
There is a solution $(\eta,\bfu,\varrho)\in Y^I\times X_\eta^I\times \widetilde W_\eta^I$ to \ref{D1}--\ref{D3}. 
Here, we have $I=(0,T_*)$, where $T_*<T$ only if $\lim_{t\rightarrow T^\ast}\|\eta(t,\cdot)\|_{L^\infty(\partial\Omega)}=\frac{L}{2}$.
The solution satisfies the energy estimate
\begin{align*}
\sup_{t\in I}\int_{\Omega_{\eta}}\varrho|\bfu|^2\dx&+\sup_{t\in I}\int_{\Omega_{ \eta}}\big(a\varrho^\gamma+\delta\varrho^\beta\big)\dx+\int_I\int_{\Omega_{\eta}}|\nabla\bfu|^2\dxt+\sup_{t\in I}\int_M|\partial_t\eta|^2\,\dd\mathcal H^2+\sup_{t\in I}K(\eta)\\
&\leq\,c\,\bigg(\int_{\Omega_{\eta}}\frac{|\bfq_0|^2}{\varrho_0}\dx+\int_{\Omega_{ \eta}}\big(a\varrho_0^\gamma+\delta\varrho^\beta_0\big)\dx+\int_I\|\bff\|_{L^\infty(\Omega_{\eta})}^2\dt\bigg)\\
&+c\bigg(\int_M\abs{\eta_0}^2
\dH+\int_M\abs{\eta_1}^2
\dH+K(\eta_0) +\int_I\norm{g}_{L^2(M)}^2\dt\bigg),
\end{align*}
provided that $\eta_1,\varrho_0,\bfq_0,\bff$ and $g$ are regular enough to give sense to the right-hand side, that $\varrho_0\geq0$ a.e. and \eqref{eq:compa} is satisfied.
The constant c is independent of $\delta$.
\end{theorem}
\begin{lemma}
\label{cor:ap1}
Under the assumptions of Theorem \ref{thm:ap}, the continuity equation holds in the renormalized sense that is
\begin{align}\label{eq:final3}
\begin{aligned}
\int_I\frac{\dd}{\dt}\int_{\Omega_\eta}\theta(\varrho)\psi\dxt&-\int_I\int_{\Omega_\eta}\Big(\theta(\varrho)\partial_t\psi +\theta(\varrho) \bfu\cdot \nabla \psi\Big)\dxt
\\
& =- \int_I\int_{\Omega_\eta}(\varrho\theta'(\varrho)-\theta(\varrho)) \diver\bfu \,\psi\dxt
\end{aligned}
\end{align}
for all $\psi\in C^\infty(\overline I\times\R^3)$ and all $\theta\in C^1(\R)$ with 
$\theta'(z)=0$ for $z\geq M_\theta$.
\end{lemma}

The proof will be split in several parts.
For a given $\varepsilon$ we gain a weak solutions $(\eta_\varepsilon,\bfu_\varepsilon,\varrho_\varepsilon)$ to \eqref{eq:visu}--\eqref{eq:visvarrho} by Theorem \ref{thm:visu}.
The estimate from Theorem \ref{thm:visu} holds uniformly with respect to $\varepsilon$. 
In particular,
\begin{align}
\label{apri:eps1}
\begin{aligned}
\sup_{t\in I}\int_{\Omega_{  \eta_\epsilon}}&\varrho_\epsilon|\bfu_\epsilon|^2\dx+\sup_{t\in I}\int_{\Omega_{ \eta_\epsilon}}\big(a\varrho_\epsilon^\gamma+\delta\varrho_\epsilon^\beta\big)\dx+\int_I\int_{\Omega_{\eta_\epsilon}}|\nabla\bfu_\epsilon|^2\dxt\\
&+\sup_{t\in I}\int_M|\partial_t\eta_\epsilon|^2\,\dd\mathcal H^2+\sup_{t\in I}K(\eta_\epsilon)\\
&\leq C(\eta_1, g,\bff,\varrho_0,\bfq_0)
\end{aligned}
\end{align}
is satisfied uniformly in $\epsilon$ for the time interval $I$.
Hence we may take a subsequence, such that
by Lemma~\ref{thm:weakstrong}, we find for $s\in (1,2)$ and for some $\alpha\in (0,1)$ fixed subsequences which satisfy
\begin{align}
\eta_\epsilon&\rightharpoonup^\ast\eta\quad\text{in}\quad L^\infty(I;W^{2,2}(M))\label{eq:conveta1}\\
\eta_\epsilon&\rightharpoonup^\ast\eta\quad\text{in}\quad W^{1,\infty}(I;L^2(M)),
\label{eq:conetat1}\\
\eta_\epsilon&\to\eta\quad\text{in}\quad C^\alpha(\overline I\times M)),
\label{eq:conetatal}\\
\bfu_\epsilon&\rightharpoonup^\eta\bfu\quad\text{in}\quad L^2(I;W^{1,2}(\Omega_{\eta_\epsilon})),\label{eq:convu1}\\
\varrho_\epsilon&\rightharpoonup^{\ast,\eta}\varrho\quad\text{in}\quad L^\infty(I;L^\beta(\Omega_{\eta_\epsilon})).\label{eq:convrho1}
\end{align}
Additionally, by testing the continuity equation with $\varrho_\varepsilon$, we find that
\begin{align}\label{est:nablarho}
\int_I\int_{\Omega_{\eta_\epsilon}}\epsilon\abs{\nabla\varrho_\varepsilon}^2\, dx\, dt\leq C,
\end{align}
with $C$ independent of $\epsilon$.
Now, using the renormalized continuity equation from Lemma \ref{cor:ap} with $\theta(z)=z^2$ and testing with $\psi\equiv 1$ we find using $\beta>4$ that
\begin{align}
\varepsilon\nabla\varrho_\varepsilon\rightarrow^\eta 0\quad\text{in}\quad L^2(I\times\Omega_{\eta_\epsilon}),\label{eq:van1}\\
\varepsilon\nabla(\bfu_\varepsilon\varrho_\varepsilon)\rightarrow^\eta 0\quad\text{in}\quad L^1(I\times\Omega_{\eta_\epsilon}).\label{eq:van2}
\end{align}
Please observe that this is the only essential point in the proof where we make use of the fact that $\varrho_\epsilon$ is uniformly bounded in $L^\beta$. 
By Lemma~\ref{thm:weakstrong} we find for $q\in (1,\frac{6\beta}{\beta+6})$, that
\begin{align}
\varrho_\varepsilon\bu_\varepsilon&\rightharpoonup^\eta  {\varrho}  {\bfu}\qquad\text{in}\qquad L^2(I, L^q(\Omega_{\eta_\epsilon}))\label{conv:rhov2}\\
{\varrho}_\varepsilon  {\bfu}_\varepsilon\otimes  {\bfu}_\varepsilon&\rightharpoonup^\eta  {\varrho}  {\bfu}\otimes  {\bfu}\qquad\text{in}\qquad L^1(I\times\Omega_{\eta_\epsilon}).\label{conv:rhovv2}
\end{align}

\subsection{Equiintegrability of the pressure}
The problem one has to handle first is that since the pressure is merrily bounded in $L^1$ in space it might converge to a measure and not a function. This is usually excluded by showing that the pressure possesses higher integrability properties. This would then allow to deduce a weakly converging subsequence in the space of better integrability and hence get a function as a limit object. In  the case of a moving domain standard procedures do not work and global higher integrability on the moving domain can not be achieved. The solution is two divide the problem in two steps: The first step is to improve the space integrability of the pressure inside the moving domain. The second step is to show, that the mass of the pressure can not be concentrated on the boundary. Combining the two results will imply equiintegrability of the pressure which is equivalent to $L^1$ compactness. The next two lemma settle that matter. The first one is merrily a localized version of the standard procedure. 
\begin{lemma}\label{prop:higher}
Let $Q=J\times B\Subset I\times\Omega_{\eta}$ be a parabolic cube. The following holds for any $\epsilon\leq\epsilon_0(Q)$
\begin{equation}\label{eq:gamma+1}
\int_{Q}\big(a\varrho_\varepsilon^{\gamma+1}
+\delta\varrho_\varepsilon^{\beta+1}\big)\,\dif x\,\dif t\leq C(Q)
\end{equation}
with constant independent of $\epsilon$.
\end{lemma}
\begin{proof}
We consider a parabolic cube $\tilde Q=\tilde J\times\tilde B$ with
$Q\Subset \tilde Q\Subset  I\times\Omega_\eta$. Due to \eqref{eq:conetatal} we can assume
that $\tilde Q\Subset  I\times\Omega_{\eta_\varepsilon}^I$ (by taking $\varepsilon$ small enough). 
Next we wish to test with $\psi\nabla\Delta^{-1}_{\tilde{B}}\varrho_\varepsilon$ where $\psi\in C^\infty_0(\tilde Q)$ with $\psi=1$ in $Q$ and $\Delta^{-1}_{\tilde{B}}\varrho_\varepsilon$ is defined as the unique
$W^{2,\beta}(\tilde B)\cap W^{1,\beta^*}_0(\tilde B)$-solution to
the equation
\begin{align}
\label{lap:1}
 -\Delta v =\varrho_\epsilon\quad\text{in}\quad\tilde B.
\end{align}
In order to deal with the term involving the time derivative we use the continuity equation. We find that
\begin{align}
\label{lap:2}
-\Delta \partial_t v=\partial_t\varrho_\epsilon=-\diver (\varrho_\epsilon \bfu_\epsilon+\varepsilon\nabla\varrho_\epsilon )
\end{align}
in the sense of distributions, such that $\partial_t\nabla\Delta^{-1}_{\tilde{B}}\varrho_\varepsilon=\nabla\Delta^{-1}_{\tilde{B}}\diver(\varrho_\epsilon \bfu_\epsilon+\varepsilon\nabla\varrho_\epsilon )$. 
Hence we find
\begin{align}\label{eq:testtheta}
\begin{aligned}
J_0&:=\int_I\int_ {\mathbb R^3}\psi \big(a\varrho_\varepsilon^{\gamma+1}+\delta\varrho_\varepsilon^{\beta+1}\big)\dxs
\\
&=\mu\int_I\int_ {\mathbb R^3}\psi \nabla \bfu_\varepsilon:\nabla^2\Delta^{-1}_{\tilde{B}} \varrho_\varepsilon\dxs+\mu\int_I\int_ {\mathbb R^3} \nabla \bfu_\varepsilon:\nabla\psi\otimes\nabla \Delta^{-1}_{\tilde{B}}\varrho_\varepsilon\dxs
\\
&+(\lambda+\mu)\int_I\int_ {\mathbb R^3}\psi \diver\bfu_\varepsilon\,\varrho_\varepsilon\dxs+(\lambda+\mu)\int_I\int_ {\mathbb R^3} \diver\bfu_\varepsilon\,\nabla\psi\cdot\nabla\Delta^{-1}_{\tilde{B}}\varrho_\varepsilon\dxs
\\
&-\int_I\int_ {\mathbb R^3}  \psi\varrho_\varepsilon \bfu_\varepsilon\otimes \bfu_\varepsilon:\nabla^2\Delta^{-1}_{\tilde{B}}\varrho_\varepsilon\dxs
-\int_I\int_ {\mathbb R^3}  \varrho_\varepsilon \bfu_\varepsilon\otimes \bfu_\varepsilon:\nabla\psi\otimes\nabla \Delta^{-1}_{\tilde{B}}\varrho_\varepsilon\dxs
\\
&+\varepsilon\int_I\int_{\mathbb R^3}\nabla(\bfu_\varepsilon\varrho_\varepsilon):\nabla^2\Delta^{-1}_{\tilde{B}}\varrho_\varepsilon\psi\dxs
+\varepsilon\int_I\int_{\mathbb R^3}\nabla(\bfu_\varepsilon\varrho_\varepsilon):\nabla\Delta^{-1}_{\tilde{B}}\varrho_\varepsilon\otimes\nabla\psi\dxs
\\
&-\int_I\int_ {\mathbb R^3} \big(a\varrho_\varepsilon^{\gamma}+\delta\varrho_\varepsilon^{\beta}\big)\nabla\psi\cdot\nabla\Delta^{-1}_{\tilde{B}}\varrho_\varepsilon\dxs -\int_I\int_ {\mathbb R^3} \varrho_\epsilon \bff\cdot \nabla \Delta^{-1}_{\tilde{B}}\varrho_\varepsilon\dxs
\\
&+\int_I\int_{\R^3}\psi\varrho_\varepsilon\bfu_\varepsilon\nabla\Delta^{-1}_{\tilde{B}}\diver(\varrho_\varepsilon\bfu_\varepsilon+\epsilon \nabla \varrho_\epsilon)\dxs
-\int_I\int_{\R^3}\partial_t\psi\varrho_\varepsilon\bfu_\varepsilon\cdot\nabla \Delta^{-1}_{\tilde{B}}\varrho_\varepsilon\dxs
\\=&:J_1+\cdots +J_{12}.
\end{aligned}
\end{align}
Our goal is to find an estimate for the expectation of $J_0$ which means that we have to find suitable bounds for all the other terms.
Using the continuity of the operator $\nabla\Delta^{-1}_{\tilde{B}}$ and Sobolev's embedding theorem, we obtain for some $p>3$ that
\begin{align}\label{eq:Linfty}
\begin{aligned}
\|\nabla\Delta^{-1}_{\tilde{B}}\varrho_\varepsilon\|_{L^\infty(\tilde B)}&\leq C\,\|\nabla^2\Delta^{-1}_{\tilde{B}}\varrho_\varepsilon\|_{L^p(\tilde B)}\leq C\,\| \varrho_\varepsilon\|_{L^p(\tilde B)}.
\end{aligned}
\end{align}
This is due to $\beta>3$. Note that in particular we have shown that $\psi\nabla\Delta^{-1}_{\tilde{B}}\varrho_\varepsilon\in L^\infty(I\times\R^3)$ uniformly in $\varepsilon$. As $\varrho_\varepsilon\in L^2(I\times \tilde Q )$ uniformly due to $\beta\geq2$ we deduce that
 $|J_1|\leq C$ as a consequence of uniform bounds on $\bfu_\epsilon$ and the continuity of the operator $\nabla^2\Delta^{-1}_{\tilde{B}}$. Similar arguments lead to the bound for $J_2,J_3,J_4$. The most critical is the convective term, $J_5$. It can be estimated using the continuity of $\nabla^2\Delta^{-1}_{\tilde{B}}$, Sobolev's embedding theorem, H\"older's inequality and \eqref{apri:eps1}
\begin{align*}
|J_5|&\leq C\,\int_{\tilde J}\| \varrho_\varepsilon\|_{L^3(\tilde Q)}\| \bfu_\varepsilon\|_{L^6(\tilde Q)}^2\| \varrho_\varepsilon\|_{L^3(\tilde Q)}\,\dd s\leq C\,\sup_{\tilde J}\| \varrho_\varepsilon\|_{L^3(\tilde B)}^2\int_{\tilde Q}\big(|\bfu_\varepsilon|^2+|\nabla \bfu_\varepsilon|^2\big)\,\dif x\,\dd s\leq C.
\end{align*}
The term $J_6$ is estimated similarly.
For $J_7$ we have on accout of \eqref{apri:eps1}, \eqref{est:nablarho} as well as \eqref{eq:Linfty} and obtain
\begin{align*}
|J_7|&\leq\,\sup_{\tilde J}\|\nabla^2\Delta^{-1}_{\tilde{B}}\varrho_\varepsilon\|_{L^3(\tilde B)}\bigg(\,\int_{\tilde J}\norm{\nabla\bfu_\varepsilon}_{L^2(\tilde B)}^2+\norm{\bfu_\varepsilon}_{L^6(\tilde B)}^2
+\epsilon^2\norm{\nabla\varrho_\varepsilon}_{L^2(\tilde B)}^2+\epsilon^2\norm{\varrho_\varepsilon}_{L^6(\tilde B)}^2\dt\bigg),
\end{align*}
which is uniformly bounded
Similar for $J_8$ is bounded taking also into account \eqref{eq:Linfty}, that
\begin{align*}
|J_8|&\leq\,\sup_{\tilde J}\|\nabla\Delta^{-1}_{\tilde{B}}\varrho_\varepsilon\|_{L^\infty(\tilde B)}^2\bigg(\,\int_{\tilde J}\norm{\nabla\bfu_\varepsilon}_{L^2(\tilde B)}^2+\norm{\bfu_\varepsilon}_{L^6(\tilde B)}^2
+\epsilon^2\norm{\nabla\varrho_\varepsilon}_{L^2(\tilde B)}^2+\epsilon^2\norm{\varrho_\varepsilon}_{L^6(\tilde B)}^2\dt\bigg).
\end{align*}
The terms of $J_9, J_10$ can be estimated using by the bounds on the operator $\nabla\Delta^{-1}_{\tilde{B}}$ and H\"older's and Young's inequality.
The same is used to estimate
\begin{equation*}
\begin{split}
 J_{11}&\leq C\int_{\tilde Q}|\varrho_\varepsilon\bu_\varepsilon|^2\dxt+\epsilon C\bigg(\int_{\tilde Q}|\varrho_\varepsilon\bu_\varepsilon|^2\dxt\bigg)^\frac{1}{2}\bigg(\int_{\tilde Q}|\nabla\varrho_\varepsilon|^2\dxt\bigg)^\frac{1}{2}
\end{split}
\end{equation*}
which is finite since we have uniformly in $\varepsilon$
\begin{equation}\label{eq:L2}
\begin{split}
\varrho_\varepsilon\bu_\varepsilon\in L^2(\tilde J;L^{\frac{6\beta}{\beta+6}}(\tilde Q))
\end{split}
\end{equation}
which is a consequence of the fact that
$$\varrho_\varepsilon\in L^\infty(\tilde J;L^\beta(\tilde B))),\qquad \bu_\varepsilon\in L^2(\tilde J;L^6(\tilde B)))$$
uniformly in $\varepsilon.$ The remaining terms involving derivatives of $\psi$ are of lower order and can be estimated even easier. 
Plugging all together we obtain \eqref{eq:gamma+1} uniformly in $\varepsilon$.
\end{proof}

The standard method as used in the proof of Lemma \ref{eq:gamma+1} does not work up to the boundary. Also the usage of the \Bogovskii-operator -- common in literature as well -- does not help (recall that our boundary depends on time and is not Lipschitz).
In the following Lemma we show equi-integrability at the boundary related to the method from \cite{Kuk}.

\begin{lemma}\label{prop:higherb}
Let $\kappa>0$ be arbitrary. There is a measurable set $A_\kappa\Subset I\times\Omega_\eta$ such that
\begin{equation}\label{eq:gamma+1b}
\int_{I\times\R^3\setminus A_\kappa}\big(a\varrho_\varepsilon^{\gamma}
+\delta\varrho_\varepsilon^{\beta}\big)\chi_{\Omega_{\eta_\epsilon}}\,\dif x\,\dif t\leq \kappa.
\end{equation}
\end{lemma}
\begin{proof}
We construct a test-function, which has a positive and arbitrarily large divergence. For this let $\phi\in C^\infty_0(S_L;[0,1])$, such that $\chi_{S_\frac{L}{2}}\leq \phi\leq \chi_{\Omega_0\cup S_L}$ and $\abs{\nabla \phi}\leq \frac{c}{L}$. 
Since we know that $\abs{\eta_\varepsilon}\leq\frac{L}2$, we find that $\phi(x)\equiv 1$ on $S_\frac{L}{2}\cap \Omega_{\eta_\varepsilon}$. We extend $\phi$ by zero to $\Omega_\eta$ and 
define
\[
\bfphi_\varepsilon(t,x)=\phi\min\set{K(s(x)-\eta_\varepsilon(t,q(x))),1}\nu(q(x)).
\]
It is well defined, since $\phi\neq 0$ only in $S_L$, where the mapping $x\mapsto (q(x),s(x))$ is well defined, see Subsection~\ref{ssec:geom}. Observe, that we take coordinates with respect to the reference geometry $\Omega$ and with respect to the reverence outer normal $\nu$ on $\partial \Omega$.
On account of $\nabla s(x)=\nu(q(x))$ we have
\begin{align*}
\partial_j \bfphi^l_\varepsilon(t,x)
&=\partial_j\phi(x)\min\set{K(s(x)-\eta_\varepsilon(t,q(x))),1}\nu(q(x))\\&+K\chi_{\set{K(s(x)-\eta_\varepsilon(t,q(x)))\leq 1}}\nu_j(q(x))\nu_l(q(x))\phi(x)\\
&-K\chi_{\set{K(s(x)-\eta_\varepsilon(t,q(x)))\leq 1}}\nabla\eta_\varepsilon(t,q(x))\cdot\partial_jq(x)\nu_l(q(x))\phi(x)
\\&+\phi(x)\min\set{K(s(x)-\eta_\varepsilon(t,q(x))),1}\partial_j\nu_l(q(x))\\
&=\xi^1_{jl}(t,x)+\xi^2_{jl}(t,x)+\xi^3_{jl}(t,x)+\xi^4_{jl}(t,x).
\end{align*}
Observe that $\bfxi^1$ and $\bfxi^4$ are uniformly bounded by some constant $c_{\bfxi}$. Moreover, for every $p\in (1,\infty)$, $q>p$, the following holds
\begin{align*}
\bigg(\int_{I}\int_{\Omega_{\eta_\varepsilon}}|\bfxi^3|^p\dxt\bigg)^{\frac{1}{p}}&\leq \,cK\bigg(\int_{I}\int_{\Omega_{\eta_\varepsilon}}\chi_{\set{K(s-\eta_\varepsilon\circ q)\leq 1}}|\nabla\eta_\varepsilon|^p\dxt\bigg)^{\frac{1}{p}}\\
&\leq \,cK\bigg(\int_{I}\int_{\Omega_{\eta_\varepsilon}}|\nabla\eta_\varepsilon|^q\dxt\bigg)^{\frac{1}{q}}\big|\set{K(s-\eta_\varepsilon\circ q)\leq 1}\big|^{\frac{1}{q'p}}\\
&\leq\,c_pK^{1-\frac{1}{pq'}}
\end{align*}
uniformly in $\varepsilon$, cf. \eqref{apri:eps1}. Estimating $\bfxi^2$ in a similar way we gain
\begin{align}\label{eq:51}
\bigg(\int_{I}\int_{\Omega_{\eta_\varepsilon}}|\bfxi^2|^p\dxt\bigg)^{\frac{1}{p}}
&\leq\,c_p\big( K^{1-\frac{1}{pp'}}+1\big)
\end{align}
for all $p<\infty$.
Finally, we use the fact, that $\nabla q^i$ are all living in the tangentplane of $\partial\Omega$ and are therefore orthogonal to $\nu(q(x))$. Hence we have
$\xi^3_{jj}=0$.
 This implies, that for every $K>0$, there is a $\kappa$, such that we have 
\begin{align}\label{eq:Phi1}
\divergence\bfphi_\varepsilon\geq K-c_\bfxi\quad\text{in}\quad \Omega_{\eta_\varepsilon}\setminus\big\{x\in\Omega_{\eta_\varepsilon}:\,\mathrm{dist}(\partial\Omega_{\eta_\varepsilon})\geq\frac{1}{K}\big\}.
\end{align}
Finally we calculate
\[
\partial_t\bfphi_\varepsilon(t,x)=-K\chi_{\set{K(s(x)-\eta_\varepsilon(t,q(x)))\leq 1}}\partial_t\eta_\varepsilon(t,q(x))\nu(q(x)).
\]
Due to \eqref{apri:eps1} we have
\begin{align}\label{eq:51b1}
\begin{aligned}
\bigg(\sup_{I}\int_{\Omega_{\eta_\varepsilon}}|\partial_t\bfphi_\varepsilon|^r\dx\bigg)^{\frac{1}{r}}&\leq \,cK\bigg(\sup_I\int_{\Omega_{\eta_\varepsilon}}|\partial_t\eta_\varepsilon|^2\dxt\bigg)^{\frac{1}{2}}\big|\set{K(s-\eta_\varepsilon\circ q)\leq 1}\big|^{\frac{2-r}{r^2}}\\
&\leq \,c\,K^{1-\frac{2-r}{r^2}}
\end{aligned}
\end{align} 
for all $r<2$, similar to \eqref{eq:51}.
Now using $\bfphi_\varepsilon$ as a test-function (note that $\bfphi_\varepsilon=0$ on $\partial\Omega_{\eta_\varepsilon}$),
we obtain via smooth approximation that 
\begin{align}\label{eq:1712}
\int_I\int_{\Omega_{\eta_\varepsilon}} (\varrho^\gamma+\delta\varrho^\beta)\,\Div\bfphi_\varepsilon\dxt\leq -
\int_I\int_{\Omega_{\eta_\varepsilon}} \varrho_\varepsilon\bfu_\varepsilon\,\partial_t\bfphi_\varepsilon\dxt +C\big(K^{1-\seb{\lambda}}+1\big),
\end{align}
for some fixed $C,\lambda>0$, where $C,\lambda$ are independent of $\varepsilon$. Here we used the uniform  integrability bounds of all other terms of the momentum equation \eqref{apri:eps} and  \eqref{eq:51}.
Taking \eqref{eq:L2} and $\beta>3$ into account we see that the remaining integral
in \eqref{eq:1712} is uniformly $p$-integrable for some exponent $p>1$ in terms of \eqref{eq:51b1} cf. \eqref{apri:eps1}. This means we have
\begin{align}\label{eq:2103}
\int_I\int_{\Omega_{\eta_\varepsilon}} (\varrho^\gamma+\delta\varrho^\beta)\,\Div\bfphi_\varepsilon\dxt\leq\,C(K^{1-\lambda}+1)
\end{align}
uniformly for some $\lambda>0$. Now, we set
$$A_\kappa=\big\{x\in\Omega_{\eta_\varepsilon}:\,\mathrm{dist}(\partial\Omega_{\eta_\varepsilon})\geq\frac{1}{K}\big\},$$
where $K=K(\kappa)$ is the solution to
$$ \frac{C(K^{1-\lambda}+1)}{K-c_\bfxi}=\kappa$$
with $C$ given in \eqref{eq:2103}. Note that such a $K$ always exists if $\kappa$ s small enough.
As a consequence of \eqref{eq:Phi1} and \eqref{eq:1712} we gain
\begin{align*}
\int_{I\times\R^3\setminus A_{\kappa}} (\varrho^\gamma+\delta\varrho^\beta)\dxt\leq\frac{1}{K-c_\bfxi}\int_{I\times\R^3\setminus A_{\kappa}} (\varrho^\gamma+\delta\varrho^\beta)\,\Div\bfphi_\varepsilon\dxt\leq\,\frac{C(K^{1-\lambda}+1)}{K-c_\bfxi}=\kappa.
\end{align*}
The claim follows.
\end{proof}

We connect Lemma~\ref{prop:higher} and Lemma~\ref{prop:higherb} to get the following corollary:
\begin{corollary}\label{prop:higherb0}Under the assumptions of Theorem~\ref{thm:ap} there existence of a function $\overline p$ such that
\[
\varrho_\epsilon^\gamma+\delta\varrho_\epsilon^\beta\weakto \overline p\text{ in }L^1(I\times \Omega_{\eta_\epsilon}),
\]
for a subsequence.
Additionally, for $\kappa>0$  arbitrary, there is a measurable set $A_\kappa\Subset I\times\Omega_\eta$ such that $\overline{p}\varrho\in L^1(A_\kappa)$ and
\begin{equation}\label{eq:gamma+1b0}
\int_{(I\times\Omega_{\eta})\setminus A_\kappa}\overline p\dxt\leq \kappa.
\end{equation}
\end{corollary}
Combining the corollary with the convergences \eqref{conv:rhovv2}--\eqref{eq:conveta1} we can pass to the limit in \eqref{eq:visu}--\eqref{eq:visvarrho} and obtain
that we have $(\bfu,\varrho,\eta,\overline p)\in X_{\regkap\eta}^I\times\widetilde W_\eta^I\times Y^I\times L^1(I\times \Omega_\eta)$ that satisfies
\[
\bfu(\cdot,\cdot+\regkap\eta\nu)=\partial_t\eta\nu_{\eta}
 \text{ in } I\times \partial\Omega,
\]
 the following continuity equation 
\begin{align}\label{eq:apvarrho}
\begin{aligned}
&\int_I\frac{\dd}{\dt}\int_{\Omega_{\eta}}\varrho \psi\dxt-\int_I\int_{\Omega_{\eta}}\Big(\varrho\partial_t\psi
+\varrho\bfu\cdot\nabla\psi\Big)\dxt=0
\end{aligned}
\end{align}
for all $\psi\in C^\infty(\overline{Q}_{\eta})$. And the coupled weak momentum equation
\begin{align}\label{eq:apulim}
\begin{aligned}
&\int_I\frac{\dd}{\dt}\int_{\Omega_{ \eta}}\varrho\bfu\cdot \bfphi\dxt-\int_I\int_{\Omega_{\eta}} \Big(\varrho\bfu\cdot \partial_t\bfphi +\varrho\bfu\otimes \bfu:\nabla \bfphi\Big)\dxt\\
&+\int_{\Omega_{\eta}}\Big(\mu\nabla\bfu:\nabla\bfphi +(\lambda+\mu)\Div\bfu\,\Div\bfphi\Big)\dxt-\int_I\int_{\Omega_{ \eta }}
\overline{p}\,\Div\bfphi\dxt
\\
&
+\int_I\frac{\dd}{\dt}\int_M \partial_t \eta b\dH-\int_M \partial_t\eta\,\partial_t b\dH + \int_M K'(\eta)\,b\dH\dt
\\&=\int_I\int_{\Omega_{ \eta}}\varrho\bff\cdot\bfphi\dxt+\int_I\int_M g\,b\,\dd x\dt
\end{aligned}
\end{align} 
 for all $(b,\bfphi)\in C^\infty_0(M)\times C^\infty(\overline{I}\times \setR^3)$ with $\mathrm{tr}_{\regkap \zeta}\bfphi=b\nu$.

It remains to show that $\overline{p}=a\varrho^\gamma+\delta\varrho^\beta$. This will be achieved in the following two subsections. 

\subsection{The effective viscous flux}
\label{subsec:strongconvdensity}
We fix $\epsilon_0>0$ and consider in the following just $\epsilon\in (0,\epsilon_0)$. Next we define
\begin{align*}
\Omega_{\epsilon_0}=\bigcap_{\epsilon\leq \epsilon_0} \Omega_{\eta_\epsilon}.
\end{align*}
It is the aim of this subsection to show for $\psi\in C_0^\infty(I\times\Omega_{\epsilon_0})$ that 
\begin{align}\label{eq:flux}
\begin{aligned}
\int_{I\times\Omega_{\eta_\epsilon}}&\psi^2\big( a\varrho_\varepsilon^\gamma+\delta\varrho_\varepsilon^\beta-(\lambda+2\mu)\Div \bfu_\varepsilon\big)\,\varrho_\varepsilon\dxt\\&\longrightarrow\int_{I\times\Omega_{\eta}} \psi^2\big( \overline{p}-(\lambda+2\mu)\Div \bfu\big)\,\varrho\dxt.
\end{aligned}
\end{align}

Testing the momentum equation with $\psi\nabla\Delta^{-1}(\psi \varrho_\varepsilon)$ implies
\begin{align*}
J_0&:=\int_I\int_ {\mathbb R^3}\psi^2 \big(\varrho_\varepsilon^{\gamma}+\delta\varrho_\varepsilon^{\beta}\big)\varrho_\varepsilon\dxs
\\
&=\mu\int_I\int_ {\mathbb R^3}\psi \nabla \bfu_\varepsilon:\nabla^2\Delta^{-1}(\psi \varrho_\varepsilon)\dxs
+\mu\int_I\int_ {\mathbb R^3} \nabla \bfu_\varepsilon:\nabla\psi\otimes\nabla\Delta^{-1}\psi \varrho_\varepsilon\dxs
\\
&+(\lambda+\mu)\int_I\int_ {\mathbb R^3}\psi^2 \diver\bfu_\varepsilon\,\varrho_\varepsilon\dxs
+(\lambda+\mu)\int_I\int_ {\mathbb R^3} \diver\bfu_\varepsilon\,\nabla\psi\cdot\nabla\Delta^{-1}(\psi \varrho_\varepsilon)\dxs
\\
&-\int_I\int_ {\mathbb R^3}  \varrho_\varepsilon \bfu_\varepsilon\otimes \bfu_\varepsilon:\nabla^2\Delta^{-1}(\psi \varrho_\varepsilon)\dxs
-\int_I\int_ {\mathbb R^3}  \varrho_\varepsilon \bfu_\varepsilon\otimes \bfu_\varepsilon:\nabla\psi\otimes\nabla\Delta^{-1}(\psi \varrho_\varepsilon)\dxs
\\
&-\int_I\int_ {\mathbb R^3} \big(\varrho_\varepsilon^{\gamma}+\delta\varrho_\varepsilon^{\beta}\big)\nabla\psi\cdot\nabla\Delta^{-1}(\psi \varrho_\varepsilon)\dxs -\int_I\int_ {\mathbb R^3} \varrho_\epsilon \bff\cdot \nabla \Delta^{-1}(\psi \varrho_\varepsilon)\dxs
\\
&-\int_I\int_{\R^3}\partial_t\psi\varrho_\varepsilon\bfu_\varepsilon\cdot\nabla \Delta^{-1}(\psi \varrho_\varepsilon)\dxs
- \int_I\int_{\R^3}\nabla(\Delta^{-1}(\partial_t\psi \varrho_\epsilon)\cdot \varrho_\epsilon \bfu_\epsilon\dxs \\&+ \int_I\int_{\R^3}\psi\varrho_\varepsilon\bfu_\varepsilon\cdot\nabla \Delta^{-1}(\psi \diver(\bfu_\epsilon \varrho_\epsilon))\dxs 
+\epsilon\int_I\int_{\R^3}\psi\varrho_\varepsilon\bfu_\varepsilon\cdot\nabla\Delta^{-1}\big(\psi\Delta \varrho_\epsilon\big)\dxs
\\
&+\varepsilon\int_I\int_{\mathbb R^3}\psi\nabla(\bfu_\varepsilon\varrho_\varepsilon)\cdot\nabla^2 \Delta^{-1}(\psi \varrho_\varepsilon)+\nabla(\bfu_\varepsilon\varrho_\varepsilon):\nabla \Delta^{-1}(\psi \varrho_\varepsilon)\otimes\nabla\psi\dxs
\\
&=J_1+\cdots +J_{11}+E_1+E_2.
\end{align*}
 Similarly, we obtain by testing the limit equation \eqref{eq:apulim}
\begin{align*}
K_0&:=\int_I\int_ {\mathbb R^3}\psi^2 \overline{p} \varrho\dxs
\\
&=\mu\int_I\int_ {\mathbb R^3}\psi \nabla \bfu:\nabla^2\Delta^{-1}(\psi \varrho)\dxs
+\mu\int_I\int_ {\mathbb R^3} \nabla \bfu:\nabla\psi\otimes\nabla\Delta^{-1}(\psi \varrho)\dxs
\\
&+(\lambda+\mu)\int_I\int_ {\mathbb R^3}\psi^2 \diver\bfu \varrho\dxs
+(\lambda+\mu)\int_I\int_ {\mathbb R^3} \diver\bfu\,\nabla\psi\cdot\nabla(\Delta^{-1}(\psi \varrho)\dxs
\\
&-\int_I\int_ {\mathbb R^3}  \varrho \bfu\otimes \bfu_:\nabla^2\Delta^{-1}(\psi \varrho)\dxs
-\int_I\int_ {\mathbb R^3}  \varrho\bfu\otimes \bfu:\nabla\psi\otimes\nabla\Delta^{-1}(\psi \varrho)\dxs
\\
&-\int_I\int_ {\mathbb R^3}\overline{p}\nabla\psi\cdot\nabla\Delta^{-1}(\psi \varrho))\dxs -\int_I\int_ {\mathbb R^3} \varrho \bff\cdot \nabla \Delta^{-1}(\psi \varrho)\dxs
\\
&-\int_I\int_{\R^3}\partial_t\psi\varrho\bfu\cdot\nabla \Delta^{-1}(\psi \varrho)\dxs
- \int_I\int_{\R^3}\nabla(\Delta^{-1}(\partial_t\psi \varrho)\cdot \varrho\bfu\dxs\\& + \int_I\int_{\R^3}\psi\varrho\bfu\cdot\nabla \Delta^{-1}(\psi \diver(\bfu \varrho))\dxs 
=K_1+\cdots +K_{11}.
\end{align*}
We define the operator $\mathcal{R}$ by $\mathcal{R}_{ij}:=\partial_j\Delta^{-1}\partial_i$ and get
\begin{equation}\label{eq:kkkn}
 \begin{split}
\int_I\int_{\mt}&\psi^2\big(a\varrho_\varepsilon^\gamma+\delta\varrho^\beta_\varepsilon
-(\lambda+2\mu)\diver\bu_\varepsilon\big)\varrho_\varepsilon\,\dif x\,\dif t
\\
&=J_1+J_2+J_4+J_6+...+J_{10} +J_{11}' +E_1+E_2
\\
&+\sum_{i,j}\int_I\int_{\mt} u^i_\varepsilon \big(\psi \varrho_\varepsilon \mathcal{R}_{ij}[\psi   \varrho_\varepsilon u^{j}_\varepsilon]-\psi\varrho_\varepsilon u^{j}_\varepsilon\mathcal{R}_{ij}[\psi \varrho_\varepsilon]\big)\,\dif x\,\dif \sigma,
 \end{split}
\end{equation}
where
$$J_{11}' =-\int_I\int_{\R^3}\psi\varrho_\varepsilon\bfu_\varepsilon\cdot\nabla \Delta^{-1}(\nabla \psi \cdot \bfu_\epsilon \rho_\epsilon))\dxs$$
Similarly, we obtain
\begin{equation}\label{kkk}
\begin{split}
\int_I\int_{\mt}&\psi^2\big(\overline{p}-(\lambda+2\mu)\diver\bu\big)\varrho\,\dif x\,\dif t\\
&=K_1+K_2+K_4+K_6+...+K_{10}+K_{11}' 
\\
&+\sum_{i,j}\int_I\int_{\mt} u^i \big(\psi \varrho \mathcal{R}_{ij}[\psi   \varrho u^{j}]-\psi\varrho u^{j}\mathcal{R}_{ij}[\psi \varrho]\big)\,\dif x\,\dif \sigma,
\end{split}
\end{equation}
where 
$$K_{11}'=-\int_I\int_{\R^3}\psi\varrho\bfu\cdot\nabla \Delta^{-1}(\nabla \psi \cdot \bfu \varrho))\dxs.$$
 Hence we now that
 \begin{align}
\nonumber
  \int_{I\times\Omega_{\eta_\epsilon}}&\psi\big( a\varrho_\varepsilon^\gamma+\delta\varrho_\varepsilon^\beta-(\lambda+2\mu)\Div \bfu_\varepsilon\big)\,\varrho_\varepsilon\dxt
  -\int_{I\times\Omega_{\eta}}\psi \big( \overline{p}-(\lambda+2\mu)\Div \bfu\big)\,\varrho\dxt
\nonumber  \\
\nonumber  &= J_1-K_1+J_2-K_2+J_4-K_4+J_6-K_6+...+J_{10}-K_{10}+J_{11}'-K_{11}'  +E_1+E_2
  \\
  \nonumber&+\sum_{i,j}\int_I\int_{\mt} u^i_\varepsilon \big(\psi \varrho_\varepsilon \mathcal{R}_{ij}[\psi   \varrho_\varepsilon u^{j}_\varepsilon]-\psi\varrho_\varepsilon u^{j}_\varepsilon\mathcal{R}_{ij}[\psi \varrho_\varepsilon]\big)\,\dif x\,\dif \sigma
  \\
  \label{eq:1} &-\int_I\int_{\mt}  u^i \big(\psi \varrho \mathcal{R}_{ij}[\psi   \varrho u^{j}]-\psi\varrho u^{j}\mathcal{R}_{ij}[\psi \varrho]\big)\,\dif x\,\dif \sigma.
 \end{align}
We will now show that the right hand side converges to $0$ with $\epsilon\to 0$. Please observe that after this preparation everything is localized and the known approach can be enforced to our problem. Nevertheless, to keep the result self contained we repeat the main steps of the argument here. 
 
 First, by the assumption on $\beta>3$ and the continuity of $\nabla \Delta^{-1}$ on $\setR^3$ we find that using additionally the fact that $\sqrt\varepsilon\nabla \varrho_\epsilon,\nabla \bfu$ are uniformly in $L^2$
\begin{equation*}
 \begin{split}
| E_2|&\leq\,C\sqrt{\varepsilon}\,\Big(\|\nabla^2\Delta^{-1}(\psi\varrho_\varepsilon)\|^3_{L^\infty(\tilde{I},(L^3(\tilde{B}))} +\|\nabla\Delta^{-1}(\psi\varrho_\varepsilon)\|^3_{L^\infty(\tilde{I},(L^3(\tilde{B}))}
\\
&\qquad+\|\nabla\bu_\varepsilon\|^3_{L^2(\tilde{Q})}+\|\bu_\varepsilon\|^3_{L^2(\tilde{I},(L^6(\tilde{B}))}+\|\sqrt\varepsilon\nabla \varrho_\epsilon\|^3_{L^2(\tilde{Q})}+\| \sqrt\epsilon\varrho_\epsilon\|^3_{L^2(\tilde{I},(L^6(\tilde{B}))}
\\
&\leq\sqrt{\varepsilon}\,\Big(\|\tilde\varrho_\varepsilon\|^3_{L^\infty(\tilde I;L^3(\tilde B))}+\|\nabla\bu_\varepsilon\|^3_{L^2(\tilde{Q})}
+\|\sqrt\varepsilon\nabla \varrho_\epsilon\|^3_{L^2(\tilde{Q})}+\| \sqrt\epsilon\varrho_\epsilon\|^3_{L^2(\tilde{I},(L^6(\tilde{B}))}\Big)
\\
&\leq \,C\sqrt{\varepsilon}.
 \end{split}
\end{equation*}
Similarly we find
$
|E_1|\leq C\sqrt{\varepsilon},
$
as well. Hence, $E_1+E_2\rightarrow 0$. Except for the last couple on the right hand side of \eqref{eq:1} all other couples term converge nicely, by the known weak and strong convergences to $0$.
The crucial point is to estimate the commutator term. We will prove it by showing that
$ \psi\varrho_\varepsilon\mathcal R[\psi  T_k(\varrho_\varepsilon) \bfu_\varepsilon]- \psi\varrho_\varepsilon \bfu_\varepsilon\mathcal R[\psi T_k(\varrho_\varepsilon)]$ converges strongly in $L^2(W^{-1,2})$. Then the crucial term converges, since $\psi \bu $ converges weakly in $L^2(W^{1,2})$.  For the identification of the limit we make use of the div-curl lemma.
From \eqref{1311c} and \eqref{conv:rhov2} we obtain that
$$\varrho_\varepsilon\rightharpoonup \varrho\quad\text{in}\quad L^{p}(\R^3)\quad \text{a.e. in }I,$$
$$\varrho_\varepsilon\bu_\varepsilon\rightharpoonup\varrho\bu\quad\text{in}\quad L^\frac{2\beta}{\beta+1}(\R^3)\quad\text{a.e. in }I.$$
Hence we can apply \cite[Lemma 3.4]{feireisl1} (to the sequences $\psi T_k(\varrho_\varepsilon)$ and $\psi   \varrho_\varepsilon u^{j}_\varepsilon$) to conclude that
$$
\psi \varrho_\varepsilon \mathcal{R}_{ij}[\psi   \varrho_\varepsilon u^{j}_\varepsilon]-\psi\varrho_\varepsilon u^{j}_\varepsilon\mathcal{R}_{ij}[\psi \varrho_\varepsilon]\rightharpoonup \psi \varrho \mathcal{R}_{ij}[\psi   \varrho u^{j}]-\psi\varrho u^{j}\mathcal{R}_{ij}[\psi \varrho]\quad\text{in}\quad L^r(\R^3)$$
a.e. in t,
where
$$\frac{1}{r}=\frac{1}{p}+\frac{\beta+1}{2\beta}<\frac{6}{5}$$
for $p$ large enough. Therefore $L^r(\mathcal O)$ is compactly embedded into $W^{-1,2}(\mathcal O)$ for $\mathcal O\Subset\R^3$. As a consequence we have
$$\psi T_k(\varrho_\varepsilon) \mathcal{R}_{ij}[\psi   \varrho_\varepsilon u^{j}_\varepsilon]-\psi\varrho_\varepsilon u^{j}_\varepsilon\mathcal{R}_{ij}[\psi T_k(\varrho_\varepsilon)]\rightharpoonup \psi T^{1,k} \mathcal{R}_{ij}[\psi   \varrho u^{j}]-\psi\varrho u^{j}\mathcal{R}_{ij}[\psi \varrho]\quad\text{in}\quad W^{-1,2}(\R^3)$$
a.e. in t using the compact support of the involved functions. Moreover, it is possible to show that for some $p>2$
\begin{equation*}
\begin{split}
\int_I&\big\|\psi \varrho_\varepsilon\mathcal{R}[\psi \varrho_\varepsilon\bu_\varepsilon]-\psi\varrho_\varepsilon\bu_\varepsilon\mathcal{R}[\psi\varrho_\varepsilon]\big\|_{W^{-1,2}(\mt)}^p\dt\\
&\leq C\,\int_{\tilde{I}}\|\varrho_\varepsilon\|_{L^{\beta+1}(\tilde B)}^{2pr}\dif t+C\,\sup_{t\in \tilde{I}}\|\varrho_\varepsilon\bu_\varepsilon\|_{L^\frac{2\beta}{\beta+1}(\tilde B)}^{2pr}\dt\leq C
\end{split}
\end{equation*}
which gives the desired convergence
$$\psi \varrho_\varepsilon\mathcal{R}[\psi \varrho_\varepsilon\bu_\varepsilon]-\psi\varrho_\varepsilon\bu_\varepsilon\mathcal{R}[\psi\varrho_\varepsilon]\rightarrow \psi \varrho\mathcal{R}[\psi \varrho\bu]-\psi\varrho\bu\mathcal{R}[\psi \varrho]\quad \text{in}\quad L^2(I;W^{-1,2}(\mt)).$$
Thus we conclude that
\begin{align}\label{eq:eq}
\begin{aligned}
\int_{I\times\mt}& \psi\varrho_\varepsilon\,u^i_\varepsilon\big(\mathcal{R}_{ij}[\psi\varrho_\varepsilon u^{j}_\varepsilon]-\psi\varrho_\varepsilon u^{j}_\varepsilon\mathcal{R}_{ij}[\psi \varrho_\varepsilon]\big)\,\dif x\,\dif t\\
&\qquad\rightarrow \int_{I\times \mt} \psi\varrho\,u^i\big(\mathcal{R}_{ij}[\psi \varrho u^{j}]-\psi \varrho u^{j}\mathcal{R}_{ij}[\psi \varrho]\big)\,\dif x\,\dif t
\end{aligned}
\end{align}
and accordingly
\begin{align}\label{eq:fluxpsi}
\begin{aligned}
\int_{I\times\mt} &\psi^2\big( a\varrho_\varepsilon^\gamma+\delta\varrho_\varepsilon^\beta-(\lambda+2\mu)\Div \bfu_\varepsilon\big)\,\varrho_\varepsilon\dxt\\&\longrightarrow\int_{I\times\mt}\psi^2 \big( \overline{p}-(\lambda+2\mu)\Div \bfu\big)\,\varrho\dxt.
\end{aligned}
\end{align}

\subsection{Renormalized solutions}
\label{sec:renep}
The aim of this section is to prove Lemma \ref{cor:ap1}. Similar to Lemma \ref{lem:warme} b) the proof is based on mollification and Lions' commutator estimate. Due to \eqref{eq:conetatal}, \eqref{eq:convrho1} and \eqref{conv:rhov2} it is easy to pass to the limit in \eqref{eq:visvarrho}. Hence we obtain
\begin{align*}
\int_I\frac{\dd}{\dt}\int_{\Omega_\eta}\varrho\,\psi\dx\dt
-\int_I\int_{\Omega_\eta}\big(\varrho\, \partial_t\psi +\varrho \bfu\cdot\nabla \psi\big)\dxt=0
\end{align*}
for all $\psi\in C^\infty(\overline I\times \R^3)$. 
We extend $\varrho$ by zero to $I\times\R^3$
and $\bfu$ by means of the extension operator
$$\mathscr E_{\eta}:W^{1,2}(\Omega_\eta)\rightarrow W^{1,p}(\R^3),$$ constructed in Lemma~\ref{lem:extension}
where $1<p<2$ (but may be chosen close to 2). 
Hence we find that
\begin{align*}
\int_I\frac{\dd}{\dt}\int_{\R^3}\varrho\,\psi\dx\dt
-\int_I\int_{\R^3}\big(\varrho\, \partial_t\psi +\varrho  \mathscr E_\eta\bfu\cdot\nabla \psi\big)\dxt=0
\end{align*}
for all $\psi\in C^\infty(\overline I\times \R^3)$. 
Now, analogous to the proof in Theorem~\ref{lem:warme}
we mollify the equation in space
using a standard convolution with parameter $\kappa>0$ in space. The following holds
\begin{align}\label{eq:2002'}
\partial_t\varrho_\kappa
+\Div\big(\varrho_\kappa \mathscr E_\eta \bfu \big)= \bfr_\kappa\quad \text{in}\quad I\times \setR^3,
\end{align}
where $\bfr_\kappa=\Div(\varrho_\kappa \mathscr E_\eta \bfu)-\Div(\varrho \mathscr E_\eta\bfu)_\kappa$.
Due to $\beta>2$ we can infer from the commutator lemma (see e.g. \cite[Lemma 2.3]{Li1}) that for a.e. $t$
\begin{align*}
\|\bfr_\kappa\|_{L^q(\R^3)}\leq \| \mathscr E_\eta\bfu\|_{W^{1,p}(\R^3)}\|\varrho\|_{L^{\beta}(\R^3)},\quad \tfrac{1}{q}=\tfrac{1}{p}+\tfrac{1}{\beta},
\end{align*}
as well as 
\begin{align}\label{eq:2002b'}
\bfr_\kappa\rightarrow0\quad\text{in}\quad L^{q}(\R^3).
\end{align}
a.e. in $I$. Now we multiply \eqref{eq:2002} by
$\theta'(\varrho_\kappa)$ and obtain
\begin{align}\label{eq:2002c'}
\begin{aligned}
\partial_t \theta(\varrho_\kappa)
&+\Div\big(\theta(\varrho_\kappa)\mathscr E_\eta \bfw \big)+\big(\varrho_\kappa\theta'(\varrho_\kappa)-\theta(\varrho_\kappa)\big)\Div\mathscr E_\eta \bfu= \bfr_\kappa\theta'(\varrho_\kappa).
\end{aligned}
\end{align}
Due to the properties of the mollification and $\theta\in C^1$
the terms $\theta(\varrho_\kappa)$ and $\theta'(\varrho_\kappa)$ converge to the correct limit (at least after taking a subsequence).
Hence multiplying \eqref{eq:2002c'} by $\psi\in C^\infty(\overline{I}\times \R^3)$ and integrating over $I\times\R^3$ this implies
\begin{align}\label{8.4}
\begin{aligned}
\int_{I}\partial_t&\int_{\R^3} \theta(\varrho)\,\psi\dxt-\int_{I\times\R^3}\theta(\varrho)\,\partial_t\psi\dxt
+\int_{I\times\R^3}\big(\varrho\theta'(\varrho)-\theta(\varrho)\big)\Div\mathscr E_\eta \bfu\,\psi\dxt\\
&=\int_{I\times\R^3}\theta(\varrho) \mathscr E_\eta \bfu\cdot\nabla\psi .
\end{aligned}
\end{align}

\subsection{Strong convergence of the density}
\label{sec:strongrhoep}
In order to deal with the local nature of \eqref{eq:fluxpsi} we use ideas from
\cite{Fe}. First of all, by the monotonicity of the mapping $z\mapsto az^\gamma+\delta z^\beta$, we find for arbitrary positive $\psi\in C^\infty_0(\Omega_{\epsilon_0})$
\begin{align*}
(\lambda+2\mu)&\liminf_{\varepsilon\rightarrow0}\int_{I\times\mt}\psi \big(\Div \bfu_\epsilon\,\varrho_\epsilon -\Div \bfu\,\varrho\big)\dxt\\
=&\liminf_{\varepsilon\rightarrow0}\int_{I\times\Omega_{\eta_\epsilon}}\Big(\psi\big( \overline{p}-(\lambda+2\mu)\Div \bfu\big)\varrho
-
\psi\big( a\varrho_\varepsilon^\gamma+\delta\varrho_\varepsilon^\beta-(\lambda+2\mu)\Div \bfu_\varepsilon\big)\,\varrho_\varepsilon\Big)\dxt
\\
+&\liminf_{\varepsilon\rightarrow0}\int_{I\times\Omega_{\eta_\epsilon}}\psi\big( a\varrho_\varepsilon^{\gamma+1}+\delta\varrho_\varepsilon^{\beta+1}- \overline{p}\varrho\big)\dxt
\\
=&\liminf_{\varepsilon\rightarrow0}\int_{I\times\Omega_{\eta_\epsilon}}\psi\big( a\varrho_\varepsilon^{\gamma}+\delta\varrho_\varepsilon^{\beta}- \overline{p}\big)\big(\varrho_\varepsilon-\varrho\big)\dxt\geq 0
\end{align*}
using \eqref{eq:fluxpsi}.
As $\psi$ is arbitrary we conclude
\begin{align}\label{8.12}
\overline{\Div \bfu\,\varrho}\geq \Div \bfu\,\varrho \quad\text{a.e. in }\quad I\times\Omega_\eta,
\end{align}
where 
\begin{align*}
\Div \bfu_\epsilon\,\varrho_\epsilon\weakto^\eta \overline{\Div \bfu\,\varrho}\quad\text{in}\quad L^1(\Omega;L^1(\Omega_{\eta_\varepsilon})),
\end{align*}
recall \eqref{eq:convu1} and \eqref{eq:convrho1}. Now, we compute both sides of
\eqref{8.12} by means of the corresponding continuity equations. Due to Lemma
\ref{cor:ap} with $\theta(z)=z\ln z$ and $\psi=\chi_{[0,t]}$ we have
\begin{align}\label{8.15}
\int_0^t\int_{\R^3}\Div \bfu_\epsilon\,\varrho_\epsilon\dxs\leq\int_{\R^3}\varrho_0\ln(\varrho_0)\dx
-\int_{\R^3}\varrho_\varepsilon(t)\ln(\varrho_\varepsilon(t))\dx.
\end{align}
Similarly, equation \eqref{8.4} yields
\begin{align}\label{8.14}
\int_0^t\int_{\R^3}\Div \bfu\,\varrho\dxs=\int_{\R^3}\varrho_0\ln(\varrho_0)\dx
-\int_{\R^3}\varrho(t)\ln(\varrho(t))\dx.
\end{align}
Combining \eqref{8.12}--\eqref{8.14} shows
\begin{align*}
\limsup_{\varepsilon\rightarrow0}\int_{\R^3}\varrho_\varepsilon(t)\ln(\varrho_\varepsilon(t))\dx\leq \int_{\R^3}\varrho(t)\ln(\varrho(t))\dx
\end{align*}
for any $t\in I$.
This gives the claimed convergence $\varrho_\varepsilon\rightarrow\varrho$ in $L^1(I\times\R^3)$ by convexity of $z\mapsto z\ln z$. Consequently, we have $\tilde p=a\varrho^\gamma+\delta\varrho^\beta$ and the proof of Theorem \ref{thm:ap} is complete.

\section{The vanishing artificial pressure limit}
\label{sec:6}

A weak solution to \eqref{eq1}--\eqref{initial} is a triple $(\eta,\bfu,\varrho)\in \times Y^I\times X_\eta^I\times\widetilde W_\eta^I$, where
$$W_\eta^I= C_w(\overline{I};L^\gamma(\Omega_\eta)),$$
 that satisfies the following.
\begin{enumerate}[label={(O\arabic{*})}]
\item\label{O1} The momentum equation in the sense that
\begin{align}\label{eq:apufinal}
\begin{aligned}
&\int_I\frac{\dd}{\dt}\int_{\Omega_{ \eta}}\varrho\bfu \cdot\bfphi\dx-\int_{\Omega_{\eta}} \Big(\varrho\bfu\cdot \partial_t\bfphi +\varrho\bfu\otimes \bfu:\nabla \bfphi\Big)\dxt
\\
&+\int_I\int_{\Omega_\eta}\Big(\mu\nabla\bfu:\nabla\bfphi +(\lambda+\mu)\Div\bfu\,\Div\bfphi\Big)\dxt-\int_I\int_{\Omega_{ \eta }}
a\varrho^\gamma\,\Div\bfphi\dxt\\
&+\int_I\bigg(\frac{\dd}{\dt}\int_M \partial_t \eta b\dH-\int_M \partial_t\eta\,\partial_t b\dH + \int_M K'(\eta)\,b\dH\bigg)\dt
\\&=\int_I\int_{\Omega_{ \eta}}\varrho\bff\cdot\bfphi\dxt+\int_I\int_M g\,b\,\dd x\dt
\end{aligned}
\end{align} 
for all $(b,\bfphi)\in C^\infty_0(M)\times C^\infty(\overline{I}\times\R^3)$ with $\mathrm{tr}_{\Omega}(\bfphi\circ \bfPsi_{\eta})=b\nu$. Moreover, we have $(\varrho\bfu)(0)=\bfq_0$, $\eta(0)=\eta_0$ and $\partial_t\eta(0)=\eta_1$. 
\item\label{O2} The continuity equation in the sense that 
\begin{align}\label{eq:apvarrho0final}
\begin{aligned}
&\int_I\frac{\dd}{\dt}\int_{\Omega_{\eta}}\varrho \psi\dxt-\int_I\int_{\Omega_{\eta}}\Big(\varrho\partial_t\psi
+\varrho\bfu\cdot\nabla\psi\Big)\dxt=0
\end{aligned}
\end{align}
for all $\psi\in C^\infty(\overline{I}\times\R^3)$ and we have $\varrho(0)=\varrho_0$. 
\item \label{O3} The boundary condition $\mathrm{tr}_\eta\bfu=\partial_t\eta\nu$ in the sense of Lemma \ref{lem:2.28}.
\end{enumerate}
\begin{theorem}\label{thm:final}
Let $\gamma>\frac{12}{7}$ ($\gamma>1$ in two dimensions).
There is a weak solution $(\eta,\bfu,\varrho)\in Y^I\times X_\eta^I\times W_\eta^I$ to \eqref{eq1}--\eqref{initial} in the sense of \ref{O1}--\ref{O3}. Here, we have $I= (0,T_*)$, with $T_*<T$ only in case $\Omega_\eta(s)$ approaches a self intersection with $s\to T_*$.
The solution satisfies the energy estimate
\begin{align*}
\sup_{t\in I}\int_{\Omega_{\eta}}&\varrho|\bfu|^2\dx+\sup_{t\in I}\int_{\Omega_{ \eta}}a\varrho^\gamma\dx+\int_I\int_{\Omega_{\eta}}|\nabla\bfu|^2\dxt+\sup_{t\in I}\int_M|\partial_t\eta|^2\,\dd\mathcal H^2+\sup_{t\in I}K(\eta)\\
\leq&\,c\,\bigg(\int_{\Omega}\frac{|\bfq_0|^2}{\varrho_0}\dx+\int_{\Omega}a\varrho_0^\gamma\dx+\int_I\|\bff\|_{L^\infty(\Omega_{\eta})}^2\dt+\int_I\norm{g}_{L^2(M)}\dt\bigg)\\
+&\,c\bigg(\int_M\abs{\eta_0}^2\dH+\int_M\abs{\eta_1}^2\dH+K(\eta_0)\bigg),
\end{align*}
provided that $\eta_0,\eta_1,\varrho_0,\bfq_0,\bff$ and $g$ are regular enough to give sense to the right-hand side, that $\varrho_0\geq0$ a.e. and \eqref{eq:compa} is satisfied.
\end{theorem}
\begin{lemma}
\label{cor:apb}
Under the assumptions of Theorem \ref{thm:final}, the continuity equation holds in the renormalized sense that is
\begin{align}\label{eq:final3b}
\begin{aligned}
\int_I\frac{\dd}{\dt}\int_{\Omega_\eta}\theta(\varrho)\psi\dxt&-\int_I\int_{\Omega_\eta}\Big(\theta(\varrho)\partial_t\psi +\theta(\varrho) \bfu\cdot \nabla \psi\Big)\dxt
\\
& =- \int_I\int_{\Omega_\eta}(\varrho\theta'(\varrho)-\theta(\varrho)) \diver\bfu \,\psi\dxt
\end{aligned}
\end{align}
for all $\psi\in C^\infty(\overline{I}\times\R^3)$ and all $\theta\in C^1(\R)$ with 
$\theta'(z)=0$ for $z\geq M_\theta$.
\end{lemma}

For a given $\delta$ we gain a weak solutions $(\eta_\delta,\bfu_\delta,\varrho_\delta)$ to \eqref{eq:apu}--\eqref{eq:apvarrho} by Theorem \ref{thm:ap}. It is defined in the interval $[0,T_*]$, where $T_*$ is restricted by the data alone.
The estimate from Theorem \ref{thm:ap} holds uniformly with respect to $\delta$. 
Hence we may take a subsequence, such that
by Lemma~\ref{thm:weakstrong}, we find for $s\in (1,2)$ and for some $\alpha\in (0,1)$ fixed subsequences which satisfy
\begin{align}
\eta_\delta&\rightharpoonup^\ast\eta\quad\text{in}\quad L^\infty(I;W_0^{2,2}(M))\label{eq:conveta}\\
\eta_\delta&\rightharpoonup^\ast\eta\quad\text{in}\quad W^{1,\infty}(I;L^2(M)),
\label{eq:conetat}\\
\eta_\delta&\to\eta\quad\text{in}\quad C^\alpha(\overline{I}\times M),
\label{eq:conetata}\\
\bfu_\delta&\rightharpoonup^\eta\bfu\quad\text{in}\quad L^2(I;W^{1,2}(\Omega_{\eta_\delta})),\label{eq:convu}\\
\varrho_\delta&\rightharpoonup^{\ast,\eta}\varrho\quad\text{in}\quad L^\infty(I;L^\gamma(\Omega_{\eta_\delta})).\label{eq:convrho}
\end{align}
By Lemma~\ref{thm:weakstrong} we find for $q\in (1,\frac{6\gamma}{\gamma+6})$, that
\begin{align}
\varrho_\delta\bu_\delta&\rightharpoonup^\eta  {\varrho}  {\bfu}\qquad\text{in}\qquad L^2(I, L^q(\Omega_{\eta_\delta}))\label{conv:rhov2delta}\\
{\varrho}_\delta  {\bfu}_\delta\otimes  {\bfu}_\delta&\rightharpoonup^\eta  {\varrho}  {\bfu}\otimes  {\bfu}\qquad\text{in}\qquad L^1(I\times\Omega_{\eta_\delta}).\label{conv:rhovv2delta}
\end{align}

Also we have, as before in Proposition \ref{prop:higher}, higher integrability of the density.
\begin{lemma}\label{prop:higher'}
Let $\gamma>\frac{3}{2}$ ($\gamma>1$ in two dimensions).
Let $Q=J\times B\Subset I\times\Omega_\eta$ be a parabolic cube and $0<\Theta\leq\frac{2}{3}\gamma-1$. The following holds for any $\delta\leq \delta_0(Q)$
\begin{equation}\label{eq:gamma+1'}
\int_{Q}\big(a\varrho_\delta^{\gamma+\Theta}
+\delta\varrho_\delta^{\beta+\Theta}\big)\,\dif x\,\dif t\leq C(Q)
\end{equation}
with constant independent of $\delta$.
\end{lemma}
\begin{proof}
The proof follows the lines of Lemma~\ref{prop:higher} with the difference that we test with $\psi\nabla\Delta^{-1}_{\tilde{B}}\varrho_\delta^\Theta$. We only show how to handle the most critical integral
$$J=\int_Q \psi\varrho_\delta\bfu_\delta\otimes\bfu_\delta:\nabla^2\Delta^{-1}_{\tilde{B}}\varrho_\delta^\Theta\dxt$$
arising from the convective term.
 The bound $\Theta\leq\frac{2}{3}\gamma-1$ is needed to estimate it.
It can be estimated using the continuity of $\nabla^2\Delta^{-1}_{\tilde{B}}$ and H\"older's inequality by
\begin{align*}
|J|&\leq\,c\int_0^T\| \varrho_\delta\|_\gamma\| \bfu_\delta\|_6^2\| \varrho^{\theta}_\delta\|_r\dt,
\end{align*}
where $r:=\frac{3\gamma}{2\gamma-3}$. We proceed, using Sobolev's inequality (note that $\bfu_\delta=0$ on $\Gamma$), by
\begin{align*}
|J|&\leq\,C\,\Big(\sup_{0\leq t\leq T}\| \varrho_\delta\|_\gamma\Big)\Big(\sup_{0\leq t\leq T}\| \varrho_\delta^{\Theta}\|_r\Big)\int_0^T\|\nabla \bfu_\delta\|_2^2\dt.
\end{align*}
We need to choose $r$ such that $\Theta r\leq \gamma$ which is equivalent to $\Theta\leq \frac{2}{3}\gamma-1$. Now, the various a-priori bounds yield $|J|\leq\,c$ uniformly in $\delta$.
\end{proof}

Similar to Lemma~\ref{prop:higherb} we can exclude concentrations of the pressure at the moving boundary. However, we have to assume $\gamma>\frac{12}{7}$ for this.
\begin{lemma}\label{prop:higherb'}
Let $\gamma>\frac{12}{7}$ ($\gamma>1$ in two dimensions).
Let $\kappa>0$ be arbitrary. There is a measurable set $A_\kappa\Subset I\times\Omega_\eta$ such that
\begin{equation}\label{eq:gamma+1b'}
\int_{I\times\R^3\setminus A_\kappa}\big(a\varrho_\delta^{\gamma}
+\delta\varrho_\delta^{\beta}\big)\,\chi_{\Omega_{\eta_\delta}}\dif x\,\dif t\leq \kappa.
\end{equation}
\end{lemma}
\begin{proof}
We follow the approach of Proposition \ref{prop:higherb} replacing $\varepsilon$
by $\delta$. So we test with
\[
\bfphi_\delta(t,x)=\phi\min\set{K(s(x)-\eta_\delta(t,q(x))),1}\nu(q(x)).
\]
The critical term is again
is
\begin{align}\label{eq:1712'}
\int_I\int_{\Omega_{\eta_\delta}} \varrho_\delta\bfu_\delta\,\partial_t\bfphi_\delta\dxt.
\end{align}
Following the proof of Proposition \ref{prop:higherb} it can be estimated provided $\gamma>3$. We want to improve this. In order to do so we write
\begin{align*}
\partial_t\bfphi_\delta&=-K\chi_{\set{K(s(x)-\eta_\delta(t,q(x)))\leq 1}}\partial_t\eta_\delta(t,q(x))\nu(q(x))\\&=-K\chi_{\set{K(s(x)-\eta_\delta(t,q(x)))\leq 1}}\bfu_\delta\circ\bfPsi_\delta(t,0,q(x)).
\end{align*}
By Lemma~\ref{lem:2.28} and since $\nabla \bfu_\delta$ is uniformly bounded in $L^2$ (recall \eqref{eq:convu}) we find that
\begin{align}\label{eq:51b}
\bfu_\delta\circ\bfPsi_\delta|_{\partial\Omega}\in L^2(I;L^q(\partial\Omega))\quad\forall q<4
\end{align}
uniformly in $\delta$ ($q<\infty$ in two dimensions). Similar to \eqref{eq:51b1} we obtain
\begin{align}\label{eq:51b1'}
\begin{aligned}
\bigg(&\int_I\bigg(\int_{\Omega_{\eta_\delta}}|\partial_t\bfphi_\delta|^r\dx\bigg)^{\frac{2}{r}}\bigg)^{\frac{1}{2}}\\&\leq \,cK\bigg(\int_I\bigg(\int_{\Omega_{\eta_\delta}}|\bfu_\delta\circ\bfPsi_\delta(t,0,q(x))|^q\dx\bigg)^{\frac{2}{q}}\big|\set{K(s-\eta_\delta(t,q))\leq 1}\big|^{\frac{2(q-r)}{qr}}\dt\bigg)^{\frac{1}{2}}\\
&\leq \,cK\bigg(\int_I\bigg(\int_{\partial\Omega_{\eta_\delta}}|\bfu_\delta\circ\bfPsi_\delta(t,q(x))|^q\dH\bigg)^{\frac{2}{q}}\dt\bigg)^{\frac{1}{2}}\sup_I\big|\set{K(s-\eta_\delta(t,q))\leq 1}\big|^{\frac{q-r}{qr}}\\
&\leq \,c\,K^{1-\frac{q-r}{qr}}
\end{aligned}
\end{align} 
for all $r<q<4$ (all $r<q<\infty$ in two dimensions) uniformly in $\delta$.
Now, the proof can be finished as in Proposition \ref{prop:higherb}.
%
We take
\begin{equation}\label{eq:L2'}
\begin{split}
\varrho_\delta\bu_\delta\in L^2(I;L^{\frac{6\gamma}{\gamma+6}}(\mt))
\end{split}
\end{equation}
into account (which follows from the uniform a-priori bounds)
 and $\gamma>\frac{12}{7}$ (which yields $\frac{6\gamma}{\gamma+6}>\frac{4}{3}$). We see that the integral
in \eqref{eq:1712'} is uniformly bounded by $K^{1-\lambda}$ for some $\lambda>0$ using H\"older's inequality \eqref{eq:51b1'} (choosing $r$ and $q$ appropriately).
\end{proof}
Lemma~\ref{prop:higher'} and Lemma~\ref{prop:higherb'} imply equicontinuity of the sequence  $\varrho_\delta^\gamma\chi_{\Omega_{\eta_\delta}}$. This yields the existence of a function $\overline p$ such that a subsequence satisfies 
\begin{align}\label{eq:limp'}
a\varrho^\gamma_\delta+\delta\varrho^\beta_\delta\rightharpoonup\overline p\quad\text{in}\quad L^{1}(I\times\R^3),\\
\label{1301}
\delta\varrho_\delta^{\beta}\rightarrow0\quad\text{in}\quad L^1(I\times\R^3).
\end{align}
Similarly to Proposition \ref{prop:higherb0} we have the following Proposition. 
\begin{proposition}\label{prop:higherb0'}
Let $\kappa>0$ be arbitrary. There is a measurable set $A_\kappa\Subset I\times\Omega_\eta$ such that
\begin{equation}\label{eq:gamma+1b0'}
\int_{I\times\R^3\setminus A_\kappa}\overline p\,\dif x\,\dif t\leq \kappa.
\end{equation}
\end{proposition}

Using \eqref{eq:limp'} and the convergences \eqref{eq:conveta}--\eqref{conv:rhovv2delta} we can pass to the limit in \eqref{eq:apu}--\eqref{eq:apvarrho} and obtain
\begin{align}\label{eq:apulim'}
\begin{aligned}
&-\int_I\int_{\Omega_{ \eta}}\varrho\bfu\cdot\partial_t\bfphi\dxt-\int_I\int_M\varrho\partial_t\eta\,\partial_t\eta\,b\,\gamma(\eta)\,\dd\mathcal H^2\dt\\
&+\int_I\int_{\Omega_{ \eta}}\Div(\varrho\bfu\otimes\bfu)\cdot\bfphi\dxt+\mu\int_I\int_{\Omega_{ \eta}}\nabla\bfu:\nabla\bfphi\dxt\\&+(\lambda+\mu)\int_I\int_{\Omega_{ \eta}}\Div\bfu\,\Div\bfphi\dxt-\int_I\int_{\Omega_{\eta}} \overline p\,\Div\bfphi\dxt
\\&
-\int_I\int_M \partial_t\eta\,\partial_t b\,\dd\mathcal H^2\dt+2\int_I \int_M K'(\eta)\,b\dH\dt
\\&=\int_I\int_{\Omega_{ \eta}}\varrho\bff\cdot\bfphi\dxt+\int_I\int_M g\,b\,\dd\mathcal H^2\dt+\int_{\Omega_{\eta_0}}\bfu_0\cdot\bfphi(0,\cdot)\dx+\int_M\eta_0\,b\,\dd\mathcal H^2
\end{aligned}
\end{align} 
for all test-functions $(b,\bfphi)$ with $\mathrm{tr}_\eta\bfphi=\partial_t\eta\nu$, $\bfphi(T,\cdot)=0$ and $b(T,\cdot)=0$. Moreover, the following holds
\begin{align}\label{eq:apvarrholim}
\int_I\int_{\Omega_{\eta}}\varrho\,\partial_t\psi\dxt-\int_I\int_{\Omega_{\eta}}\Div(\varrho\,\bfu)\,\psi\dxt=\int_{\Omega_{\eta_0}}\varrho_0\,\psi(0,\cdot)\dx
\end{align}
for all $\psi\in C^\infty(\overline{I\times\Omega_\eta})$.
It remains to show that $\overline{p}=a\varrho^\gamma$.

\subsection{The effective viscous flux}
\label{subsec:strongconvdensity'}
We define the $L^\infty$-truncation
\begin{align}\label{eq:Tk'}
T_k(z):=k\,T\Big(\frac{z}{k}\Big)\quad z\in\mr,\,\, k\in\N.
\end{align}
Here $T$ is a smooth concave function on $\mr$ such that $T(z)=z$ for $z\leq 1$ and $T(z)=2$ for $z\geq3$.
It is the aim of this subsection to show that 
\begin{align}\label{eq:flux'}
\begin{aligned}
\int_{I\times\Omega_{\eta_\delta}}&\big( a\varrho_\delta^\gamma+\delta\varrho_\delta^\beta-(\lambda+2\mu)\Div \bfu_\delta\big)\,T_k(\varrho_\delta)\dxt\\&\longrightarrow\int_{I\times\Omega_{\eta}} \big( \overline{p}-(\lambda+2\mu)\Div \bfu\big)\,T^{1,k}\dxt.
\end{aligned}
\end{align}
For this stepp we are able to use the theory established in \cite{Li1,Li2} on a local level. We fix a small $\delta_0$ and consider an arbitrary cube $\tilde Q=\tilde J\times\tilde B\Subset I\times\bigcup_{\delta\in [0,\delta_0]} \Omega_{\eta_\delta}$.
 To this end, we can choose $\theta=T_k$ in the renormalized continuity equation for $\tilde\varrho_\delta$, cf. Lemma \ref{cor:ap1}. Hence we find
\begin{align*}
\partial_t T_k(\varrho_\delta)+\Div\big(T_k(\varrho_\delta)\bfu_\delta\big)+\big(T_k'(\varrho_\delta)\varrho_\delta-T_k(\varrho_\delta)\big)\Div\bfu_\delta=0
\end{align*}
in the sense of distributions in $I\times\R^3$. In order to pass to the limit in this equation, let $T^{1,k}$ denote the weak limit of $T_k(\varrho_\delta)$ and let $T^{2,k}$ denote the weak limit of $\big(T_k'(\varrho_\delta)\varrho_\delta-T_k(\varrho_\delta)\big)\Div\bfu_\delta$ (here it might be necessary to pass to a subsequence).
To be more precise, the following holds 
\begin{align}
 T_k(\varrho_\delta)&\weakto {T}^{1,k}\quad\text{in}\quad C_w(I;L^p( \R^3))\quad\forall p\in[1,\infty),\label{eq:Tk1'}\\
\big(T_k'(\varrho_\delta)\varrho_\delta-T_k(\varrho_\delta)\big)\Div\bfu_\delta&\rightharpoonup{T}^{2,k}
\quad\text{in}\quad L^2(I\times\R^3).\label{eq:Tk2'}
\end{align}
So letting $\delta\rightarrow0$ yields
\begin{align}\label{eq:Tkdis'}
\partial_t T^{1,k}+\Div\big( T^{1,k}\bfu\big)+T^{2,k}=0
\end{align}
in the sense of distributions in $I\times\R^3$.
Here we used that $\beta>3$ and that
\begin{align}\label{1311c}
\begin{aligned}
 T_k(\varrho_\delta)&\rightarrow T^{1,k}\quad\text{in}\quad L^2(I;W^{-1,2}(\R^3)),\\
\bfu_\delta&\rightharpoonup\bfu\quad\text{in}\quad L^2(I;W^{1,2}(\R^3)),
\end{aligned}
\end{align}
which is a consequence of \eqref{eq:Tk1'} (with $p>\frac{6}{5}$).

Next we take $Q$ with $Q\Subset\tilde Q\Subset I\times\Omega_{\eta_n}$ and the cut off function $\psi\in C^\infty_0(\tilde Q)$ with $0\leq\psi\leq 1$ and $\psi\equiv1$ on $Q=J\times B$. Now, we test \eqref{eq:apu} with $\psi\nabla\Delta^{-1}(\psi T^k(\varrho_\delta))$ and \eqref{eq:apulim'} with $\psi\nabla\Delta^{-1}(\psi T^{1,k})$.
Using the similar argument as in Subsection~\ref{subsec:strongconvdensity} we find that
\begin{align}\label{eq:fluxpsi'}
\begin{aligned}
\int_{I\times\mt} &\psi^2\big( a\varrho_\delta^\gamma-(\lambda+2\mu)\Div \bfu_\delta\big)\,T_k(\varrho_\delta)\dxt\\&\longrightarrow\int_{I\times\mt}\psi^2 \big( \overline{p}-(\lambda+2\mu)\Div \bfu\big)\,T^{1,k}\dxt.
\end{aligned}
\end{align}
We have to remove $\psi$ in order to conclude.
For some given $\kappa>0$ we choose a measurable set  
in accordance to Lemma~\ref{prop:higherb'} and Corollary~\ref{prop:higherb0'} for $\delta_0$ small enough (using the fact that $\eta_\delta\rightrightarrows\eta$). With no loss of generality we can assume that $\partial A_\kappa$ is regular. Hence we can cover $A_\kappa$ with parabolic cubes $Q_i=J_i\times B_i$ such that
\begin{align*}
A_\kappa\subset  \bigcup_i Q_i\Subset \bigcap_{\delta\in [0,\delta_0]}\big(I\times\Omega_{\eta_\delta}\big).
\end{align*}
They can be chosen in such a way that we find $\psi_i$ a partition of unity for the family $Q_i$ such that $\psi_i\in C_0^\infty(Q_i)$ and 
\begin{align*}
\sum \psi_i =1 \text{ on }A_\kappa.
\end{align*}
In particular \eqref{eq:fluxpsi'} holds with $\psi=\psi_i$. We gain
\begin{align*}
\int_{I\times\Omega_{\eta_\delta}}&\big( \varrho_\delta^\gamma+\delta\varrho_\delta^\beta-(\lambda+2\mu)\Div \bfu_\delta\big)\,T_k(\varrho_\delta)\dxt
\\
&=
\int_{I\times\Omega_{\eta_\delta}}\Big(1-\sum_i \psi_i\Big) \varrho_\delta^\gamma+\delta\varrho_\delta^\beta-(\lambda+2\mu)\Div \bfu_\delta\big)\,T_k(\varrho_\delta)\dxt\\
&+\sum_i\int_{Q_i}\psi_i\big( \varrho_\delta^\gamma+\delta\varrho_\delta^\beta-(\lambda+2\mu)\Div \bfu_\delta\big)\,T_k(\varrho_\delta)\dxt.
\end{align*}
Using \eqref{eq:gamma+1b'} the first integral on the right-hand side is bounded by $\kappa$. Using \eqref{eq:fluxpsi'} and \eqref{eq:gamma+1b0'} we find that
\begin{align*}
\lim_{\delta\to 0}
\biggabs{\int_{I\times\Omega_{\eta_\delta}}\big( \varrho_\delta^\gamma+\delta\varrho_\delta^\beta-(\lambda+2\mu)\Div \bfu_\delta\big)\,T_k(\varrho_\delta)\dxt-\int_{I\times\Omega_{\eta}}\big( \overline{p}(\lambda+2\mu)\Div \bfu\big)\,T^{1,k}\dxt}
\end{align*}
is bounded by $\kappa$. As $\kappa$
is arbitrary
we finally conclude that \eqref{eq:flux'} is satisfied.

\subsection{Renormalized solutions}
The aim of this section is to prove Lemma \ref{cor:apb}.
In order to do so it suffices to use the continuity equation and \eqref{eq:flux'} again on the whole space.

First observe that since $\varrho_\delta$ is renormalized solution to the continuity equation by  
\eqref{eq:final3}, i.e. we have
 we find that 
\begin{align}\label{eq:Tk0''}
\begin{aligned}
\partial_t \theta(\varrho_\delta)+\Div\big(\theta(\varrho_\delta)\bfu_\delta\big)&+\big(\theta'(\varrho_\delta)\varrho_\delta
-\theta(\varrho_\delta)\big)\Div\bfu_\delta=0
\end{aligned}
\end{align}
in the sense of distributions on $I\times\R^3$. Note that \eqref{eq:Tk0''} holds in particular for $\theta(z)=z$, which implies that the continuity equation can be regarded as a PDE on the whole-space.
We are interested in the particular choice $\theta=T_k$, where the cut-off functions $T_k$ are given by \eqref{eq:Tk'}.\\
We have to show that, similar to \eqref{eq:Tk0''}, equation \eqref{eq:Tkdis'} actually holds globally. 
Note that we can improve \eqref{eq:Tk1'} and \eqref{eq:Tk2'} in the sense that
To be more precise, the following holds
\begin{align}
 T_k(\varrho_\delta)&\rightarrow {T}^{1,k}\quad\text{in}\quad C_w(\overline I;L^p(\R^3))\quad\forall p\in[1,\infty),\label{eq:Tk1b'}\\
\big(T_k'(\varrho_\delta)\varrho_\delta-T_k(\varrho_\delta)\big)\Div\bfu_\delta&\rightharpoonup{T}^{2,k}
\quad\text{in}\quad L^2(I\times\R^3)\label{eq:Tk2b'}.
\end{align}
The first convergence is a consequence of \eqref{eq:Tk0''} with $\theta=T_k$ and the uniform $L^\infty$ bounds of $T_k$, the second convergence follows just by the usual compactness in Lebesgue spaces.
Due to compactness of the embedding 
$C_w(\overline I;L^p(\mathcal O))\hookrightarrow L^2(I;W^{1,2}(\mathcal O))$ for $\mathcal O\Subset\R^3$
we infer from \eqref{eq:Tk1b'} 
\begin{align*}
T_k(\varrho_\delta)\Div\bfu_\delta&\rightharpoonup{T}^{1,k}\Div\bfu
\quad\text{in}\quad L^2(I\times\R^3).
\end{align*}
So, choosing $\theta=T_k$ in \eqref{eq:Tk0''} letting $\delta\rightarrow0$ yields
\begin{align}\label{eq:rendelta'}
\begin{aligned}
0=&\frac{\dd}{\dt}\int_{\setR^3}T^{1,k}\psi\dx-\int_{\setR^3}\big(T^{1,k}\partial_t\psi +T^{1,k} \bfu\cdot \nabla \psi\big)\dx
\\
&\quad + \int_{\setR^3}T^{2,k}\psi\dx
\end{aligned}
\end{align}
for all $\psi\in C^\infty(\overline I\times \setR^3)$.
This means that we have
\begin{align}\label{eq:Tk''}
\partial_t T^{1,k}+\Div\big( T^{1,k}\bfu\big)+T^{2,k}= 0
\end{align}
in the sense of distributions on $I\times\R^3$. Note that we extended
$\varrho$ by zero to $\R^3$.
The next step is to show
\begin{align}\label{eq.amplosc''}
\limsup_{\delta\rightarrow0}\int_{I\times\mt}|T_k(\varrho_\delta)-T_k(\varrho)|^{\gamma+1}\dxt\leq C,
\end{align}
where $C$ does not depend on $k$. The proof of \eqref{eq.amplosc''} follows exactly the arguments from the classical setting with fixed boundary
 (see \cite[Lemma 4.4]{feireisl1} and \cite{feireisl2}) using \eqref{eq:flux'} and the uniform bounds on $\bfu$. We explain the details for the convenience of the reader. First, not that we have
\begin{align*}
\lim_{\delta\rightarrow0}&\int_I\int_{\R^3}\Big((\varrho_\delta^\gamma+\delta\varrho_\delta^\beta) T_k(\varrho_\delta)-\overline{p}T^{1,k}\Big)\dxt\\
&=\lim_{\delta\rightarrow0}\bigg(\int_I\int_{\R^3}\Big(\varrho_\delta^\gamma+\delta\varrho_\delta^\beta -(\varrho^\gamma+\delta\varrho^\beta)\Big)\big(T_k(\varrho_\delta)-T^{1,k}\big)\dxt\\&+\int_I\int_{\R^3}\big(\overline{p}-(\varrho^\gamma+\delta\varrho^\beta)\big)\big( T_k(\varrho)-T^{1,k}\big)\dxt\bigg).
\end{align*}
By convexity of $z\mapsto z^\gamma+\delta z^\beta$ we conclude that
\begin{align}\nonumber
\lim_{\delta\rightarrow0}\int_I\int_{\R^3}\big((\varrho_\delta^\gamma+\delta\varrho_\delta^\beta) T_k(\varrho_\delta)-\overline{p}T^{1,k}\big)\dxt&\geq\lim_{\delta\rightarrow0}
\int_I\int_{\R^3}\big(\varrho_\delta^\gamma -\varrho^\gamma\big)\big(T_k(\varrho_\delta)-T^{1,k}\big)\dxt\\
&\geq \int_I\int_{\R^3}|T_k(\varrho_\delta)-T_k(\varrho)|^{\gamma+1}\dxt.
\label{eq:2601a'}
\end{align}
Moreover, we have
\begin{align}
\nonumber
\limsup_{\delta\rightarrow0}\int_I\int_{\R^3}\big(\Div\bfu_\delta \,T_k(\varrho_\delta)-\Div\bfu \,T^{1,k}\big)\dxt&=\limsup_{\delta\rightarrow0}\int_I\int_{\R^3}\big( \,T_k(\varrho_\delta)-T^{1,k}\big)\Div\bfu_\delta\dxt\\
\label{eq:2601b'}&\leq \,c\,\limsup_{\delta\rightarrow0}\|T_k(\varrho_\delta)-T^{1,k}\|_{L^2(I\times \R^3)}\\
&\leq \,c\,\limsup_{\delta\rightarrow0}\|T_k(\varrho_\delta)-T_k(\varrho)\|_{L^2(I\times\R^3)}.\nonumber
\end{align}
Now we combine \eqref{eq:2601a'} end  \eqref{eq:2601b'} with \eqref{eq:flux'} to conclude \eqref{eq.amplosc''}.\\
By a standard smoothing procedure we can consider "renormalized solutions" for $T^{1,k}$ and deduce from (\ref{eq:Tk''}) that
\begin{align}\label{eq:Tkren'}
\partial_t \theta(T^{1,k})+\Div\big(\theta(T^{1,k})\bfu\big)
+\big(\theta'(T^{1,k})T^{1,k}-\theta(T^{1,k})\big)\Div\bfu+\theta'(T^{1,k})T^{2,k}=0
\end{align}
in the sense of distributions $I\times \R^3$. Where we use that $\theta'(z)=0$ for $z\geq M_\theta$.
We want to pass to the limit $k\rightarrow\infty$. On account of (\ref{eq.amplosc''}) we have
for all $p\in(1,\gamma)$
\begin{align*}
\|T^{1,k}-\varrho\|_{L^p(I\times\mt)}^p&\leq \liminf_{\delta\rightarrow0}\|T_k(\varrho_\delta)-\varrho_\delta\|_{L^p( I\times\mt)}^p\\
&\leq 2^p\liminf_{\delta\rightarrow0}\int_{[|\varrho_\delta|\geq k]}|\varrho_\delta|^p\dxt\\\
&\leq 2^pk^{p-\gamma}\liminf_{\delta\rightarrow0}\int_{I\times\mt}|\varrho_\delta|^\gamma\dxt\longrightarrow 0,\quad k\rightarrow\infty.
\end{align*}
So we have
\begin{align}
T^{1,k}\rightarrow\varrho\quad\text{in}\quad L^p(I\times\mt)\label{eq:T1klim'}
\end{align}
as $k\rightarrow\infty$.
Therefore, we are left to show that 
\begin{align}
\theta'(T^{1,k})T^{2,k}\rightarrow0\quad\text{in}\quad L^1(I\times\mt)\text{ with }k\to \infty.\label{eq:T2klim'}
\end{align}
Recall that $\theta$ has to satisfy $\theta'(z)=0$ for all $z\geq M$ for some $M=M(\theta)$.
We define
\begin{align*}
Q_{k,M}:=\big\{(t,x)\in I\times\mt;\;T^{1,k}\leq M\big\}
\end{align*}
and gain by weak lower semicontinuety that
\begin{align*}
\begin{aligned}
\int_{I\times\mt}&|\theta'(T^{1,k})T^{2,k}|\dxt\leq \sup_{z\leq M}|\theta'(z)|\int_Q\chi_{Q_{k,M}}|T^{2,k}|\dxt\\
&\leq \,C\,\liminf_{\delta\rightarrow0}\int_{I\times\mt}\chi_{Q_{k,M}}\big|(T_k'(\varrho_\delta)\varrho_\delta-T_k(\varrho_\delta))\Div\bfu_\delta\big|\dxt\\
&\leq\,C\,\sup_\delta\|\Div\bfu_\delta\|_{L^2(I\times\mt)}\liminf_{\delta\rightarrow0}\|T_k'(\varrho_\delta)\varrho_\delta-T_k(\varrho_\delta)\|_{L^2(Q_{k,M})}.
\end{aligned}
\end{align*}
It follows from interpolation that
\begin{align}\label{eq:425'}
\begin{aligned}
&\|T_k'(\varrho_\delta)\varrho_\delta-T_k(\varrho_\delta)\|^2_{L^2(Q_{k,M})}\\&\qquad\qquad\leq \|T_k'(\varrho_\delta)\varrho_\delta-T_k(\varrho_\delta)\|_{L^1(I\times\mt)}^\alpha\|T_k'(\varrho_\delta)\varrho_\delta-T_k(\varrho_\delta)\|_{L^{\gamma+1}(Q_{k,M})}^{(1-\alpha)(\gamma+1)},
\end{aligned}
\end{align}
where $\alpha=\frac{\gamma-1}{\gamma}$. Moreover, we can show similarly to the proof of
\eqref{eq:T1klim'}
\begin{align}\label{eq:426'}
\begin{aligned}
\|T_k'(\varrho_\delta)\varrho_\delta-T_k(\varrho_\delta)\|_{L^1(I\times\mt)}
&\leq \,C\,k^{1-\gamma}\sup_\delta\int_{I\times\mt}|\varrho_\delta|^\gamma\dxt\\
&\longrightarrow 0,\quad k\rightarrow\infty.
\end{aligned}
\end{align}
So it is enough to prove
\begin{align}\label{eq:427'}
\sup_\delta\|T_k'(\varrho_\delta)\varrho_\delta-T_k(\varrho_\delta)\|_{L^{\gamma+1}(Q_{k,M})}\leq C,
\end{align}
independently of $k$. As $T_k'(z)z\leq T_k(z)$ there holds by the definition of $Q_{k,M}$
\begin{align*}
&\|T_k'(\varrho_\delta)\varrho_\delta-T_k(\varrho_\delta)\|_{L^{\gamma+1}(Q_{k,M})}\\&\leq \,2\Big(\|T_k(\varrho_\delta)-T_k(\varrho)\|_{L^{\gamma+1}(I\times\mt)}+\|T_k(\varrho_\delta)\|_{L^{\gamma+1}(Q_{k,M})}\Big)\\
&\leq \,2\Big(\|T_k({\varrho}_\delta)-T_k(\varrho)\|_{L^{\gamma+1}(I\times\mt)}+\|T_k(\varrho_\delta)-T^{1,k}\|_{L^{\gamma+1}(I\times\mt)}+\|T^{1,k}\|_{L^{\gamma+1}(Q_{k,M})}\Big).\\
&\leq \,2\Big(\|T_k(\varrho_\delta)-T_k(\varrho)\|_{L^{\gamma+1}(I\times\mt)}+\|T_k(\varrho_\delta)-T^{1,k}\|_{L^{\gamma+1}(I\times\mt)}\Big)+CM.
\end{align*}
Firstly we find that (\ref{eq.amplosc''}) and \eqref{eq:Tk1'} imply \eqref{eq:427'}. Hence secondly \eqref{eq:425'}--\eqref{eq:427'} imply \eqref{eq:T2klim'}. So we can pass to the limit in (\ref{eq:Tkren'}) and gain
\begin{align}\label{eq:ren'}
\partial_t \theta(\varrho)+\Div\big(\theta(\varrho)\bfu\big)
+\big(\theta'(\varrho)\varrho-\theta(\varrho)\big)\Div\bfu=0
\end{align}
in the sense of distributions on $I\times \R^3$ and. The proof of Lemma~\ref{cor:apb} is complete. \\

\subsection{Strong convergence of the density}
\label{sec:strongrhoep'}

We introduce the functions $L_k$ by
\begin{align*}
L_k(z)=\begin{cases}z\ln z,&0\leq z< k\\
z\ln k+z\,\int_k^z T_k(s)/s^2\,ds,&z\geq k
\end{cases}
\end{align*}
We can choose $\theta=L_k$ in \eqref{eq:ren'} such that
\begin{align}\label{eq:1512'}
\partial_t L_k(\varrho)+\Div\big(L_k(\varrho)\bfu\big)
+T_k(\varrho)\Div\bfu=0
\end{align}
in the sense of distributions on $\varphi\in I\times\R^3$.
We also have that
\begin{align*}
\partial_t L_k(\varrho_\delta)+\Div\big(L_k(\varrho_\delta)\bfu_\delta\big)
+T_k(\varrho_\delta)\Div\bfu_\delta =0
\end{align*}
in the sense of distributions, cf. \eqref{eq:Tk0''}.

Using the testfunction $\psi\equiv 1$ on both equations implies
\begin{align}\label{eq:reneps'''}
\begin{aligned}
\int_{\R^3}L_k(\varrho_\delta)\dx&-\int_{\R^3}T_k(\varrho_\delta(0))\varphi(0)\dx
-
\int_0^t\int_{\R^3}T_k(\varrho_\delta)\Div\bfu_\delta\,\dxs\leq 0.
\end{aligned}
\end{align}
and
\begin{align}\label{eq:renb'}
\begin{aligned}
\int_{\R^3}L_k(\varrho)\dx&-\int_{\R^3}L_k(\varrho(0))\varphi(0)\dx
-
\int_0^t\int_{\R^3}T_k(\varrho)\Div\bfu\,\dxs=0.
\end{aligned}
\end{align}
The difference of both equations reads as 
\begin{align*}
\int_{\mathbb R^3}\big(L_k(\varrho_\delta)(t)-L_k(\varrho)(t)\big)\,\dx&\leq\int_{\mathbb R^3}\big(L_k(\varrho_\delta)(0)-L_k(\varrho)(0)\big)\,\dx\\
&+\int_0^t\int_{\mathbb R^3}\big(T_k(\varrho)\Div\bfu-T_k(\varrho)\Div\bfu\big)\,\dxs.
\end{align*}
We have the following convergences for all $p\in(1,\gamma)$ 
\begin{align*}
L_k(\varrho_\delta)\rightarrow L^{1,k}\quad\text{in}\quad C_w(\overline I;L^p(\R^3)),\quad\delta\rightarrow0,\\
\varrho_\delta\ln(\varrho_\delta)\rightarrow L^{2,k}\quad\text{in}\quad C_w(\overline I;L^p(\R^3)),\quad\delta\rightarrow0,
\end{align*}
which is a consequence of the fundamental theorem on Young measures (see, for instance, \cite[Thm. 4.2.1, Cor. 4.2.19]{MNRR}) and the convergence of ${\varrho}_\delta$ in $C_w(I;L^\beta(\R^3))$. The latter one follows from the a-priori information on $\varrho$ in combination with the control over the distributional time derivative of $\varrho_\delta$ coming from the continuity equation (considered on the whole-space).
So we gain (using also the fact, that $\varrho_\delta(0)=\varrho_0=\varrho(0)$) 
\begin{align}\label{1611'}
\begin{aligned}
\int_{\mathbb R^3}\big(L^{1,k}(t)-L_k(\varrho)(t)\big)\dx
&\leq\limsup_{\delta\to 0}\int_0^t\int_{\mathbb R^3}\big(T_k(\varrho)\Div\bfu-T_k(\varrho_\delta)\Div\bfu_\delta\big)\dxs\\
&\leq\int_0^t\int_{\mathbb R^3}\big(T_k(\varrho)-T^{1,k}\big)\Div\bfu\,\dxs,
\end{aligned}
\end{align}
cf. \eqref{eq:Tk1b'}.
Due to (\ref{eq:T1klim'}) the right-hand side tends to zero if $k\rightarrow\infty$ such that
\begin{align*}
\lim_{k\rightarrow\infty}\int_{\mathbb R^3}\big(L^{1,k}(t)-L_k(\varrho)(t)\big)\dx\leq0.
\end{align*}
So, we have shown
\begin{align*}
\lim_{\delta\rightarrow0}\int_I\int_{\mathbb R^3} \varrho_\delta\ln\varrho_\delta\dxt\leq\int_I\int_{\mathbb R^3} \varrho\ln\varrho\dxt.
\end{align*}
By weak lower semi-continuity for convex functionals we the converse inequality holds as well.
This finally means that
\begin{align*}
\int_I\int_{\mathbb R^3} \varrho_\delta\ln\varrho_\delta\dxt\longrightarrow\int_I\int_{\mathbb R^3} \varrho\ln\varrho\dxt.
\end{align*}
Convexity of $z\mapsto z\ln z$ yields strong convergence of $\varrho_\delta$.
Hence, due to \eqref{eq:apulim'}, the proof of Theorem \ref{thm:final} is shown, for the time interval $[0,T_*]$, with $T_*$ depending on the data alone. In the next section we will show, how the interval of existence can be prolongated by a change of coordinates.
\subsection{Maximal interval of existence}
The interval of existence in Theorem~\ref{thm:final} is restricted by the quantities of the given data, as well as the geometry of $\partial\Omega$. By our assumption at the initial geometry, we find that $\Omega_{\eta(T_*)}$ has no self intersection.
We define $\eta^*=(\eta(T_*))_\kappa$, where $\kappa$ is a convolution operator in space.  We define $\tilde{\Omega}=\Omega_{\eta^*}\in C^4$.
If $\kappa$ is conveniently small, than also $\tilde{\Omega}$ has no self intersection either. Especially, there exists some $\tilde{L}>0$, such that on
\[
\tilde{S}_{\tilde{L}}:=\set{x\in \setR^3:\text{ dist}(x,\partial\Omega_{\eta^*})\leq \tilde L},
\]
the function (see the beginning of Subsection~\ref{ssec:geom})
 \begin{align*}
\tilde{\Lambda}:\partial\tilde\Omega\times(-\tilde{L},\tilde{L})\rightarrow \tilde S_{\tilde{L}},\quad 
\tilde{\Lambda}(\tilde q,\tilde s)=\tilde q+\tilde s\nu(q),
\end{align*}
is well defined. Here we have $(q,s)=\Lambda^{-1}(\tilde{q})$ and $\nu$ is the outer normal of the initial geometry $\partial\Omega$.
This implies that for $\tilde{\zeta}: \partial\tilde{\Omega}\to [-\tilde{L},\tilde{L}]$, we may associate
\begin{align*}
\Omega_{\tilde\zeta}:=\tilde{\Omega}\setminus \tilde{S}_{\tilde{L}}\cup\{x\in \tilde{S}_{\tilde{L}}:\,\,\tilde{s}(x)<\tilde{\zeta}(\tilde{q}(x))\}.
\end{align*}
By the definition of $\tilde L$, there exists a diffeomorphism
\[
\bfPsi_{\tilde\zeta}:\tilde{\Omega}\to \Omega_{\tilde{\zeta}},
\]
cf. Lemma \ref{lem:diffeo}.
In particular, we may define the function 
$$\zeta:\partial{\Omega}\to \setR,\quad  q\mapsto\tilde\zeta(q+\nu\eta^*(q))+\eta^*(q).$$ It satisfies
\begin{align*}
 \zeta(q)\in [\eta^*(q)-\tilde{L},\eta^*(q)+\tilde{L}]\text{ for all }q\in \partial\Omega
\end{align*}
where $(q,s)=\Lambda^{-1}(\tilde{q})$.
This implies that the mapping
\[
\bfPsi_{\zeta}:=\bfPsi_{\tilde{\zeta}}\circ\bfPsi_{\eta^*}:\Omega\to \Omega_\zeta=\Omega_{\tilde{\zeta}},
\]
is a well-defined diffemomorphism.
The transformation can also be inverted. 
For any $\eta: \partial \Omega \to \setR$ with
$$
\eta(q)\in [\eta^*(q)-\tilde{L},\eta^*(q)+\tilde{L}]\text{ for all }q\in \partial\Omega
$$
we may define
\begin{align*}
\tilde\eta:\partial\tilde{\Omega}\to [-\tilde{L},\tilde{L}],\quad
\tilde q\mapsto\eta(\tilde{q}-\nu(q)\eta^*(q))-\eta^*(q).
\end{align*}
By this construction we have changed the coordinates
 set $\Omega_\eta$ since
\[
\bfPsi_{\eta}=\bfPsi_{\tilde\eta}\circ \bfPsi_{\eta^*}:\Omega\to \Omega_{\eta}=\Omega_{\tilde{\eta}}.
\]
This transformation can be used to extend the solution. To be precise, we set
\begin{itemize}
\item $\tilde{\eta}_0=\eta(T_*)$,
\item $\tilde{\eta}_1=\partial_t\eta(T^*)$,
\item $\tilde{\varrho}_0=\varrho(T^*)$,
\item $\tilde{\bfq}_0=\varrho(T^*)\bfu(T^*)$.
\end{itemize}
By the construction above, we can associate to any $\tilde{\zeta}\in C\big([T^*,T^{**}]\times\partial\tilde{\Omega}, [-\frac{\tilde{L}}2,\frac{\tilde{L}}2]\big)$ a function $\zeta\in C([T^*,T^{**}]\times\partial\Omega)$ such that 
\[
\eta(q)\in \Big[\eta^*(q)-\frac{\tilde{L}}2,\eta^*(q)+\frac{\tilde{L}}2\Big]\text{ for all }q\in \partial\Omega.
\]
Moreover, the mappings $\bfPsi_{\zeta}:\Omega\to \Omega_{\zeta}$ and $\bfPsi_{\regkap\zeta}:\Omega\to \Omega_{\regkap\zeta}$ are both well defined, provided we choose $\kappa$ small enough. Now, 
first Theorem~\ref{thm:decu} provides a solution $(\eta,\bfu)$ to any given pair $$(\zeta,\bfv)\in C([T^*,T^{**}]\times\partial\Omega, \setR)\times L^2([T^*,T^{**}]\times \setR^3).$$ Second, we wish to get a fixpoint by applying Theorem~\ref{thm:regu}. The only modification is that the fixpoint mapping has to be adjusted slightly. Indeed, the fixed point has to be found in the set
\begin{align*}
D&:=\Big\{(\tilde\zeta,\bfv)\in C([T^*,T^{**}]\times \partial\tilde\Omega)\times L^2([T^*,T^{**}]\times\R^3)):\\&\qquad \qquad\,\,\tilde\zeta(0)=\tilde\eta(T^*),\,\,\|\tilde{\zeta}\|_{L^\infty}\leq \frac{\tilde L}2,\,\,\|\bfv\|_{L^2([T^*,T^{**}]\times\R^3)}\leq K\Big\}.
\end{align*}
Here $K$ has to be adjusted to $T^{**}$ in accordance with the proof of Theorem~\ref{thm:regu}. Finally, we set $F:D\rightarrow \mathfrak P(D)$,
\[
F:(\tilde\bfv, \tilde\zeta)\mapsto \Big\{(\tilde\bfu,\tilde\eta):\,(\bfu,\eta)\text{ solves }\eqref{eq:regudc}\text{ with $(\bfv, \zeta)$ and satisfies the energy bounds}\Big\}.
\]
where $\tilde\eta$ is defined via the solution $\eta$ by $\tilde\eta=\eta(\tilde{q}-\nu(q)\eta^*(q))-\eta^*(q)$ as introduced above. The rest of the argument of Theorem~\ref{thm:regu} does not change, since the $L^\infty$ bounds of $\eta,\zeta$ (which are critical for the fixed point argument) do not change by coordinate transformations. Once the fixed point is established, we may pass to the limit with $\kappa,\epsilon$ and $\delta$ as before. Observe, that in Subsection~\ref{sec:renep} one has to use the extension operator from Lemma \ref{lem:extension} with respect to the coordinate transformation $\bfPsi_{\tilde{\eta}}$ as it satisfies $\|\tilde\eta\|_\infty<\frac{L}{2}$ (here we use the fact that $\Omega_{\tilde{\eta}}=\Omega_\eta$ by our construction). \\
We remark that the solution $\eta$ and $\Omega_{\regkap\eta}$ are defined via the same reference coordinates $\partial\Omega$. This means it truly extends the solution and we can extend the interval of existence.\\
Finally, the above procedure can be iterated until a selfintersection is approached. This finishes the proof of Theorem \ref{thm:final}.

\centerline{\bf Acknowledgement}
\noindent{
S. Schwarzacher
gratefully acknowledges the support of the project LL1202 fnanced by the Ministry of Education,
Youth and Sports and the the program PRVOUK~P47, financed by Charles University in Prague.}\\\

\centerline{\bf Declaration}
\noindent{
The authors declare that there are no conflicts of interest.}


\end{document}